\documentclass[a4paper,11pt]{article}
\usepackage{fullpage,hyperref}
\usepackage[british]{babel}
\usepackage{wrapfig}
\usepackage[comma, square, numbers]{natbib}
\usepackage{yfonts}
\usepackage[utf8]{inputenc}
\usepackage{amssymb}
\usepackage{amsthm}
\usepackage{graphics}
\usepackage{amsmath}
\usepackage{dirtytalk}
\usepackage{amstext}
\usepackage{engrec}
\usepackage{floatflt}
\usepackage{rotating}
\usepackage[safe,extra]{tipa}
\usepackage[font={footnotesize}]{caption}
\usepackage{framed}
\usepackage{listings}
\usepackage{subcaption}
\usepackage[titletoc,title]{appendix}

\newcommand{\mc}{\mathcal}
\newcommand{\mb}{\mathbb}

\newcommand{\R}{\mb R}
\newcommand{\C}{\mb C}
\newcommand{\N}{\mb N}
\newcommand{\Z}{\mb Z}

\newcommand{\T}{\mb T}

\newcommand{\M}{\mc M}

\newcommand{\eea}{\end{align}}

\renewcommand{\epsilon}{\varepsilon}
\renewcommand{\bar}{\overline}
\renewcommand{\tilde}{\widetilde}

\newcommand{\bo}{\boldsymbol}
\renewcommand{\phi}{\varphi}

\DeclareMathOperator{\Tr}{Tr}
\DeclareMathOperator{\Col}{Col}

\DeclareMathOperator{\Jac}{Jac}

\newcommand{\degree}{\delta}
\newcommand{\redu}{g}
\newcommand{\auxex}{G}

\newcommand{\expansion}{\bar \sigma}
\newcommand{\neig}{{\epsilon_\Lambda}}
\newcommand{\leb}{m}
\renewcommand\upsilon{\theta}
\newcommand\tea{t} 
\newcommand{\B}{B}

\usepackage[normalem]{ulem}

\newtheorem{theorem}{Theorem}[section]
\newtheorem{mtheorem}{Theorem}

\newtheorem*{Hoefd}{ Hoeffding inequality}
\newtheorem{corollary}{Corollary}[section]
\newtheorem{lemma}{Lemma}[section]
\newtheorem{proposition}{Proposition}[section]
\theoremstyle{definition}
\newtheorem{definition}{Definition}[section]

\theoremstyle{remark}
\newtheorem{remark}{Remark}[section]
\newtheorem{example}{Example}[section]

\newtheoremstyle{algorithm}
{4pt}
{4pt}
{}
{}
{}
{:}
{\newline}
{}

\usepackage{xcolor}

\newtheorem{algorithm}{Algorithm}

\newcommand{\balgorithm}{\begin{algorithm}\begin{framed}\ }
\newcommand{\ealgorithm}{\end{framed}\end{algorithm}}

\newcommand{\pippo}{\begin{code} \  \begin{lstlisting}[frame=single]}
\newcommand{\pappo}{\end{lstlisting} \end{code}}

\newcommand{\bd}{\begin{definition}}
\newcommand{\ed}{\end{definition}}

\newcommand{\bt}{\begin{theorem}}
\newcommand{\et}{\end{theorem}}
\newcommand{\bp}{\begin{proposition}}
\newcommand{\ep}{\end{proposition}}
\newcommand{\bc}{\begin{corollary}}
\newcommand{\ec}{\end{corollary}} 
\newcommand{\bl}{\begin{lemma}}
\newcommand{\el}{\end{lemma}}
\newcommand{\br}{\begin{remark}}
\newcommand{\er}{\end{remark}}

\DeclareMathOperator{\Lip}{Lip}
\DeclareMathOperator{\Id}{Id}

\renewcommand{\thefootnote}{\fnsymbol{footnote}}

\title{Heterogeneously Coupled Maps:  \\hub dynamics and emergence across connectivity layers. } 
\author{Tiago Pereira, Sebastian van Strien and Matteo Tanzi}

\begin{document}
\maketitle
\begin{abstract}
The aim of this paper is to rigorously study dynamics of Heterogeneously  Coupled Maps (HCM). Such  systems are determined by a  network with heterogeneous degrees.  Some nodes, called hubs,  are very well connected while most nodes interact with few others. The local dynamics on each node is chaotic,  coupled with other nodes according to the network structure.  Such high-dimensional systems are  hard to understand in full, nevertheless we are able to describe the system over exponentially large time scales.  In particular, we   show that the dynamics of hub nodes can be very well approximated by a low-dimensional system.  This allows us to establish the emergence of macroscopic behaviour such as coherence of dynamics among hubs of  the same connectivity layer (i.e. with the same number of connections), and chaotic behaviour of the poorly connected nodes. The HCM we study provide a  paradigm to explain why and how the dynamics of the network can change across layers.
\end{abstract}

\let\thefootnote\relax
\footnote{\emph{Emails:} tiagophysics@gmail.com, s.van-strien@imperial.ac.uk, matteotanzi@hotmail.it}


\tableofcontents

\section{Introduction}

Natural and artificial complex systems are often modelled as distinct units interacting on a network. Typically such networks have a heterogeneous structure characterised by different scales of connectivity \cite{BA}. 
Some nodes called {\em hubs} are highly connected while the remaining nodes have only a small number of connections (see Figure \ref{Fig1a} for an illustration). Hubs provide a short pathway between nodes making the network well connected and resilient and play a crucial role in the description and understanding of complex networks. 

In the brain, for example, hub neurons are able to synchronize while other neurons remain out of synchrony.  This particular behaviour shapes the network dynamics towards a healthy state \cite{bonifazi2009gabaergic}. Surprisingly,  disrupting  synchronization between hubs can lead to malfunction of the brain. The fundamental dynamical role of hub nodes is not restricted to neuroscience, but is found in the study of epidemics \cite{Epidemic}, power grids \cite{Motter}, and many other fields. 

Large-scale simulations of networks suggest that the mere presence of hubs hinders global collective properties. That is, when the heterogeneity in the degrees of the network is strong, complete synchronization is observed to be unstable \cite{Nishikawa}. However, in certain situations hubs can undergo a transition to collective dynamics \cite{Paths,Hubs,CAS}.  Despite the large amount of  recent work, a mathematical understanding of dynamical properties of such networks remains elusive. 

In this paper, we introduce the concept of Heterogeneously Coupled Maps (referred to as HCM in short), where the heterogeneity comes from the network structure modelling the interaction. HCM describes the class of problems discussed above incorporating the non-linear and extremely high dimensional behaviour  observed in these networks. 
High dimensional systems are notoriously difficult to understand. 
HCM is no exception. Here, our approach is to describe the dynamics at the expense of an arbitrary small, but fixed fluctuation, over exponentially large time scales. In summary, we obtain

(i) \emph{Dimensional reduction for hubs for finite time.} Fixing a given accuracy, we can describe the dynamics of the hubs by a low dimensional model for a finite time $T$. The true dynamics of a hub and its low dimensional approximation are the same up to the given accuracy. The time $T$ for which the reduction is valid is exponentially large in the network size. For example, in the case of a star network (see Section \ref{Sec:StarNetExamp}), we can describe the hubs with $1$\% accuracy in networks with $10^6$ nodes for a time up to roughly $T = e^{30}$ for a set of initial conditions of measure roughly $1-e^{-10}$. This is arguably the only behaviour one will ever see in practice.

(ii) \emph{Emergent dynamics changes across connectivity levels}. The dynamics of hubs can drastically change depending on the degree and synchronization (or more generally phase lockng) naturally emerges between hub nodes. This synchronization is not due to a direct mutual interaction between hubs (as in the usual \say{Huygens} synchronization)  but results from the common environment that the hub nodes experience. 

Before presenting the general setting and precise statements in Section \ref{Sec:SettRes}, we informally discuss these results and illustrate the rich dynamics that emerges in HCM due to heterogeneity.

\subsection{Emergent Dynamics on  Heterogeneously Coupled Maps (HCM).} 

Figure \ref{Fig1a} is a schematic representation of a heterogeneous network with three different types of nodes: massively connected hubs (on top), moderately connected hubs having half as many connections of the previous ones (in the middle), and low degree nodes (at the bottom). Each one of this three types constitutes a connectivity layer, meaning a subset of the nodes in the network having approximately the same degree.  When uncoupled, each node is identical and supports chaotic dynamics. Adding the coupling, different behaviour can emerge for the three types of nodes. In fact, we will show examples where the dynamics of the hub at the top approximately follows a periodic motion, the hub in the middle stays near a fixed point, and the nodes at the bottom remain chaotic. Moreover, this behaviour persists for exponentially large time in the size of the network, and  it is robust under small perturbations.
\begin{figure}[htbp]
\centering
\includegraphics[width=3in]{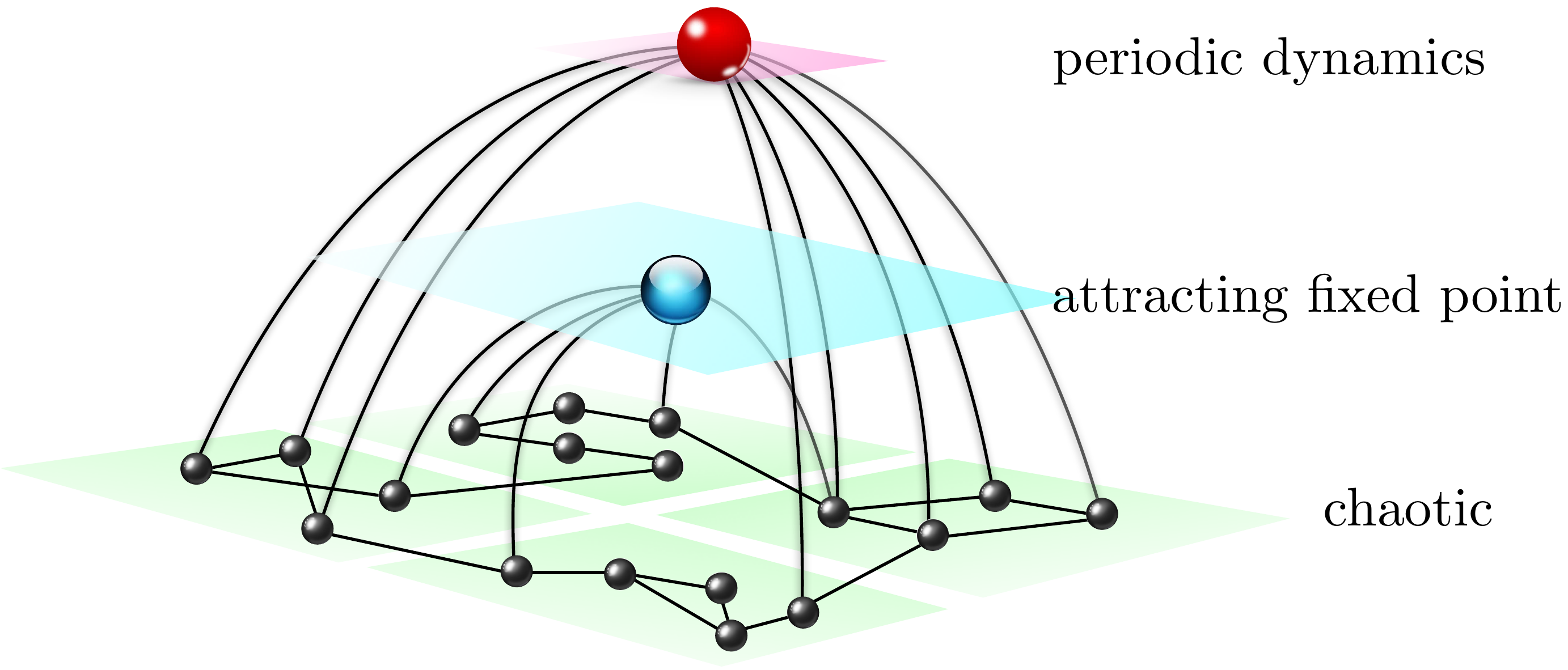}
\caption{The dynamics across connectivity layers change
depending on the connectivity of the hubs.
We will exhibit an example where the hubs with the highest number of connections  (in red, at the top) have periodic
dynamics. In the second connectivity layer, where hubs have half of the
number of connection (in blue, in the middle), the dynamics sits around a fixed point. In
the bottom layer of poorly connected nodes the dynamics is chaotic. (Only one hub has been drawn on the top two layers for clarity of the picture). }
\label{Fig1a}
\end{figure}\\

\noindent
 
\medskip

\noindent
{\bf Synchronization because of common environment}. Our theory uncovers the mechanism responsible for high correlations among the hubs states, which is observed in experimental and numerical observations. The mechanism turns out to be different from synchronization (or phase lockng) due to mutual interaction, i.e. different from \say{Huygens} synchronization.  In HCM, hubs display highly correlated behaviour even in the
absence of direct connections between themselves.   The poorly connected
layer consisting of a huge number of weakly connected nodes plays the role of a kind of 
{\em \say{heat bath}} providing a common forcing to the hubs which is responsible for the emergence of coherence.
\medskip

\noindent

\subsection{Hub Synchronization  and Informal Statement of Theorem~\ref{Thm:Main}}

{\bf The Model.}  {\it A network of coupled dynamical systems} is the
datum $(G,f,h,\alpha)$, where
$G$ is a labelled graph of the set of nodes $\mc N=\{1,...,N\}$, $f : \mb T \rightarrow \mb T$ is
the local dynamics at  each node of the graph, $h\colon \mb T\times\mb T \to  \R$  
is a coupling function that describes pairwise interaction between nodes, and $\alpha\in
\R$ is  the coupling strength.  We take $f$ to
be a Bernoulli map,  $z \mapsto \sigma z \, \mbox{mod 1}$, for some integer
$\sigma>1$. This is in agreement with the observation that the 
local dynamics is chaotic in many applications \cite{Izhikevich2007dynamical,weiss1988,shil2001}.
The graph $G$ can be represented by its adjacency matrix $A = (A_{in})$ which determines the
connections among nodes of the graph.  If $A_{in} = 1$, then there is an directed edge of the graph going from $n$ and pointing at 
$i$. $A_{in}=0$ otherwise. 
The degree  $d_i:=\sum_{n=1}^N A_{in}$ is the number of incoming edges at $i$. 
For sake of simplicity, in this introductory section we consider undirected graphs ($A$ is symmetric), unless otherwise specified, but our results hold in greater generality (see Section~\ref{Sec:SettRes}).

The dynamics on the network is described by
\begin{equation}
{z}_i(t+1) = {f}({ z}_i(t)) + \frac{\alpha}{\Delta} \sum_{n=1}^N
A_{in} h(z_i(t), z_n(t))\mod 1,\quad\mbox{ for } i=1,\dots,N.
\label{md1}
\end{equation}
In the above equations,  $\Delta$ is a structural parameter of the network equal to the maximum degree. Rescaling of the coupling strength in (\ref{md1}) dividing by $\Delta$ allows to scope the parameter regime for which interactions contribute with an order one term to the evolution of the hubs.

For the type of graphs we will be considering we have that the  degree $d_i$ of the nodes
$1,\dots,L$ are much smaller than the incoming degrees of nodes $L+1,\dots,N$. A prototypical sequence of heterogeneous degrees is  
\begin{equation}\label{eq:layeredER}
{\bf d}(N) = ( \underbrace{d,\dots, d}_{L}, \underbrace{\kappa_m\Delta,\dots,\kappa_m\Delta}_{M_m}, \dots,
 \underbrace{\kappa_{2}
\Delta ,\dots, \kappa_{2} \Delta }_{M_{2}}, \underbrace{
\Delta,\dots, \Delta }_{M_1} ). 
\end{equation} 
with $\kappa_m<\dots < \kappa_2 <1$ fixed and $d/\Delta$ small when $N$ is large, 
then we will refer to blocks of nodes corresponding to  $(\kappa_{i}
\Delta ,\dots, \kappa_{i} \Delta)$ as the {\em $i$-th connectivity layer} of the network, and to a graph $G$ having sequence of degrees prescribed by Eq. \eqref{eq:layeredER} as a \emph{layered heterogeneous graph}.
(We will make all this more precise below.)  

It is a consequence of stochastic stability of uniformly expanding maps, that for very small coupling strengths, the network dynamics will remain
chaotic. That is, there is an $\alpha_0>0$ such that for all  $0\le
\alpha < \alpha_0$ and any large $N$, the system will preserve
an ergodic absolutely continuous invariant measure \cite{Keller2}. When $\alpha$ increases,
one reaches a regime where the less connected nodes still feel a small
contribution coming from interactions, while the hub nodes receive an
order one perturbation. In this situation, uniform hyperbolicity
and the absolutely continuous invariant measure do not persist in
general.

\medskip
\noindent 
{\bf The Low-Dimensional Approximation for the Hubs.} Given a hub $i_j\in\mc N$ in the $i$-th connectivity layer, our result gives a
one-dimensional approximation of its dynamics in terms of $f$, $h$, $\alpha$ and the connectivity $\kappa_i$ of the layer. The idea is the following. 
Let $z_1,\dots,z_N\in \mb T$ be the state of each node, and assume that this collection of $N$ points are spatially distributed
in $\mb T$ approximately according to the invariant measure $m$ of the local map $f$ (in this case the
Lebesgue measure on $\mb T$).   Then the coupling term in (\ref{md1}) is a  \emph{mean field} (Monte-Carlo) approximation 
of the corresponding integral: 
\begin{equation} \frac{\alpha}{\Delta} \sum_{n=1}^N
A_{i_jn} h(z_{i_j}, z_n) \approx \alpha \kappa_i\int h(z_{i_j},y) dm(y)
\label{Eq:MeanField}
\end{equation}
where $d_{i_j}$ is the incoming degree at $i_j$ and
 $ \kappa_{i}:=d_{i_j} / \Delta$ is its normalized incoming degree. 
 The parameter $\kappa_i$ determines the effective coupling
strength. 
Hence, the right hand side of expression (\ref{md1}) at the node $i_j$ is approximately equal to the {\em reduced} map
\begin{eqnarray}\label{Eq:RedEqInt}
\redu_{i_j}(z_{i_j}):=f(z_{i_j})+\alpha \kappa_i\int h(z_{i_j},y) dm(y), 
\end{eqnarray}
Equations (\ref{Eq:MeanField}) and (\ref{Eq:RedEqInt}) clearly show the \say{heat bath} effect that the common environment 
has on the highly connected nodes.

\medskip 
\noindent
{\it Ergodicity ensures the persistence of the heat bath role of the low degree nodes.}
It turns out that the joint behaviour at poorly connected nodes is essentially ergodic. This will imply that at each moment of time  the cumulative average effect
on  hub nodes is predictable and far from negligible.     In this way, the low degree nodes play the role of a heat bath providing a sustained forcing to the hubs.

\medskip

Theorem~\ref{Thm:Main} below makes this idea rigorous  for a suitable class of networks. We state
the result precisely and in full generality in Section~\ref{Sec:SettRes}. 
For the moment assume that  the number of hubs is small, does not depend on
the total number $N$ of nodes,  and that the degree of
the poorly connected nodes is relatively small, namely only a logarithmic function of $N$. 
 For these networks our theorem implies  the following \\

\begin{minipage}[c]{14cm}
{\it
\emph{\textbf{Theorem A (Informal Statement in Special Case).}}
Consider the dynamics \eqref{md1} on a layered heterogeneous graph. If the degrees of the hubs are sufficiently large, i.e. $\Delta = O(N^{1/2+\epsilon})$,
and the reduced dynamics $g_j$ are hyperbolic, then for any hub $j$ 
\[
z_{j}(t+1)=\redu_j(z_j(t))+\xi_j(t),
\]
where the size of fluctuations $\xi_j(t)$ is below any fixed threshold
for  $0\leq t\leq T$, with
$T$ exponentially large in $\Delta$, and any initial condition outside a subset of measure exponentially small in $\Delta$.
} \\
\end{minipage}\\

\noindent
{\it Hub Synchronization Mechanism.} 
When $\xi_j(t)$ is small
and $g_j$ has an attracting periodic orbit, then  $z_j(t)$ 
will be close to this attracting orbit after a short time and it will remain close to the orbit for an exponentially large time $T$.  
As a consequence, if two hubs have approximately the same degree $d_j$, even if they share no common neighbour,  they feel the same mean effect from the \say{heat bath} and so they appear to be attracted to the same periodic orbit (modulo small fluctuations) exhibiting highly coherent behaviour.

The dimensional reduction provided in Theorem A is robust 
 persisting under small perturbation of the dynamics $f$, of the coupling function $h$ 
 and  under addition of small independent noise. Our results show that the fluctuations $\xi(t)$, as functions of the initial condition, are small in the $C^0$ norm on most of the phase space,
but notice that they can be very large with respect to the $C^1$ norm. Moreover, they are correlated, and
with probability one, $\xi(t)$ will be large for some $t>T$.

\medskip

\noindent
{\bf Idea of the Proof.}  The proof of this theorem consists of two steps. 
Redefining {\it ad hoc} the system in the region of
phase space where fluctuations are above a chosen small threshold,
we obtain a system which exhibits good hyperbolic properties that we state
in terms of invariant cone-fields of expanding and contracting
directions. We then show that the set of initial conditions
for which the fluctuations remain below this small threshold up to time
$T$ is large, where $T$ is estimated as in the above informal statement of the theorem. 

\subsection{Dynamics Across Connectivity Scales: Predictions and Experiments}\label{subsec:predictions+experiments}

In the setting above, consider $f(z) =2z $ mod\,1 and the following 
simple coupling function: 
\begin{equation}\label{h}
h(z_i,z_n) = - \sin 2\pi z_i + \sin 2\pi z_n.
\end{equation}
Since $\int_{0}^1 \sin (2\pi y)\, dy=0$, the reduced equation, see Eq. (\ref{Eq:RedEqInt}), becomes 
\begin{eqnarray}\label{g}
\redu_j(z_j) = T_{\alpha \kappa_j}(z_j)  \mbox{ where }T_\beta(z)= 2 z   - \beta \sin (2\pi z)   \mod 1 .
\end{eqnarray}
A bifurcation analysis shows that for $\beta\in
I_{E}:=[0,1/2\pi)$ the map is globally expanding, while for $\beta
\in I_{F}:=(1/2\pi,3/2\pi)$ it has an attracting fixed point at $y=0$.
Moreover, for  $\beta \in I_{p}:=(3/2\pi,4/2\pi]$ it has an
attracting periodic orbit of period two. In fact, it follows from a recent result in \cite{MR3336841} that the set of parameters $\beta$ for which $T_\beta$ is hyperbolic, as specified by Definition~\ref{Def:AxiomA} below, is open and dense. (See Proposition~\ref{Prop:AppTbeta1} and \ref{Prop:AppTbeta2} in the Appendix for a rigorous treatment). Figure~\ref{Fig2} shows the graphs and bifurcation diagram of $T_\beta$ varying $\beta$.
\begin{figure}[htbp]
\begin{subfigure}{0.5\textwidth} 
\centering
\includegraphics[scale=0.35]{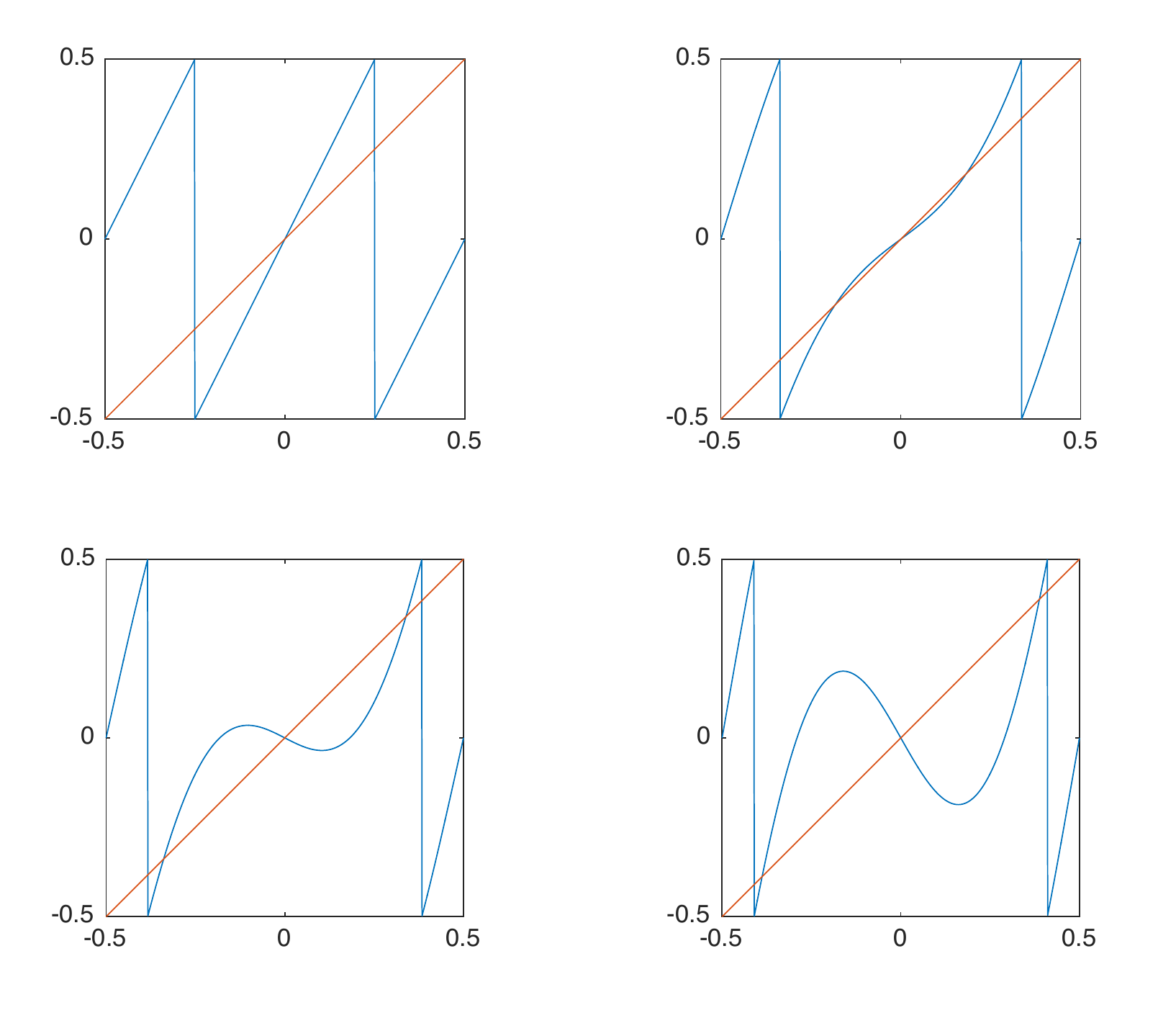}
\end{subfigure}
\begin{subfigure}{0.5\textwidth} 
\centering
\includegraphics[width=2in]{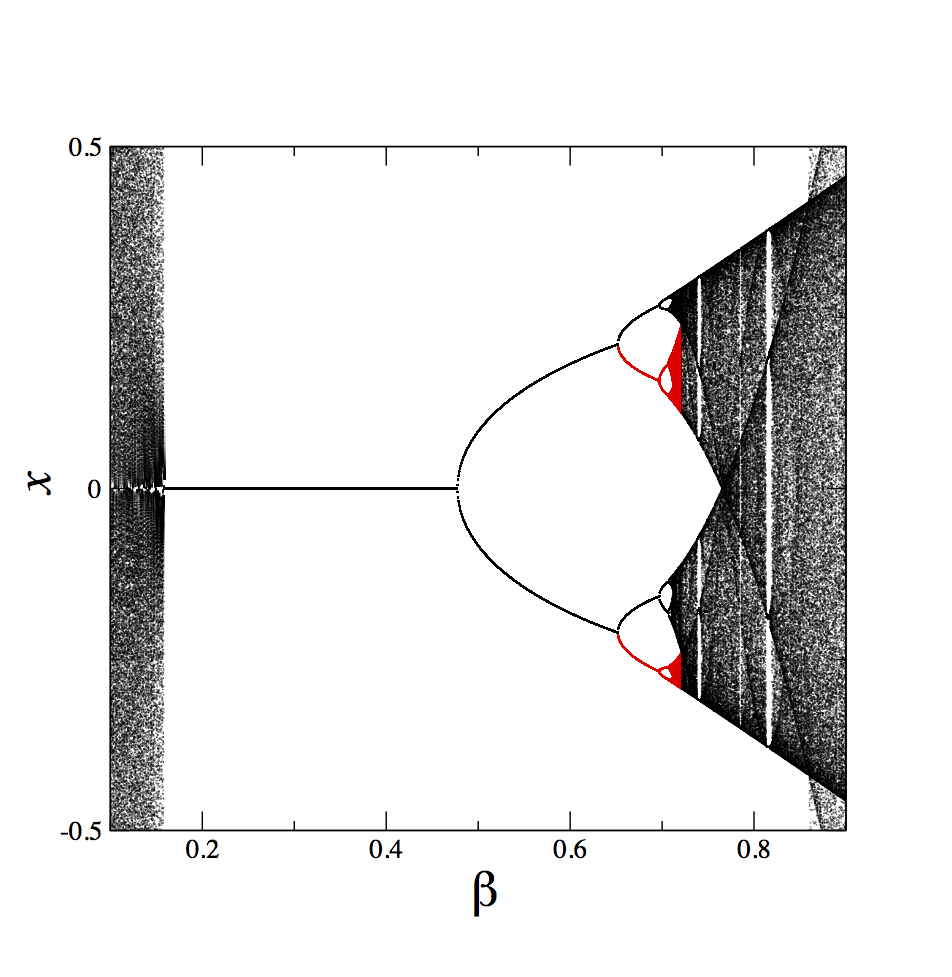}
\end{subfigure}
\vspace{-0,2cm}
\caption{On the left the graphs of $T_\beta$ for $\beta=0,0.2,0.4,0.6$. On the right the bifurcation diagram for the reduced dynamics of hubs. We considered the identification $\mathbb{T} =
[-1/2,1/2]/\!\! \sim$. We obtained the diagram numerically. To build the bifurcation diagram we reported a segment of a typical orbit of length $10^3$, for a collection of values
of the parameter $\beta$.}
\label{Fig2}
\end{figure}

\subsubsection{Predicted Impact of the Network Structure}\label{Sec:PredImpNet}
 To illustrate the impact of
the structure, we fix the coupling strength $\alpha = 0.6$ and
consider a heterogeneous network
with four levels of connectivity including three types of hubs and
poorly connected nodes.
The first highly connected hubs have $\kappa_1=1$. In the second layer,
hubs have half of the number of connections of the first layer
$\kappa_2=1/2$. And finally, in the last layer, hubs have one fourth of
the connections of the main hub $\kappa_3=1/4$.
The parameter  $\beta_j=\alpha\kappa_j$ determines the effective coupling, and so for
the three levels $j=1,2,3$ we predict different types of dynamics looking at the bifurcation diagram. The predictions are summarised in 
Table \ref{Dyn}.
\begin{table}[htbp]
\centering
\begin{tabular}{lcl} 
Connectivity Layer & Effective Coupling $\beta$ & Dynamics \\
\hline
\hline
hubs with $\kappa_1  =1$ & 0.6& Periodic  \\
hubs with $\kappa_2  =1/2$ & 0.3& Fixed Point \\
hubs with  $\kappa_3 = 1/4$ & 0.15 & Uniformly Expanding \\
\hline
  \end{tabular}
  \caption{Dynamics across connectivity scales }
  \label{Dyn}
\end{table}

\subsubsection{Impact of the Network structure in Numerical Simulations of Large-Scale Layered Random Networks} We have considered the above situation in numerical simulations where we took a layered random network, described in equation (\ref{eq:layeredER}) above, with $N=10^5$, $\Delta = 500$,  $w = 20$, $m=2$, $M_1=M_2=20$, $\kappa=1$ and  $\kappa_2=1/2$. 
The layer with highest connectivity is made of  $20$ hubs connected to $500$ nodes, and the second layer is made of $20$ hubs connected
to $250$ nodes.  The local dynamics is again given by  $f(z)= 2z
$ $\mod1$, the coupling as in Eq. (\ref{h}). We fixed the coupling strength at $\alpha=0.6$  as in Section \ref{Sec:PredImpNet} so that Table \ref{Dyn} summarises the theoretically predicted dynamical behaviour for the two layers. 
We choose initial conditions for each of the $N$ nodes independently and according to
the Lebesgue measure. Then we evolve this $10^5$ dimensional system for $10^6$ iterations. Discarding 
the $10^6$ initial iterations as transients, we plotted the next
$300$ iterations. The result is shown in Figure \ref{Simulation}. 
In fact, we found essentially the same picture when we  only plotted the first $300$ iterations, with the difference that 
the first $10$ iterates or so are not yet in the immediate basin of the periodic attractors.  
The simulated dynamics in Figure \ref{Simulation} is in excellent agreement with the predictions of Table \ref{Dyn}. 

\begin{figure}[htbp]
\centering
\includegraphics[width=5in]{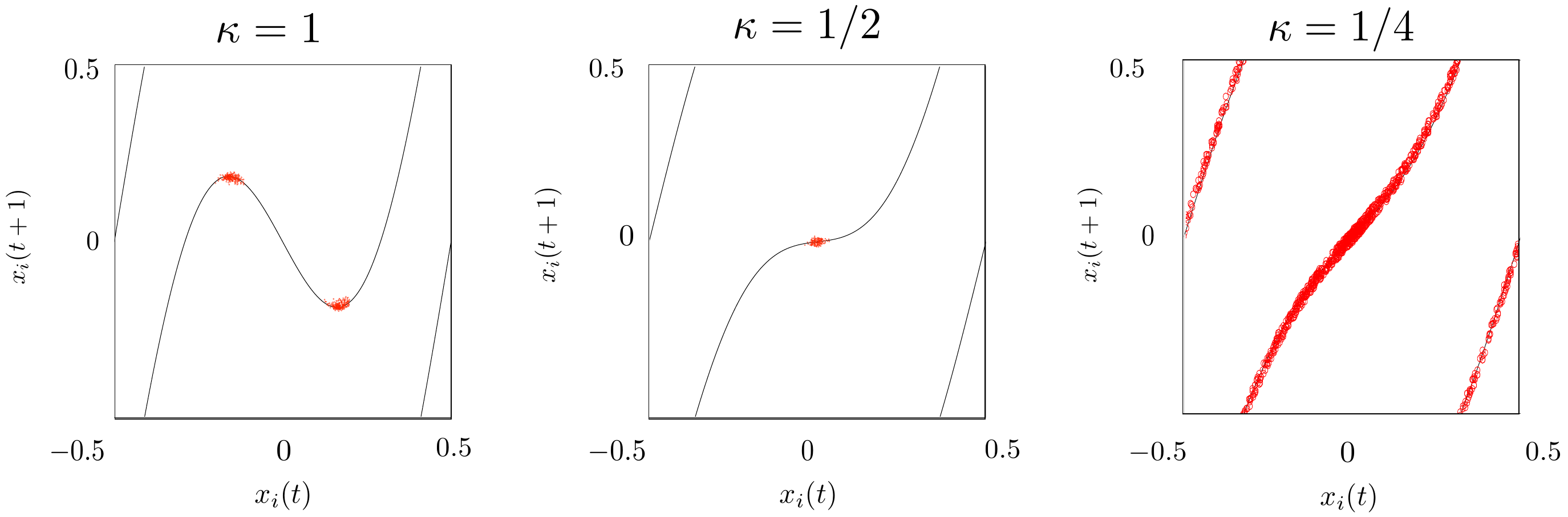}
\vspace{-0.2cm}
\caption{Simulation results of the dynamics of a layered graph with
two layers of hubs.
We plot the return maps $z_i(t) \times z_i (t+1)$. The solid line is
the low dimensional approximation of the hub dynamics given by Eq. (\ref{g}). The red circles are
points taken from the hub time-series. In the first layers of hubs
($\kappa=1$) we observe a dynamics very close to the periodic orbit
predicted by $g_1$, in the second layer ($\kappa=1/2$) the dynamics of
the hubs stay near a fixed point, and in the third layer ($\kappa=1/4$) the dynamics is still uniformly expanding. }
\label{Simulation}
\end{figure}

\subsection{Impact of Network Structure on Dynamics: Theorems~\ref{MTheo:B} and \ref{MTheo:C} } 

\noindent
 The importance of network structure in shaping the dynamics has been highlighted by many studies \cite{Gol-Stewart2006, Field-etal-2011, nijholt2016graph} where network topology and its symmetries shape bifurcations patterns and synchronization spaces. Here we continue with this philosophy and show the dynamical feature that are to be expected in HCM. In particular 
 one has that fixing the local dynamics and the
coupling, the network structure dictates the resulting dynamics. In fact we show that \\

\begin{minipage}[c]{14cm}
{\it
there is an open set of coupling functions such that
homogeneous networks globally synchronize but heterogeneous networks do not.
However, in heterogeneous networks, hubs can undergo a
transition to coherent behaviour.
}
\end{minipage}
\vspace{0.4cm}

\noindent
In Subsection~\ref{Sec:SettRes} the content of this claim is given a rigorous formulation in Theorems~\ref{MTheo:B} and \ref{MTheo:C}.  

\subsubsection{Informal Statement of Theorem~\ref{MTheo:B} on Coherence of Hub Dynamics}
Consider a graph $G$ with sequence of degrees given by Eq. (\ref{eq:layeredER}) with 
 $M:=\sum_{k=1}^m M_k $, each $M_i$ being the number of nodes in the $i-$th connectivity layer. Assume  
\begin{equation}
\Delta = \mc O(N^{1/2 + \varepsilon})\mbox{, }M = \mc O(\log N) \mbox{~and~} d =\mc  O(\log N)  \label{eq:layeredER2} 
\end{equation}
which implies that $L\approx N$ when $N$ is large. Suppose that $f(x)=2x\mod 1$ and that $h(z_i,z_n)$ is as in Eq. \eqref{h}. 

\

\def\dist{\mbox{dist}}
\begin{minipage}[c]{14cm}
{\it
\emph{\textbf{Theorem B (Informal Statement in Special Case).}} For every connectivity layer $i$ and hub node ${i_j}$ in this layer, there exists an interval  $I \subset \mb R$ of coupling strengths so that for any 
$\alpha \in I$, the reduced dynamics $T_{\alpha\kappa_i}$ (Eq. \eqref{g}) has at most two periodic attractors $\{\bar z(t)\}_{t=1}^p$ and $\{-\bar z(t)\}_{t=1}^p$ and there is $s\in\{\pm 1\}$ and $t_0\in[p-1]$
 \[
\dist(z_{i_j}(t+t_0), s \bar z(t\mbox{ mod } p))\le \xi 
 \] 
for $1/\xi\le t\le T$, with $T$ exponentially large in $\Delta$, and for any initial condition outside a set of small measure.
} \\
\end{minipage}

\noindent
Note that in order to have $1/\xi \ll T$ one needs $\Delta$ to be large. Theorem~\ref{MTheo:B} proves that one can generically tune the coupling strength or the hub connectivity so that the hub dynamics follow, after an initial transient, a periodic orbit. 


\subsubsection{Informal Statement of Theorem C Comparing Dynamics on Homogeneous and Heterogeneous Networks}

{\bf Erd\"os-R\'enyi model for homogeneous graphs} In contrast to layered graphs which are prototypes of heterogeneous networks, the classical Erd\"os-R\'enyi model is a prototype of a homogenous  random graph. By homogeneous, we mean that the expected degrees of the nodes are the same. This model defines an undirected random graph where
each link in the graph is a Bernoulli random variable with the same success probability $p$ (see Definition~\ref{Def:ErdRen} for more details). We choose $p > \log N / N$  so that in the limit that $N\rightarrow\infty$ almost every random graph is connected (see \cite{Bollobas}). 
\
\vspace{0.3cm}

\noindent
{\bf Diffusive Coupling Functions} The coupling functions satisfying
\[
h (z_i,z_j) = - h(z_j,z_i) \mbox{~and~} h(z,z)=0.
\]
 are called {\em diffusive} The function $h$ is sometimes required to satisfy $\partial_1h(z,z)>0$  to ensure that the coupling has an \say{attractive} nature. Even if this is not necessary to our computations, the examples in the following and in the appendix satisfy this assumption. 
For each network $G$, we consider the corresponding system of coupled maps defined by (\ref{md1}). 
In this case the
subspace
\begin{equation}\label{Eq:SyncManif}
\mathcal{S} := \{ (z_1,...,z_N) \in \mathbb{T^N} \, : \,
z_1 = z_2 =\cdots =z_N \}
\end{equation}
is invariant. $\mc S$ is called the {\em synchronization manifold} on which all nodes of the network follow the same orbit.
Fixing the local dynamics $f$ and the coupling function $h$, 
we obtain the following dichotomy of stability and instability of synchronization depending on whether the
graph is homogeneous or heterogeneous. 
\vspace{0.3cm}

\begin{minipage}[c]{14cm}
{\it
\emph{\textbf{Theorem C (Informal Statement).}}
\begin{itemize}
\item[a)] Take a diffusive coupling function $h(z_i,z_j)=\phi(z_j-z_i)$ with $\frac{d\phi}{dx}(0)\neq 0$. Then for almost every asymptotically large Erd\"os-R\'enyi graph and any diffusive coupling function in a sufficiently small neighbourhood of $h$ there is an interval $I\subset \R$ of coupling strengths for which $\mc S$ is stable (normally attracting).
\item[b)] For any diffusive coupling function $h(x,y)$, and for any sufficiently large heterogeneous layered graph $G$ with sequence of degrees satisfying \eqref{eq:layeredER} and \eqref{eq:layeredER2}, $\mc S$ is unstable.
\end{itemize}
} 
\end{minipage}\\
\

\begin{example}
Take $f(z)=2z\mod1$ and 
\[
h(z_i,z_j)=\sin(2\pi z_j-2\pi z_i)+\sin(2\pi z_j)-\sin(2\pi z_i).
\]
It follows from the proof of Theorem~\ref{MTheo:C} a) that  almost every asymptotically large Erd\"os-R\'enyi graph has a stable synchronization manifold for some values of the coupling strength ($\alpha\sim 0.3$) while any sufficiently large layered heterogeneous graph do not  have any stable synchronized orbit. However, in a layered graph $G$ the reduced dynamics for a hub node in the $i-$th layer is
\begin{align*}
g_{i_j}(z_{i_j})&=2z_{i_j}+\alpha \kappa_i\int \left[\sin(2\pi y-2\pi z_{i_j})+\sin(2\pi y)-\sin(2\pi z_{i_j})\right]dm(y)\mod 1\\
&=2z_{i_j}-{\alpha\kappa_i}\sin(2\pi z_{i_j})\mod 1\\
&=T_{{\alpha\kappa_i}}(z_{i_j}).
\end{align*}
By Theorem~\ref{MTheo:B} there is an interval for the coupling strength ($\alpha\kappa_i\sim 0.3$) for which $g_{i_j}$ has an attracting periodic sink and the orbit of the hubs in the layer follow this orbit (modulo small fluctuations) exhibiting coherent behaviour.
\end{example}

 \noindent 
{\bf Acknowledgements:} The authors would like to thank Mike Field, Gerhard Keller,  Carlangelo Liverani and Lai-Sang Young for 
fruitful conversations. We would also like to acknowledge the anonymous referee for finding many typos and providing useful comments. The authors also acknowledge funding by the  European Union ERC AdG grant no 339523 RGDD, 
the Imperial College Scholarship Scheme and the FAPESP CEPID grant no 213/07375-0.

\noindent



\noindent

\section{Setting and Statement of the Main Theorems } \label{Sec:SettRes}

Let us consider a directed graph $G$ whose set of nodes is $\mc N=\{1,\dots,N\}$ and set of directed edges $\Epsilon\subset \mc N\times \mc N$.  
In this paper we will be only concerned with in-degrees of a node, namely the number of edges that point to that node (which counts the contributions to the interaction felt by that node). Furthermore we suppose, in a sense that will be later specified, that the in-degrees $d_1,\dots,d_L$ 
of the nodes  $\{1,\dots,L\}$  are low compared to the size of the network while the in-degrees $d_{L+1},\dots,d_{N}$ of the nodes 
are comparable to the size of the network. For this reason, the first $L$ nodes will be called {\em low degree nodes} and the remaining $M=N-L$ nodes
will be called {\em  hubs}. Let $A$ be the adjacency matrix of $G$
\[
A = (A_{in})_{1\le i,n\le N}
\]
whose entry $A_{ij}$ is equal to one if the edge going from node $j$ to node $i$ is present, and zero otherwise.
So $ d_{i}=\sum_{j=1}^{N} A_{ij}$.
The important \emph{structural parameters} of the network are:
\begin{itemize}
\item  $L, M$ the number of low degree nodes, resp. hubs; $N=L+M$, the total number of nodes; 
\item $\Delta:=\max_{i}d_{i}$, the maximum in-degree of the hubs;
\item $\degree:=\max_{1<i\le L} d_i$, the maximum in-degree of the low degree nodes. 
\end{itemize}
The building blocks of the dynamics are:
\begin{itemize}
\item  the \emph{local dynamics}, $f:\mb T\rightarrow\mb T$, , $f(x)=\sigma x \mod1$, for some integer $\sigma\ge 2$; 
\item  the \emph{coupling function}, $h:\mb T \times\mb T\rightarrow \R$ which we assume is 
  $C^{10}$;
\item  the \emph{coupling strength}, $\alpha\in\R$.
\end{itemize}
We require the the coupling to be $C^{10}$ to ensure sufficiently fast decay of the Fourier coefficients in $h$. This is going to be useful in Section \ref{App:TruncSyst}. 
Expressing the coordinates  as $z=(z_1,...,z_{N})\in\mb T^N$, the discrete-time evolution is given by a map $F:\mb T^{N}\rightarrow\mb T^{N}$ defined by $z':=F(z)$
with 

\begin{equation}
z_i'= f(z_i)+\frac{\alpha}{\Delta}\sum_{n=1}^N A_{in}h(z_i,z_n)\mod 1 \quad , \quad   i=1,...,N. \label{Eq:CoupDyn}
\end{equation}

Our main result shows that low and high degree nodes will develop different dynamics when $\alpha$ is not too small. To simplify the formulation of our main theorem, 
we write $z=(x,y)$, with $x=(x_1,...,x_L):=(z_1,\dots,z_L) \in\mb T^L$ and  $y=(y_1,...,y_M):=(z_{L+1},\dots,z_N)\in\mb T^M$. 
Moreover, decompose 
\[
A=\left( \begin{array}{cc} A^{ll} & A^{lh} \\ A^{hl} & A^{hh} \end{array} \right)
\]
where $A^{ll}$ is a $L\times L$ matrix, etc.  Also write $A^{l}=(A^{ll} \, A^{lh} )$ and $A^{h}=(A^{hl} \, A^{hh})$.
In this notation we can write the map : 
\begin{align}
x_i'&= f(x_i)+\frac{\alpha}{\Delta}\sum_{n=1}^N A_{in}h(x_i,z_n)  \mod 1& i=1,...,L\label{Eq:CoupDyn1}\\
y_j'&=\redu_j (y_j)+\xi_j(z)  \quad \,\, \mod 1& j=1,...,M\label{Eq:CoupDyn2}
\end{align}
where, denoting the Lebesgue measure on $\mb T$ as $\leb_1$,
 \begin{equation}
\redu_j(y):=f(y)+\alpha\kappa_j \int h(y,x)d\leb_1(x)\mod 1, \quad\mbox{ }\quad \kappa_j:=\frac{d_{j+L}}{\Delta},  \label{Eq:MeanFieldMaps}
\end{equation}  
and 
\begin{equation} 
\xi_j(z):=\alpha \left[ \frac{1}{\Delta}\sum_{n=1}^N A^{h} _{jn}h(y_j,z_n) - \kappa_j  \int h(y_j,x)d\leb_1(x) \right].\label{Eq:average}\end{equation}

Before stating our theorem, let us give an intuitive argument why we write $F$ in the form (\ref{Eq:CoupDyn1}) and  (\ref{Eq:CoupDyn2}), and why for a very long
time-horizon one can model the resulting dynamics quite well by 
\[
x_i'\approx f(x_i) \quad\mbox{ and }\quad y_j'\approx\redu_j(y_j).
\]
To see this, note that for a heterogeneous network, the number of nonzero terms in the sum in \eqref{Eq:CoupDyn1} is an order of magnitude smaller than $\Delta$.
Hence when $N$ is large, the interaction felt by the low degree nodes becomes very small and therefore we have approximately $x_i'\approx f(x_i)$. 
So the low degree nodes are \say{essentiallly} uncorrelated from each other.  Since  the Lebesgue measure on $\mb T$, $\leb_1$, is $f$-invariant and since this measure is exact for the system, one can expect $x_i$, $i=1,\dots,L$ to behave as independent uniform random variables on $\mb T$, at least for
\say{most of the time}. Most of the $d_j=\kappa_j \Delta$ incoming  connections of  hub $j$ are with low degree nodes. 
It follows that the sum in (\ref{Eq:average}) should converge to 
\[
 \kappa_j  \int h(y_j,x)d\leb_1(x)
\]
 when $N$ is large, and so $\xi_j(z)$ should be close to zero. 

Theorem~\ref{Thm:Main} of this paper is a result which makes this intuition precise. 
In the following, we let $N_r(\Lambda)$ be the $r$-neighborhood of a set $\Lambda$
and we define  one-dimensional maps $\redu_j:\mb T\rightarrow\mb T$, $j=1,\dots,M$ to be hyperbolic
in a uniform sense. 


\begin{definition}[A Hyperbolic Collection of 1-Dimensional Map, see e.g.  \cite{MS}]\label{Def:AxiomA}
Given $\lambda\in (0,1)$, $r>0$ and $m,n\in\N$, we say that $\redu :\mb T\rightarrow\mb T$ is $(n,m,\lambda,r)$-hyperbolic if 
 there exists an attracting set $\Lambda\subset \mb T$, with
\begin{enumerate}
\item $\redu(\Lambda)=\Lambda$,
\item $|D_x\redu^n|<\lambda$ for all $x\in N_r(\Lambda)$,
\item $|D_x\redu^n|>\lambda^{-1}$ for all $x\in N_r(\Upsilon)$ where $\Upsilon:=\mb T\backslash W^s(\Lambda)$,
\item for each $x\notin N_r(\Upsilon)$, we have $g^k(x)\in N_r(\Lambda)$ for all $k\ge m$,
\end{enumerate} 
where $W^s(\Lambda)$ is the union of the stable manifolds of the attractor
\[
W^s(\Lambda):=\{x\in\mb T\mbox{ s.t. }\lim_{k\rightarrow\infty}d(\redu^k(x), \Lambda)=0\}.
\]
\end{definition}

It is well known, see  e.g. \cite[Theorem IV.B]{MS} that for each $C^2$ map $g\colon \mb T \to \mb T$ (with non-degenerate critical points), 
 the attracting sets are periodic and have uniformly bounded period. If we assume that $g$ is also hyperbolic, 
 we obtain a bound on the number of periodic attractors. A globally expanding map is hyperbolic since it correspond to the case where $\Lambda=\emptyset$.

%

We now give a precise definition of what we mean by heterogeneous network.
\begin{definition}
We say that a network with parameters $L,M,\Delta,\degree$ is \emph{$\eta$-heterogeneous} with $\eta>0$  if there is $p,q\in[1,\infty)$ with $1=1/p+1/q$, such that the following conditions are met:
\begin{align}
\Delta^{-1}L^{1/p}\degree^{1/q}&<\eta\tag{H1}\label{Eq:ThmCond1}\\
\Delta^{-1/p}M^{2/p}&<\eta \tag{H2}\label{Eq:ThmCond2'}\\
\Delta^{-1}ML^{1/p}&<\eta \tag{H3}\label{Eq:ThmCond2}\\
\Delta^{-2}L^{1+2/p}\degree&<\eta\tag{H4}\label{Eq:ThmCond3}
\end{align}
\end{definition}
\begin{remark}
\eqref{Eq:ThmCond1}-\eqref{Eq:ThmCond3} arise as sufficient conditions for requiring that the coupled system $F$ is 
\say{close} to the product system $f\times\dots \times f\times \redu_1\times \dots\times \redu_M:\mb T^{L+M}\rightarrow \mb T^{L+M}$ and preserve good hyperbolic properties on most of the phase space. They are verified in many common settings, as is shown in Appendix~\ref{Sec:ApRandGrap}. An easy example to have in mind where those conditions are asymptotically satisfied as $N\to \infty$ for every $\eta>0$, is the case where $M$ is constant (so $L\sim N$) and $\degree\sim L^{\tau}$, and $\Delta\sim L^\gamma$ with $0\leq \tau<1/2$ and $(\tau+1)/2<\gamma<1$. In particular the layered heterogeneous graphs satisfying \eqref{eq:layeredER2} in the introduction to the paper have these properties.
\end{remark}

\setcounter{mtheorem}{0}
\begin{mtheorem}\label{Thm:Main}
Fix $\sigma$, $h$ and an interval $[\alpha_1,\alpha_2]\subset\R$ for the parameter $\alpha$. Suppose that for all $1\leq j\leq M$ and $\alpha\in[\alpha_1,\alpha_2] $, each of the maps 
$\redu_j$, $j=1,\dots,M$ is $(n, m,\lambda,r)$-hyperbolic. Then there exist $\xi_0,\eta, C>0$ such that if the network is $\eta-$heterogeneous, for every $0<\xi<\xi_0$ and  for every $1\leq T\leq T_1$ with
 \[
 T_1=\exp[C\Delta\xi^2], 
 \]
 there is a set of initial conditions $\Omega_T\subset \mb T^{N}$ with
 \[
 m_{N}(\Omega_T)\geq 1-\frac{(T+1)}{T_1},
 \]
 such that for all $(x(0),y(0))\in\Omega_T$ 
 \[
\left |\xi_j(z(t))\right|<\xi,\quad\forall 1\leq j\leq M\mbox{ and }1\leq t\leq T.
 \]
\end{mtheorem}

\begin{remark}
The result hold under conditions \eqref{Eq:ThmCond1}-\eqref{Eq:ThmCond3} with $\eta$ sufficiently small, but uniform in the local dynamical parameters. Notice that $p$ has a different role in \eqref{Eq:ThmCond1}, \eqref{Eq:ThmCond2}, \eqref{Eq:ThmCond3}  and in  \eqref{Eq:ThmCond2'} so that a large $p$ helps the first one, but hinders the second and viceversa for a small $p$. 
\end{remark}

%

The proof of Theorem \ref{Thm:Main} will be presented separately in the case where $\redu_j$ is an expanding map of the circle for all the hubs (Section \ref{Sec:ExpRedMapsGlob}), and when at least one of the $\redu_j$ have an attracting point  (Section \ref{Sec:RedMapNonAtt}). 

The next theorem, is a consequence of results on density of hyperbolicity in dimension one and Theorem \ref{Thm:Main}. It shows that the hypothesis on hyperbolicity of the reduced maps $\redu_j$ is generically satisfied, and that generically one can tune the coupling strength to obtain reduced maps with attracting periodic orbits resulting in regular behaviour for the hub nodes. 

\setcounter{mtheorem}{1}

\def\dist{\mbox{dist}}
\begin{mtheorem}[Coherent behaviour for hub nodes]  \label{MTheo:B} 
For each $\sigma\in\N$, $\alpha\in \R,\kappa_j\in (0,1]$, there is an open and dense set $\Gamma\subset C^{10}(\mb T^2;\R)$ such that, for all coupling functions $h\in \Gamma$, 
$\redu_j\in C^{10}(\mb T,\mb T)$, defined by Eq. \eqref{Eq:MeanFieldMaps}, is hyperbolic (as in Definition \ref{Def:AxiomA}).

There is an open and dense set $\Gamma'\subset C^{10}(\mb T^2;\R)$ such that for all $h\in\Gamma'$ there exists an interval $I\subset\R$ for which if $\alpha\kappa_j\in I$ then $g_j$ has a nonempty and finite periodic attractor. Furthermore, suppose that $h\in\Gamma'$, the graph $G$ satisfies the assumptions of Theorem \ref{Thm:Main} for some $\xi>0$ sufficiently small, and that for the hub $j\in\mc N$,
$\alpha \kappa_j \in I$. Then there exists $C>0$ and $\chi\in (0,1)$ so that the following holds. 
Let $T_1:=\exp[C\Delta\xi^2]$. There is a set of initial conditions $\Omega_{T}\subset \mb T^{N}$ with
 \[
 m_{N}(\Omega_T)\geq 1-\frac{(T+1)}{T_1} - \xi^{1-\chi}
 \]
so that for all $z(0)\in\Omega_T$ there is a periodic orbit of $g_{j}$,  $O=\{\bar z(k)\}_{k=1}^p$, for which 
 \[
\dist(z_{j}(t), \bar z(t \mbox{ mod } p))\le \xi
 \] 
 for each $1/\xi\le t\le T\le T_1$. 
\end{mtheorem}

\begin{proof}
See Appendix \ref{appendix:thmc}. 
\end{proof}

\begin{remark}
In the setting of the theorem above, consider the case where two hubs $j_1,j_2\in\mc N$ have the same connectivity $\kappa$, and their reduced dynamics $g_{j_i}$ have a unique attracting periodic orbit. In this situation their orbits  closely follow this unique orbit (as prescribed by the theorem) and, apart from a phase shift $\tau\in\N$, they will be close one to another resulting in highly coherent behaviour:
\[
\dist(z_{j_1}(t),z_{j_2}(t+\tau))\le2\xi
\]  
under the same conditions of Theorem \ref{MTheo:B}.
In general, the attractor of $g_{j_i}$ is the union of a finite number of attracting periodic orbits. Choosing initial conditions for the hubs's coordinates in the same connected component of the basin of attraction of one of the periodic orbits yield the same coherent behaviour as above. 
\end{remark}

In the next theorem we show that for large heterogeneous networks, in contrast with the case of homogeneous networks, coherent behaviour of the hubs is the most one can hope for, and global synchronisation is unstable.
\begin{definition}[Erd\"os-R\'enyi Random Graphs \cite{Bollobas}] \label{Def:ErdRen} For every $N$ and $p$, an Erd\"os-R\'enyi random graph is a discrete probability measure on the set $\mc G(N)$ of undirected graphs on $N$ vertices which assigns independently probability $p\in (0,1)$ to the presence on any of the edge. 

Calling $\mb P_p$ such probability and $(A_{ij})$ the symmetric adjacency matrix of a graph randomly chosen according to $\mb P_p$, $\{A_{ij}\}_{j\ge i}$ are i.i.d random variables equal to 1 with probability $p$, and to 0 with probability $1-p$.  
\end{definition}
{\it
\begin{mtheorem}[Stability and instability of synchrony] \label{MTheo:C}
\
 
\begin{itemize}
\item[a)]  Take a diffusive coupling function $h(x,y)=\phi(y-x)$ for some $\phi:\mb T\rightarrow \R$ with $\frac{d\phi}{dx}(0)\neq 0$. For any coupling function $h'$ in a sufficiently small neighbourhood of $h$, there is an interval $I\subset \R$ of coupling strengths such that for  any $p$~$\in \left(\frac{\log N}{N},1\right]$ there exists a subset of undirected homogeneous graphs ${\mc G}_{Hom}(N) \subset \mc G(N)$, with $\mb P_p({\mc G}_{Hom}(N))\to 1$
as $N\to \infty$ so that  for any 
$\alpha \in \mc I $ the synchronization manifold $\mathcal{S}$, defined in Eq. \eqref{Eq:SyncManif}, is
locally exponentially stable (normally attracting) for each network coupled on $G\in {\mc G}_{Hom}(N)$.

\item[b)] Take any sequence of graphs $\{G(N)\}_{N\in\N}$ where $G(N)$ has $N$ nodes and non-decreasing sequence of degrees $\bo d(N)=(d_{1,N},...,d_{N,N})$. Then, if $d_{N,N}/d_{1,N}\rightarrow\infty$ for $N\rightarrow\infty$,  for any diffusive coupling $h$ and coupling strength $\alpha\in\R$ there is $N_0\in \N$ such that the synchronization manifold $\mathcal{S}$ is unstable for the network coupled on $G(N)$ with $N>N_0$.

\end{itemize}
\end{mtheorem}
}
\begin{proof}
See Appendix \ref{App:RandGrap}.
\end{proof}

\setcounter{mtheorem}{2}
%
 
\subsection{Literature Review and the Necessity of a New Approach for HCM} \label{subsec:literature}
We just briefly recall the main lines of research on  dynamical systems coupled in networks to highlight the need of a new perspective that meaningfully describe HCM. For more complete surveys see \cite{porter2014dynamical,Fern}.
\begin{itemize}
\item {\bf Bifurcation Theory} \cite{Gol-Stewart98,Gol-Stewart2006,Koiller-Young,Field-etal-2011,rink2015coupled}. In this approach typically there exists a low dimensional invariant set where the interesting behaviour happens. Often the equivariant group structure is used to obtain a center manifold reduction. In our case the networks are not assumed to have symmetries (e.g. random networks) and the relevant invariant sets are fractal like containing unstable manifolds of very high dimension  (see Figure \ref{Fig:Attractor}). For these reasons it is difficult to frame HCM in this setting or use perturbative arguments. 
\item The study of {\bf Global Synchronization} \cite{Kuramoto84,Barahona2002,Mirollo2014,Pereira2014} deals with the convergence of orbits to a low-dimensional invariant manifold where all the nodes evolve coherently. HCM do not exhibit global synchronization. The synchronization manifold in Eq. \eqref{Eq:SyncManif} is unstable (see Theorem \ref{MTheo:C}). Furthermore, many works \cite{Balint,strogatz2000kuramoto} deal with global synchronization when the network if fully connected (all-to-all coupling) by studying the uniform mean field in the thermodynamic limit. On the other hand, we are  interested in the case of a finite size system and when the mean field is not uniform across connectivity layers.
\item The statistical description of {\bf Coupled Map Lattices} \cite{Kaneko1992,BunSinai,Baladi1,Baladi2,Keller1,Keller2,Keller3,chazottes2005dynamics,selley2016symmetry}  deals with maps coupled on homogeneous graphs and considers the persistence and ergodic properties of invariant measures when the magnitude of the coupling strength goes to zero. In our case the coupling regime is such that hub nodes are subject to an order one perturbation coming from the dynamics. Low degree nodes still feel a small contribution from the rest of the network, however, its magnitude depends on the system size and to make it arbitrarily small the dimensionality of the system must increase as well. 
\end{itemize}

It is worth mentioning that dynamics of coupled systems with different subsystems appears also in \emph{slow-fast system} dynamics \cite{MR3064670,MR3556527,MR2316999}. Here, loosely speaking, some (slow) coordinates evolve as \say{$id +\epsilon h$} and the others have good ergodic properties. In this case one can apply ergodic averaging and obtain a good approximation of the slow coordinates for time up to time $T\sim\epsilon^{-1}$. In our case, spatial rather than time ergodic averaging takes place and there is no dichotomy on the time scales at different nodes. Furthermore, the role of the perturbation parameter is played by $\Delta^{-1}$ and we obtain $T=\exp(C\Delta)$, rather than the polynomial estimate obtained in slow-fast systems.

\section{Sketch of the Proof and the Use of a {\lq}Truncated{\rq} System}
\label{sec:sketch}

\subsection {A Trivial Example Exhibiting Main Features of HCM} \label{Sec:StarNetExamp}
We now present a more or less trivial example which already presents all the main features of heterogeneous coupled maps, namely 
\begin{itemize}  \item existence of a set of \emph{\say{bad} states} with large fluctuations of the mean field, \item control on the \emph{hitting time} to the bad set, \item \emph{finite time} exponentially large on the size of the network.\end{itemize}
Consider the evolution of $N=L+1$ doubling maps on the circle $\mb T$ interacting on a {\em star network}  with nodes $\{1,...,L+1\}$ and set of directed edges $\mc E=\{(i,L+1): 1\leq i\leq L\}$ (see Figure \ref{Fig:Star}). The hub node $\{L+1\}$  has an incoming directed edge from every other node of the network, while the other nodes have just the outgoing edge. Take as interaction function the diffusive coupling $h(x,y):=\sin(2\pi y)-\sin (2\pi x)$. Equations \eqref{Eq:CoupDyn1} and \eqref{Eq:CoupDyn2} then become

\begin{figure}[htbp]
\centering
\includegraphics[width=1.8in]{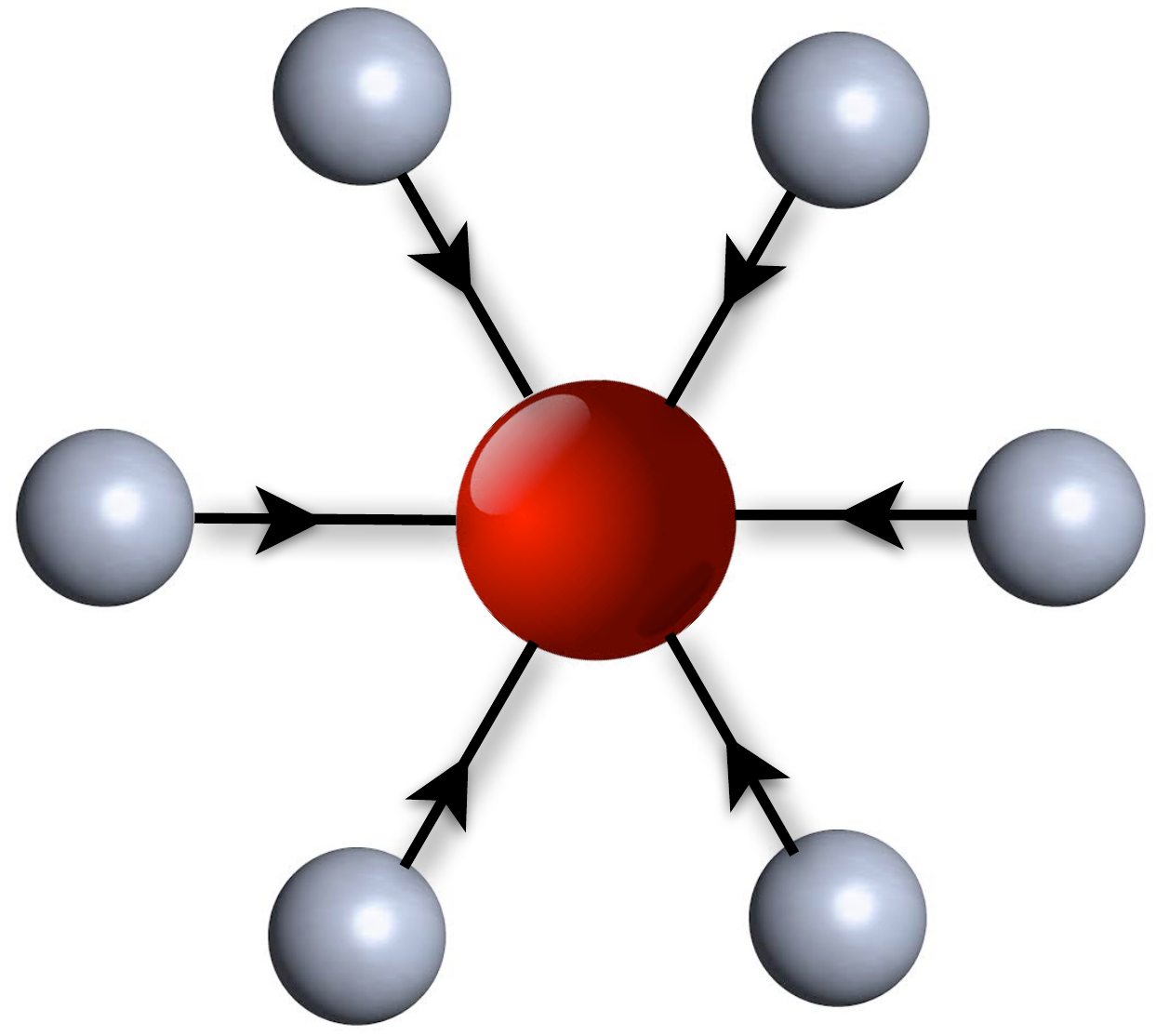}
\vspace{-0.4cm}
\caption{Star network with only incoming arrows.}
\label{Fig:Star}
\end{figure}
\begin{align}
x_i(t+1)&= 2x_i(t)&\mod 1\quad& 1\leq i\leq L\label{Eq:StarNetEq1}\\
y(t+1)\phantom{l}&= 2y(t)+ \frac{\alpha}{L}\sum_{i=1}^{L}\left[\sin(2\pi x_i(t))-\sin(2\pi y(t))\right]&\phantom{....}\mod 1.\quad&\label{Eq:StarNetEq2}
\end{align}
The low degree nodes evolve as an uncoupled doubling map making the above a \emph{skew-product} system on the base $\mb T^{L}$ akin to the one extensively studied in \cite{MR1862809}.
One can rewrite the dynamics of the forced system (the hub) as
\begin{equation}\label{Eq:StarHub}
y(t+1)=2y(t)-\alpha\sin(2\pi y(t))+\frac{\alpha}{L}\sum_{i=1}^L\sin(2\pi x_i(t))
\end{equation}
and notice that defining $\redu(y):=2y-\alpha\sin(2\pi y)\mod 1$, the evolution of $y(t)$ is given by the application of $\redu$ plus a noise term 
\begin{equation}\label{Eq:Fluct}
\xi(t)=\frac{\alpha}{L}\sum_{i=1}^L\sin(2\pi x_i(t))
\end{equation}
depending on the low degree nodes coordinates. The Lebesgue measure on $\T^L$ is invariant and mixing for the dynamics restricted to first $L$ uncoupled coordinates. The set of bad states where fluctuations \eqref{Eq:Fluct} are above a fixed threshold $\epsilon>0$ is
\begin{align*}
\mc B_\epsilon&:=\left\{x\in\mb T^L: \left|\frac{1}{L}\sum_{i=1}^L\sin(2\pi x_i)-\mb E_\leb[\sin(2\pi x)]\right|>\epsilon\right\}\\
&=\left\{x\in\mb T^L: \left|\frac{1}{L}\sum_{i=1}^L\sin(2\pi x_i)\right|>\epsilon\right\}
\end{align*}
Using large deviation results one can upper bound the measure of the set above as
\[
\leb_L(\mc B_{\epsilon})\leq \exp(-C\epsilon^2 L).
\]
($C>0$ is a constant uniform on $L$ and $\epsilon$, see the  Hoeffding Inequality in Appendix \ref{App:TruncSyst} for details). Since we know that the dynamics of the low degree nodes is ergodic with respect the measure $\leb_L$ we have the following information regarding the time evolution of the hub.
\begin{itemize}
\item  The set $\mc B_\epsilon$ has positive measure. Ergodicity of the invariant measure implies that a generic initial condition will visit $\mc B_\epsilon$ in finite time, making any mean-field approximation result for infinite time hopeless.
\item As a consequence of Kac Lemma, the average hitting time to the set $\mc B_\epsilon$ is $\leb_L(\mc B_{\epsilon})^{-1}\ge\exp(C\epsilon^2 L)$, thus exponentially large in the dimension.
\item From the invariance of the measure $\leb_L$, for every $1\leq T\leq \exp(C\epsilon^2 L)$ there is $\Omega_T\subset \mb T^{L+1}$ with measure $\leb_{L+1}(\Omega_T)>1-T\exp(-C\epsilon^2 L)$ such that $\forall x\in\Omega_T$ and for every $1\leq t\leq T$
\[
\left| \frac{1}{L}\sum_{i=1}^L\sin(2\pi x_i(t))\right|\leq\epsilon.
\]
\end{itemize}


\subsection{Truncated System}
We obtain a description of the coupled system by restricting our attention to a subset of phase space where the evolution prescribed by equations \eqref{Eq:CoupDyn1} and \eqref{Eq:CoupDyn2} resembles the evolution of the uncoupled mean-field maps, and we redefine the evolution outside this subset in a convenient way. This leads  to the 
definition of a {\em truncated } map $F_\epsilon:\mb T^{N}\rightarrow\mb T^{N}$, for which the fluctuations of the mean field averages are artificially 
cut-off at the level $\epsilon>0$, resulting in a well behaved hyperbolic dynamical system. In the following sections we will then determine existence and bounds on the invariant measure for this system and prove that the  portion of phase space where the original system and the truncated  one coincide is almost full measure with a remainder exponentially small in the parameter $\Delta$. 


Note that  since $h\in C^{10}(\mb T^2;\R)$, its Fourier series  
\[
h(x,y)=\sum_{s=(s_1,s_2)\in \mb Z^2}c_{s}\theta_{s_1}(x)\upsilon_{s_2}(y),
\]
where $c_s\in\R$ and $\theta_{i}:\mb T\rightarrow [0,1]$ form a base of trigonometric functions, converges uniformly and absolutely on $\mb T^2$. Furthermore, for all $s\in\mb Z^2$
\begin{equation}
|c_s|\leq \frac{\|h\|_{C^{10}}}{|s_1|^5|s_2|^5}.
\end{equation}
Taking $\bar \theta_{s_1}=\int \theta_{s_1}(x) d\leb_1(x) $ we get 
\begin{equation} \xi_j(z):=\alpha \sum_{s\in\mb Z^2}c_s \left[ \frac{1}{\Delta}\sum_{n=1}^L A^h_{jn} \theta_{s_1}(z_n) - \kappa_j \bar \theta_{s_1}  \right]\upsilon_{s_2}(y_j)+\frac{\alpha}{\Delta}\sum_{n=1}^M A^h_{jn}h(y_j,y_n)  
\label{Eq:E'}
\end{equation}
For every $\epsilon>0$ choose a $C^{\infty}$ map $\zeta_\epsilon:\R\rightarrow\R$ with $\zeta_\epsilon(t)=t$ for $|t|<\epsilon$, $\zeta_\epsilon(t)=2 \epsilon$ for $|t|>2 \epsilon$. So
for each $\epsilon>0$, the function $t\mapsto |D_t\zeta_\epsilon|$ is uniformly bounded in $t$ and $\epsilon$. We define the evolution for the truncated  dynamics 
 $F_\epsilon\colon \mb T^{L+M}\rightarrow\mb T^{L+M}$ by the following modification of equations (\ref{Eq:CoupDyn1}) and (\ref{Eq:CoupDyn2}):
\begin{align}
x_i'&= f(x_i)+\frac{\alpha}{\Delta}\sum_{n=1}^N A_{in}h(x_i,z_n)  \mod 1& i=1,...,L \label{Eq:CoupDyn1'}
\\
y_j'&=\redu_j(y_j) +\xi_{j,\epsilon}(z)  \quad \,\, \mod 1& j=1,...,M\label{Eq:CoupDyn2'}
\end{align}
where the expression of $\xi_{j,\epsilon}(z)$ modifies that of $\xi_{j}(z)$ in \eqref{Eq:E'}: 
\begin{equation}
\xi_{j,\epsilon}(z):= \alpha\sum_{s\in\mb Z^2}c_s \zeta_{\epsilon |s_1|}\left(\frac{1}{\Delta}\sum_{i=1}^L A_{ji}\theta_{s_1}(x_i) - \kappa_j \bar \theta_{s_1} \right)\upsilon_{s_2}(y_j)+  \frac{\alpha}{\Delta}\sum_{n=1}^M A^h_{jn}h(y_j,y_n). \label{eq:xijeps}
\end{equation}
 So the only difference between $F$ and $F_\epsilon$ are the cut-off functions $\zeta_{\epsilon|s| }$ appearing in (\ref{eq:xijeps}). 
For every $\epsilon>0$, $j\in\{1,...,M\}$ and $s_1\in\Z$ define
\begin{equation}\label{Eq:DefBadSetComp}
\mc B_{\epsilon}^{(s_1,j)}:=\left\{x\in\mb T^{L}:\left|\frac{1}{\Delta}\sum_{i=1}^L A^h_{ji}\theta_{s_1}(x_i)-\kappa_j\bar\theta_{s_1}\right|>\epsilon|s_1|\right\}.
\end{equation}
The set where $F$ and $F_\epsilon$ coincide is $\mc Q_\epsilon\times\mb T^M$, with
\begin{equation}\label{Eq:DefQdelta}
\mc Q_\epsilon:=\bigcap_{j=1}^M\bigcap_{s_1\in\mb Z}\mb T^L\backslash \mc B_\epsilon^{(s_1,j)}
\end{equation}
the subset of $\mb T^L$ where all the fluctuations of the mean field averages of the terms of the coupling are less than the imposed threshold. 
The set $\mc B_\epsilon:=\mc Q_\epsilon^c$, is the portion of phase space for the low degree nodes were the fluctuations exceed the threshold, and the systems $F$ and $F_\epsilon$ are different. Furthermore we can control the perturbation introduced by the term $\xi_{j,\epsilon}$ in equation \eqref{Eq:CoupDyn2} so that $F_\epsilon$ is close to the hyperbolic uncoupled product map $\bo f:\mb T^N\rightarrow\mb T^N$ 
\begin{equation}\label{Eq:UncSystMF}
\bo f(x_1,..,x_L,y_1,...,y_M):=(f(x_1),...,f(x_L),\redu_1(y_1),...,\redu_M(y_M)).
\end{equation}
All the bounds on relevant norms of $\xi_{j,\epsilon}$ are reported in Appendix \ref{App:TruncSyst}. 
To upper bound the Lebesgue measure $\leb_L(\mc B_\epsilon)$ we use the  Hoeffding's inequality (reported in Appendix \ref{App:TruncSyst}) on concentration of the average of independent bounded random variables.

\begin{proposition}\label{Prop:UppBndBLeb}
\begin{equation}\label{Eq:UppBndBLeb}
\leb_L(\mc B_\epsilon)\leq\frac{\exp\left[-\frac{\Delta\epsilon^2}{2}+\mc O(\log M)\right]}{1-\exp\left[-\frac{\Delta\epsilon^2}{2}\right]}.
\end{equation}
\end{proposition}
\begin{proof}
See Appendix \ref{App:TruncSyst}.
\end{proof}
This gives an estimate of the measure of the bad set with respect to the reference measure invariant for the uncoupled maps. In the next section we use this estimate to upper bound the measure of this set with respect to SRB measures for $F_\epsilon$, which is the measure giving statistical informations on the orbits of $F_\epsilon$.
\begin{remark}
Notice that in \eqref{Eq:UppBndBLeb} we expressed the upper bound only in terms of orders of functions of the network parameters, but all the constants could be rigorously estimated in terms of the coupling function and the other dynamical parameters of the system. In particular, where the expression of the coupling function was known one could have obtained better estimates on the concentration via large deviation results (see for example Cram\'er-type inequalities in \cite{dembo2009large}) which takes into account more than just the upper and lower bounds of $\theta_s$. In what follows, however, we will be only interested in the order of magnitudes with respect to the aforementioned parameters of the network ($\Delta$, $\degree$, $L$, $M$).  
\end{remark}

\subsection{Steps of the Proof and Challenges}
The basic steps of the proof are the following:
\begin{itemize}
\item[(i)] First of all we are going to restrict our attention to the case where the maps $\redu_j$ satisfy Definition~\ref{Def:AxiomA} with $n=1$
\item[(ii)]  Secondly, hyperbolicity of the map $F_\epsilon$ is established for an $\eta-$heterogeneous network with $\epsilon,\eta>0$ small. This is achieved by constructing forward and backward invariant cone-fields made of expanding and contracting directions respectively for the cocycle defined by application of $D_z F_\epsilon$ \eqref{Eq:DiffAuxMap}. 
\item[(iii)] Then we estimate the distortion of the maps along the unstable directions, keeping all dependencies on the structural parameters of the network explicit.
\item[(iv)] We then use a geometric approach employing what are sometimes called \emph{standard pairs}, \cite{climenhaga2016geometric}, to estimate the regularity properties of the SRB measures for the endomorphism $F_\epsilon$, and the hitting time  to the set $\mc B_\epsilon$
\item[(v)] Finally we show that Mather's trick allows us to generalise the proofs to the case in which $\redu_j$ satisfy Definition~\ref{Def:AxiomA} with $n\neq 1$.
\end{itemize}

We consider separately the cases where all the reduced maps $\redu_j$ are expanding and when some of them have non-empty attractor (Section \ref{Sec:ExpRedMapsGlob} and Section \ref{Sec:RedMapNonAtt}). At the end of Section \ref{Sec:RedMapNonAtt} we put the results together to obtain the proof of Theorem \ref{Thm:Main}.

In the above points we treat $F_\epsilon$ as a perturbation of a product map where the magnitude of the perturbation depends on the network size. In particular, we want to show that $F_\epsilon$ is close to the uncoupled product map $\bo f$.
To obtain this, the dimensionality of the system needs to increase, changing the underlying phase space. This leads to two main challenges. First of all, increasing the size of the system propagate nonlinearities of the maps and reduces the global regularity of the invariant measures.  Secondly, the situation is inherently different from usual perturbation theory where one considers a parametric family of dynamical systems on the same phase space. Here, the parameters depend on the system's dimension. As a consequence one needs to make all estimates explicit on the system size. For these reasons we find the geometric approach advantageous with respect to the functional analytic approach \cite{keller1999stability} where the explicit dependence of most constants on the dimension are hidden in the functional analytic machinery.

\paragraph{Notation} As usual, we write $\mc O(N)$ and $\mc O(\epsilon)$ for an expression so that $\mc O(N)/N$ resp. $\mc O(\epsilon)/\epsilon$ is bounded as $N\to \infty$ and $\epsilon\downarrow 0$.
We use short-hand notations $[n]:=\{1,...,n\}$ for the natural numbers up to $n$.
%

Throughout the paper $\leb$, $\leb_{n}$ stand for the Lebesgue measure on $\mb T$ and $\mb T^n$ respectively. Given an embedded manifold $W\subset \mb T^{N}$, $m_W$ stands for the Lebesgue measure induced on $W$.  

We indicate with $D_xG$ the differential of the function $G$ evaluated at the point $x$ in its domanin.

\section{Proof of Theorem~\ref{Thm:Main} when all Reduced Maps are Uniformly Expanding}\label{Sec:ExpRedMapsGlob}
In this section we assume 
that the collection of reduced maps $\redu_{j}$, $j=1,\dots,M$, from 
equation \eqref{Eq:MeanFieldMaps} is uniformly expanding. As shown in Lemma~\ref{lem:n=1}
this means that we can assume that there exists $\lambda\in (0,1)$,  so that $|\redu_j(x) |\ge \lambda^{-1}$
for all $x\in \mb T$ and all $j=1,\dots,M$.

First of all pick $1\le p\le \infty$, let $1\le q\le \infty$ be so that $1/p+1/q=1$ and  consider the norm defined as 
\[
\|\cdot\|_p:=\|\cdot\|_{p,\R^L}+\|\cdot\|_{p,\R^M}
\]
where $\|\cdot \|_{p,\R^k}$ is the usual $p-$norm on $\R^k$. $\|\cdot\|_p$ induces the operator norm of any linear map $\mc L\colon \R^{N}\to \R^{N}$, namely  
\[
\|\mc L\|_p:=\sup_{\substack{v\in\R^{N} \\ \|v\|_p=1}}\frac{\|\mc Lv\|_p}{\|v\|_p}. 
\]
and the distance $d_p:\mb T^{N}\times\mb T^{N}\rightarrow \R^+$ on $\mb T^{N}$.

\begin{theorem}\label{Thm:InvMeasDenExpand}
There are $\eta_0,\epsilon_0>0$ such that under \eqref{Eq:ThmCond1}-\eqref{Eq:ThmCond3} with $\eta<\eta_0$ and for all $\epsilon<\epsilon_0$ there exists an absolutely continuous invariant probability measure $\nu$ for $F_\epsilon$. The density $\rho=d\nu/dm_N$ satisfies for all $z,\bar z\in\mb T^N$
\begin{equation}\label{Eq:InvDensProp}
\frac{\rho(z)}{\rho(\bar z)}\leq\exp\left\{ad_p(z,\bar z)\right\},\quad a=\mc O(\Delta^{-1}\degree L)+\mc O(M).
\end{equation}

\end{theorem}

In Section \ref{ConExp} we obtain conditions on the heterogenous structure of the network 
which ensure that the truncated  system $F_\epsilon$ is sufficiently close, in the $C^1$ topology, to the uncoupled system $\bo f$, in Eq. \eqref{Eq:UncSystMF}, with the hubs evolving according to the low-dimensional approximation $\redu_j$, for it to preserve expansivity when the 
network is large enough. In this setting $F_\epsilon$ is a uniformly expanding endomorphism and therefore has an absolutely continuous invariant measure $\nu$ whose density $\rho=\rho_\epsilon$ is a fixed point of the transfer operator of $F_\epsilon$
\[
P_\epsilon:L^1\left(\mb T^{N},m_{N}\right)\rightarrow L^1\left(\mb T^{N},m_{N}\right).
\]
(See Appendix \ref{Ap:TranOp} for a quick review on the theory of transfer operators). 
For our purposes we will also require bounds on $\rho$ which are explicit on the structural parameters of the network (for suitable $\epsilon$).
In Section \ref{InvConeSec} we obtain bounds on the distortion of the Jacobian of $F_\epsilon$ (Proposition \ref{Prop:DistJac}), which in turn allow us to prove the existence of a cone of functions with controlled regularity which is invariant under the action of $P_\epsilon$ (Proposition \ref{Prop:ConInv}) and to which $\rho$ belongs. 
To obtain the conclusion of Theorem \ref{Thm:Main}, we need that the  $\nu$-measure of the bad set
is small which will be obtained from an upper bound for the supremum of the functions in the invariant cone. 
This is what is shown in Section \ref{Sec:ProofTHM1} under some additional conditions on the network.

\subsection{Global Expansion of $F_\epsilon$}\label{ConExp}

 \begin{proposition}\label{Prop:ExpAuxSys}
Suppose that for every $j\in[M]$ the reduced map $\redu_j$  
is uniformly expanding, i.e. there exists $\lambda\in (0,1)$ so that $|D_y\redu_j|>\lambda^{-1}$ for all $y\in\mb T$. Then
 \begin{itemize}
 \item[(i)] there exists $C_\#$ (depending on $\sigma$, $h$ and $\alpha$ only) such that for every $1\leq p\leq \infty$, $z\in\mb T^{N}$, and $w\in\R^{N}\backslash\{0\}$
\[
\frac{\|(D_{z}F_\epsilon) w\|_p}{\|w\|_p}\geq \left[\min\{\sigma,\lambda^{-1} -\epsilon C_\#\}-\mc O(\Delta^{-1}\degree)-\mc O(\Delta^{-1/p}M^{1/p})-\mc O(\Delta^{-1}N^{1/p}\degree^{1/q})\right];
\]
\item[(ii)] there exists $\eta>0$ such that if \eqref{Eq:ThmCond1} and \eqref{Eq:ThmCond2} are satisfied together with
\begin{align}\label{Eq:EpsilonCond}
\epsilon<\frac{\lambda^{-1}-1}{C_\#}
\end{align}
then there exists  $\bar \sigma>1$ (not depending on the parameters of the network or on $p$), so that 
\[
\frac{\|(D_{z}F_\epsilon)w\|_p}{\|w\|_p}\geq\bar\sigma>1,\quad \forall z\in\mb T^{N},\mbox{ }\forall w\in\R^{N}\backslash\{0\}.
\]
\end{itemize}
\end{proposition}
\begin{proof}
To prove (i), let $z=(x,y)\in \mb T^{L+M}$ and  $w=\binom{u}{v}\in \R^{L+M}$ and
$$
\binom{u'}{v'}=D_zF_\epsilon\binom{u}{v},\quad u'\in\R^{L},v'\in\R^M.
$$
\def\hj{\j^*} \def\hm{m^*}
Using \eqref{Eq:CoupDyn1'}-\eqref{Eq:CoupDyn2'}, or \eqref{Eq:DiffAuxMap}, 
we obtain that for every $1\leq i\leq L$ and $1\leq j\leq M$, 
\begin{align*}
u'_i&=\left[D_{x_i}f+\frac{\alpha}{\Delta}\sum_{n=1}^N A_{in}h_1(x_i,z_n)\right]u_i+\frac{\alpha}{\Delta}\sum _{n=1}^N A_{in}h_1(x_i,z_n)w_n \\
v'_j&=\sum_{\ell=1}^L\partial_{x_\ell}\xi_{j,\epsilon}u_\ell+\frac{\alpha}{\Delta}\sum_{m=1}^M A^{hh}_{jm}h_2(y_j,y_m)v_m+\left[D_{y_j}\redu_j+\partial_{y_j}\xi_{j,\epsilon}\right]v_{j}.
\end{align*}
where $h_1$ and $h_2$ denote the partial derivatives with respect to the first and second variable. Hence
\[
\|u'\|_{p,\R^L}\geq \left(\sigma-\mc O(\degree\Delta^{-1})\right)\|u\|_p-\mc O(\Delta^{-1}L^{1/p})\max_{i=1,\dots,L} \left[\sum_{n=1}^NA_{in}|w_n|\right].
\]
Recall that, for any $k\in\N$, if $w\in\R^k$ then 
\begin{equation}\label{Eq:Ineq1pspaces}
\|w\|_{1,\R^k}\leq k^{1/q}\|w\|_{p,\R^k}, \mbox{ with }\frac{1}{p}+\frac{1}{q}=1
\end{equation}
for every $1\leq p\leq \infty$. Thus
\[
\sum_{n=1}^N A_{in} |w_n|\leq \degree^{1/q}\left(\sum_{n=1}^NA_{in}|w_n|^p \right)^{1/p}\leq \degree^{1/q}\|w\|_p
\]
since at most $\degree$ terms are non-vanishing in the sum $\left(\sum_{n=1}^NA_{in}|w_n|^p \right)$,  we can view 
as a vector in $\R^\degree$, which implies
\[
\|u'\|_{p,\R^L}\geq\left(\sigma-\mc O(\degree\Delta^{-1})\right)\|u\|_p-\mc O(\Delta^{-1}L^{1/p}\degree^{1/q})\|w\|_p
\]
Analogously using the estimates in Lemma~\ref{Lem:XiProp}
\begin{align}
\|v'\|_{p,\R^M}&\geq  \left(\lambda^{-1}-\epsilon C_\# \right)\|v\|_p-\mc O(\Delta^{-1}M^{1/p})\max_{j=1,\dots,M} \left[\sum_{n}A_{jn}|w_n|\right]\label{Eq:pnormExpUppBnd2}\\
&\geq \left(\lambda^{-1}-\epsilon C_\#-\mc O(\Delta^{-1}M) \right)\|v\|_p-\mc O(\Delta^{-1/p}M^{1/p})\|w\|_p\nonumber
\end{align}
since in the sum  $\sum_{n}A_{jn}|w_n|$ in \eqref{Eq:pnormExpUppBnd2}, at most $\Delta$ terms are different from zero and since $\Delta^{-1}\Delta^{1/q}=\Delta^{-1/p}$. This implies
\begin{align*}
\frac{\|(u',v')\|_p}{\|(u,v)\|_p}&=\frac{\|u'\|_{p,\R^L}+\|v'\|_{p,\R^M}}{\|(u,v)\|_p}\geq\\
&\geq \left[\min\{\sigma-\mc O(\Delta^{-1}\degree),\lambda^{-1} -\epsilon C_\#-\mc O(\Delta^{-1}M)\}-\mc O(\Delta^{-1/p}M^{1/p})-\mc O(\Delta^{-1}L^{1/p}\degree^{1/q})\right] 
\end{align*}

For the proof of (ii), notice that condition \eqref{Eq:EpsilonCond} implies that $\min\{\sigma,\lambda^{-1}-\epsilon C_\#\}>1$ and conditions \eqref{Eq:ThmCond1}-\eqref{Eq:ThmCond2} imply that the $\mc O$ are bounded  by $\eta$ and so 
\[
\frac{\|D_{z}F_\epsilon w\|_p}{\|w\|_p}\geq \min\{\sigma,\lambda^{-1}-\epsilon C_\#\}-\mc O(\eta),\quad \forall w\in\R^{N}\backslash\{0\}
\]
and choosing $\eta>0$ sufficiently small one obtains the proposition. 
\end{proof}
Now that we have proved that $F_\epsilon$ is expanding, we know from the ergodic theory of expanding maps, that it also has an invariant measure we call $\nu$, with density $\rho=d\nu/d\leb_{N}$. The rest of the section is dedicated to upper bound $\nu(\mc Q_{\epsilon})$.

\subsection{Distortion of $F_\epsilon$}\label{InvConeSec}
%

\begin{proposition}\label{Prop:DistJac}
If conditions \eqref{Eq:ThmCond1}-\eqref{Eq:ThmCond2} are satisfied then there exists $\epsilon_0$ 
(depending only on $\sigma$, $|\alpha|$ and the coupling function $h$) such that if $\epsilon<\epsilon_0$ then for every $z,\bar z\in\mb T^{N}$
\[
\frac{|D_{z}F_\epsilon|}{|D_{\bar z}F_\epsilon|}\leq \exp\left\{\left[\mc O(\Delta^{-1}\degree L)+\mc O(M)\right]d_\infty(z,\bar z)\right\}.
\]
\end{proposition}
\begin{proof}
To estimate the ratios consider the matrix $\mc D(z)$ obtained from $D_{z} F_\epsilon$ factoring $D_{x_i}f=\sigma$ out of the $i$-th column ($i\in[N]$), and $D_{y_j}\redu_j$ out of the $(j+L)$-th column ($j\in[M]$). Thus 
\begin{equation}\label{Eq:DExpression}
[\mc D(z)]_{k,\ell}:=\left\{\begin{array}{lr}
1+\frac{\alpha}{\Delta}\sum_{n=1}^LA_{kn}\frac{h_1(x_k,z_n)}{\sigma}& k=\ell \leq L,\\
\frac{\alpha}{\Delta}A_{k\ell}\frac{h_2(x_k,x_\ell)}{\sigma}& k\neq \ell \leq L,\\
\frac{\partial_{x_\ell}\xi_{k-L,\epsilon}}{\sigma}& k>L, \ell\leq L,\\
\frac{\alpha}{\Delta}A_{k\ell}\frac{h_2(y_{k-L},y_{\ell-L})}{D_{y_{\ell-L}}\redu_{\ell-L}}&k\neq \ell > L,\\
1+\frac{\partial_{y_{k-L}}\xi_{k-L,\epsilon}}{D_{y_{\ell-L}}\redu_{\ell-L}}&k= \ell>L.
\end{array}\right.
\end{equation}
and 
$$
\frac{|D_{z}F_\epsilon|}{|D_{\bar z}F_\epsilon|}=\frac{\prod_{j=1}^MD_{y_j}\redu_j}{\prod_{j=1}^MD_{\bar y_j}\redu_j}\cdot\frac{|\mc D(z)|}{|\mc D(\bar z)|}.
$$
For the first ratio: 
\begin{align}
\prod_{j=1}^M\frac{D_{y_j}\redu_j}{D_{\bar y_j}\redu_j}&=\prod_{j=1}^M\left(1+\frac{D_{y_j}\redu_j-D_{\bar y_j}\redu_j}{D_{\bar y_j}\redu_j}\right)\leq\prod_{j=1}^M\left(1+\mc O(1)|y_j-\bar y_j|\right)\leq\exp\left[\mc O(M)d_\infty(y,\bar y)\right].\label{Eq:Ineqtildef}
\end{align}

To estimate the ratio $\frac{|\mc D(z)|}{|\mc D(\bar z)|}$ we  will apply Proposition \ref{Prop:EstTool} in  Appendix~\ref{Ap:TechComp}. To this end define the matrix
\[
B(z):=\mc D(z)-\Id.
\]
First of all we will prove that for every $1\leq p< \infty$ and $z\in\mb T^{N}$, $B(z)$ has operator norm bounded by
\begin{equation}\label{Eq:bOperatorNormBound}
\|B(z)\|_p\leq\max\{\mc O(\Delta^{-1}M), C_\#\epsilon\}+\mc O(\Delta^{-1/p}M^{1/p})+\mc O(\Delta^{-1}N^{1/p}\degree^{1/q})
\end{equation}
where $C_\#$ is a constant uniform on the parameters of the network and $1/p+1/q=1$. Indeed,  consider $\binom{u}{v}\in\R^{L+M}$ and $\binom{u'}{v'}:=B(z)\binom{u}{v}$. Then

\begin{align*}
u'_i&=\left[\frac{\alpha}{\Delta}\sum_{n=1}^LA_{in}\frac{h_1(x_i,z_n)}{\sigma}\right]u_i+\frac{\alpha}{\Delta}\sum _{\ell=1}^LA^{ll}_{i\ell}\frac{h_1(x_i,x_\ell)}{\sigma}u_\ell+\frac{\alpha}{\Delta}\sum_{m=1}^MA^{lh}_{im}\frac{h_2(x_i,y_m)}{D_{y_m}\redu_m}v_{m}\\
v'_j&=\sum_{\ell=1}^{L}\frac{\partial_{x_\ell}\xi_{j,\epsilon}}{\sigma}u_\ell+\frac{\alpha}{\Delta}\sum_{m=1}^MA^{hh}_{jm}\frac{h_2(y_j,y_m)}{D_{y_m}\redu_m}v_m+ \frac{\partial_{y_{j}}\xi_{j,\epsilon}}{D_{y_{j}}\redu_{j}}v_{j}.
\end{align*}

Using estimates analogous the ones used in the proof of Proposition \ref{Prop:ExpAuxSys}
\begin{align*}
\|u'\|_{p,\R^L}&\leq \mc O(\Delta^{-1}\degree)\|u\|_p+\mc O(\Delta^{-1})\max_i\left[\sum_{\ell=1}^LA^{ll}_{i\ell}|u_\ell|+\sum_{m=1}^MA^{lh}_{im}|v_m|\right]\\
&\leq\mc O(\Delta^{-1}\degree)\|u\|_p+\mc O(\Delta^{-1}N^{1/p}\degree^{1/q})\|(u,v)\|_p\\
\|v'\|_{p,\R^M}&\leq C_\#\epsilon\|v\|_p+\mc O(\Delta^{-1}N^{1/p})\max_i\left[\sum_{\ell}A^{ll}_{i\ell}|u_\ell|+\sum_mA^{lh}_{im}|v_m|\right]\\
&\leq C_\#\epsilon\|v\|_p+\mc O(\Delta^{-1/p}M^{1/p})\|(u,v)\|_p
\end{align*}
so using conditions \eqref{Eq:ThmCond1}, \eqref{Eq:ThmCond2'}, we obtain \eqref{Eq:bOperatorNormBound}:
\begin{align*}
\frac{\|(u',v')\|_p}{\|(u,v)\|_p}&\leq\max\{\mc O(\Delta^{-1}\degree), C_\#\epsilon\}+\mc O(\Delta^{-1/p}M^{1/p})+\mc O(\Delta^{-1}L^{1/p}\degree^{1/q})\\
&\leq C_\#\epsilon+\mc O(\eta).
\end{align*}
Taking $C_\# \epsilon<1$ and  $\eta>0$ sufficiently small, ensures that $\|B(z)\|_p\leq\lambda<1$ for all $z\in\mb T^{N}$.
Now we want to estimate the norm $\|\cdot\|_p$ of columns of $B-\bar B$ where 
\[
B:=B(z)\quad\mbox{and}\quad\bar B:=B(\bar z).
\]
For $1\leq i\leq L$, looking at the entries of $\mc D(z)$, \eqref{Eq:DExpression}, it is clear that the non-vanishing entries $[B(z)]_{ik}$ for $k\neq i$ are Lipschitz functions with Lipschitz constants of the order $\mc O(\Delta^{-1})$:
$$
|B_{ik}-\bar B_{ik}|\leq A_{ik}\mc O(\Delta^{-1})d_\infty(z,\bar z)
$$
Instead, for $k=i$,
\begin{align*}
\left|B_{ii}-\bar B_{ii}\right|&=\frac{\alpha}{\Delta}\left|\sum_\ell A^{ll}_{in}(h_1(x_i,x_\ell)-h_1(\bar x_i,\bar x_\ell))+\sum_m A^{lh}_{im} (h_1(x_i,x_m)-h_1(\bar x_i,\bar x_m))\right|\\
&\leq \frac{\alpha}{\Delta}\sum_\ell A^{ll}_{i\ell}|h_1(x_i,x_\ell)-h_1(\bar x_i,\bar x_\ell)|+\frac{\alpha}{\Delta}\sum_mA^{lh}_{im}|h_1(x_i,y_m)-h_1(\bar x_i,\bar y_m)|\\
&\leq\mc O(\Delta^{-1}\degree)d_\infty(z,\bar z).
\end{align*}
which implies
\begin{align*}
\|\Col^i[B-\bar B]\|_p&= \left(\sum_{k\in[L]}|B_{ik}-\bar B_{ik}|^p\right)^{\frac{1}{p}}+\left(\sum_{k\in[L+1,N]}|B_{ik}-\bar B_{ik}|^p\right)^{\frac{1}{p}}\\
&\leq \left(\sum_{k\in [L]\backslash\{i\}}|B_{ik}-\bar B_{ik}|^p\right)^{\frac{1}{p}}+\left(\sum_{k\in[L+1,N]}|B_{ik}-\bar B_{ik}|^p\right)^{\frac{1}{p}}+\mc O(\Delta^{-1}\degree)d_\infty(z,\bar z)\\
&\leq 2\left(\sum_{k\neq i}A_{ik}\right)^{\frac{1}{p}}\mc O(\Delta^{-1})d_\infty(z,\bar z)+\mc O(\Delta^{-1}\degree)d_\infty(z,\bar z)\\
&\leq \mc O(\Delta^{-1}\degree)d_\infty(z,\bar z).
\end{align*}
For $1\leq j\leq M$, looking again at \eqref{Eq:DExpression} the non-vanishing entries of $[B(z)]_{(j+L)k}$ for $k\neq j+N$ are Lipschitz functions with Lipschitz constants of the order $\mc O(\Delta^{-1})$, while $[B(z)]_{(j+L)(j+L)}$ has Lipschitz constant of order $\mc O(1)$, thus
\begin{align*}
\|\Col^{j+L}[B-\bar B]\|_p&=\left(\sum_{k\in[L]}|B_{(j+L)k}-\bar B_{(j+L)k}|^p\right)^{\frac{1}{p}}+\left(\sum_{k\in[L+1,N]}|B_{(j+L)k}-\bar B_{(j+L)k}|^p\right)^{\frac{1}{p}}\\
&\leq \left(\sum_{k\in[L]}|B_{(j+L)k}-\bar B_{(j+L)k}|^p\right)^{\frac{1}{p}}+\left(\sum_{k\in[L+1,N]\backslash \{j+L\}}|B_{(j+L)k}-\bar B_{(j+L)k}|^p\right)^{\frac{1}{p}}+\\
& \quad \quad \quad \quad + \mc O(1)d_\infty(z,\bar z)\\
&\leq 2\left(\sum_{k\neq j+L}A_{(j+L)k}\right)^{\frac{1}{p}}\mc O(\Delta^{-1})d_\infty(z,\bar z)+\mc O(1)d_\infty(z,\bar z)\\
&\leq \mc O(1)d_\infty(z,\bar z).
\end{align*}

Proposition \ref{Prop:EstTool} from  Appendix~\ref{Ap:TechComp} now implies that 
\begin{equation}\label{Eq:RatDestimate}
\frac{|\mc D(z)|}{|\mc D(\bar z)|}\leq\exp\left\{\sum_{k=1}^{N}\|\Col^k[B-\bar B]\|_p\right\}\leq \exp\left\{(\mc O(\Delta^{-1}\degree L)+\mc O(M))d_\infty(z,\bar z)\right\}.
\end{equation}
\end{proof}

\subsection{Invariant Cone of Functions}

Define the cone of functions
\[
C_{a,p}:=\left\{\phi:\mb T^{N}\rightarrow \R^+:\quad \frac{\phi(z)}{\phi(\bar z)}\leq\exp[ad_p(z,\bar z)],\mbox{ }\forall z,\bar z\in\mb T^{N}\right\}.
\]
This is convex and has finite diameter (see for example \cite{MR0087058, MR0336473} or \cite{ViaSdds}). 
We now use the result on distortion from the previous section to determine the parameters $a>0$ such that $C_{a,p}$ is invariant under the action of the transfer operator $P_\epsilon$. Since $C_{a,p}$ has finite diameter with respect to the Hilbert metric on the cone, see  \cite{ViaSdds},  $P_\epsilon$ is a contraction restricted to this set and its unique fixed point is the only invariant density which thus belongs to $C_{a,p}$. In the next subsection, 
we will use this observation to conclude the proof of Theorem \ref{Thm:Main} in the expanding case.

\begin{proposition}\label{Prop:ConInv}
Under conditions \eqref{Eq:ThmCond1}-\eqref{Eq:ThmCond2}, for every $a>a_c$, where $a_c$ is of the form 
\begin{equation}\label{Eq:ConditionOna}
a_c=\frac{\mc O(\Delta^{-1}\degree L)+\mc O(M)}{1-\expansion},
\end{equation}
$C_{a,p}$ is invariant under the action of the transfer operator $P_\epsilon$ of $F_\epsilon$, i.e. $P_\epsilon(C_{a,p})\subset C_{a,p}$.
\end{proposition}
\begin{proof}
Since $F_\epsilon$ is a local expanding diffeomorphism, its transfer operator, $P_\epsilon$, has expression
$$
(P_\epsilon\phi)(z)=\sum_{i}\phi(F_{\epsilon,i}^{-1}(z))\left|D_{F_{\epsilon,i}^{-1}(z)}F_\epsilon\right|^{-1}
$$
where $\{F_{\epsilon,i}\}_i$ are surjective invertible branches of $F_\epsilon$. Suppose $\phi\in C_{a,p}$. Then
\begin{align*}
\frac{\phi(F_{\epsilon,i}^{-1}(z))}{\phi(F_{\epsilon,i}^{-1}(\bar z))}\frac{\left|D_{F_{\epsilon,i}^{-1}(\bar z)}F_\epsilon\right| }{\left|D_{F_{\epsilon,i}^{-1}(z)}F_\epsilon\right|}&\leq \exp\left\{ad_p(F_{\epsilon,i}^{-1}(z),F_{\epsilon,i}^{-1}(\bar z))\right\}\exp\left\{\left[\mc O(\Delta^{-1}\degree L)+\mc O(M)\right] d_\infty(F_{\epsilon,i}^{-1}(z),F^{-1}_{\epsilon,i}(\bar z))\right\} \\
&\leq\exp\left\{\left[a+\mc O(\Delta^{-1}\degree L)+\mc O(M)\right] d_p(F_{\epsilon,i}^{-1}(z),F^{-1}_{\epsilon,i}(\bar z))\right\}  \\
&\leq\exp\left\{\left[\expansion^{-1}a+\mc O(\Delta^{-1}\degree L)+\mc O(M)\right] d_p(z,\bar z)\right\}.
\end{align*}
Here we used that $d_\infty(z,\bar z)\leq d_p(z,\bar z)$ for every $1\leq p<\infty$
Hence
\begin{align*}
\frac{(P_\epsilon\phi)(z)}{(P_\epsilon\phi)(\bar z)}&=\frac{\sum_{i}\phi(F_{\epsilon,i}^{-1}(z))|D_{F_{\epsilon,i}^{-1}(z)}F_\epsilon|^{-1}}{\sum_{i}\phi(F_{\epsilon,i}^{-1}(\bar z))|D_{F_{\epsilon,i}^{-1}(w)}F_\epsilon|^{-1}}\\
&\leq\exp\left[\left(\expansion ^{-1}a+\mc O(\Delta^{-1}\degree L)+\mc O(M)\right)d_p(z,\bar z)\right].
\end{align*}
It follows that if $a>a_c$ then 
$C_{a,p}$ is invariant under $P_\epsilon$.
\end{proof}
\begin{proof}[Proof of Theorem~\ref{Thm:InvMeasDenExpand}]
The existence of the absolutely continuous invariant probability measure is standard from the expansivity of $F_\epsilon$. The regularity bound on the density immediately follows from Proposition~\ref{Prop:ConInv} and from the observation (that can be found in \cite{ViaSdds}) that the cone $\mc C_{a,p}$ has finite dimeter with respect to the projective Hilbert metric. This in particularly means that $P_\epsilon$ is a contraction with respect to this metric and has a fixed point.
\end{proof}
\subsection{Proof of Theorem \ref{Thm:Main} in the Expanding Case}\label{Sec:ProofTHM1} 
Property \eqref{Eq:InvDensProp} of the invariant density provides an upper bound for its supremum which depends on the parameters of the network and proves the statement of Theorem~\ref{Thm:Main} in the expanding case.  
\begin{proof}[Proof of Theorem \ref{Thm:Main}]
Since under conditions \eqref{Eq:ThmCond1}-\eqref{Eq:ThmCond2} in Theorem \ref{Thm:Main},
Proposition \ref{Prop:ConInv} holds, the invariant density for $F_\epsilon$ belongs to the cone $C_{a,p}$, $\rho\in C_{a,p}$, for $a>a_c$. Since $\rho$ is a continuous density, it has to take value one at some point in its domain. This together with the regularity condition given by the cone implies that
$$
\sup_{z\in\mb T^{N}}\rho(z)\leq \exp\left\{{\mc O(\Delta^{-1}\degree L^{1+1/p})+\mc O(M L^{1/p})}\right\}.
$$
Using the upper bound \eqref{Eq:UppBndBLeb},
\begin{align*}
\nu(\mc B_\epsilon\times \mb T^M)&=\int_{\mc B_\epsilon\times \mb T^M}\rho(z)d\leb_{N}(z)\\
&\leq\leb_{N}(\mc B_\epsilon\times \mb T^M)\sup_z\rho(z)\\
&\leq\exp\left\{-\Delta\epsilon^2/2+\mc O(\Delta^{-1}\degree L^{1+1/p})+\mc O(ML^{1/p})\right\}.
\end{align*}
From the invariance of $\rho$ and thus of $\nu$, for any $T\in\N$
\[
\nu\left(\bigcup_{t=0}^{T} F_\epsilon^{-t}(\mc B_\epsilon\times \mb T^M)\right)\leq (T+1)\nu(\mc B_\epsilon\times\mb T^M)\leq (T+1)\exp\left\{-\Delta\epsilon^2/2+\mc O(\Delta^{-1}\degree L^{1+1/p})+\mc O(ML^{1/p})\right\}.
\]
Using again that $\rho\in C_{a,p}$, and  \eqref{Eq:ThmCond1} and \eqref{Eq:ThmCond2}, 
\begin{align*}
\leb_N\left(\bigcup_{t=0}^{T}F_\epsilon ^{-t}(\mc B_\epsilon\times \mb T^M)\right)&= \int_{\bigcup_{t=0}^{T}F_\epsilon^{-t}(\mc B_\epsilon\times \mb T^M)}\rho^{-1}d\nu\\
&\leq \nu\left(\bigcup_{t=0}^{T}F_\epsilon^{-t}(\mc B_\epsilon\times \mb T^M)\right)\exp\left\{\mc O(\Delta^{-1}\degree L^{1+1/p})+\mc O(ML^{1/p})\right\}\\
&\leq (T+1)\exp\left\{-\Delta\epsilon^2/2+\mc O(\Delta^{-1}\degree L^{1+1/p})+\mc O(ML^{1/p})\right\}\\
& \leq (T+1) \exp \left\{  -\Delta\epsilon^2/2 +  \mc  O(\eta) \Delta \right\} .
\end{align*}
%
Where we used \eqref{Eq:ThmCond3} to obtain the last inequality.
Hence, the set  
\[
\Omega_T=\mb T^{N}\backslash\bigcup_{t=0}^{T} F_\epsilon^{-t}(\mc B_\epsilon\times \mb T^M)
\]
for $\eta>0$ sufficiently small satisfies the assertion of the theorem.
\end{proof}

\section{Proof of Theorem~\ref{Thm:Main} when some Reduced Maps have Hyperbolic Attractors}\label{Sec:RedMapNonAtt}

In this section, we allow for the situation where some (or possibly all)  reduced maps have periodic attractors. For this reason, 
we introduce the new structural parameter $M_u\in\N_0$ such that, after renaming the hub nodes, the reduced dynamics $\redu_j$ is expanding for $1\leq j \leq M_u$, 
 while for $M_u<j\leq M$, $\redu_j$ has a hyperbolic periodic attractor $\Lambda_j$. 
Let us also define $M_s=M-M_u$. We also assume that $g_j$ are $(n,m,\lambda,r)$-hyperbolic with $n=1$. We will show how to drop this assumption in Lemma~\ref{lem:n=1}.


\textcolor{blue}{}
As in the previous section, the goal is to prove the existence of a set of large measure whose points take a long time to enter the set $\mc B_\epsilon$ where fluctuations are above the threshold. To achieve this, we study the ergodic properties of $F_\epsilon$ restricted to a certain forward invariant set $\mc S$ and prove that the statement of Theorem \ref{Thm:Main} holds true for initial conditions taken in this set. Then in Section \ref{Sec:AttFullStatProof} we extend the reasoning to the remainder of the phase space and prove the full statement of the theorem.  

For simplicity we will sometimes write $(z_u,z_s)$ for a point in  $\mb T^{L+M_u}\times \mb T^{M_s}=\mb T^N$ and $z_u=(x,y_u)\in \mb T^{L}\times\mb T^{M_u}$. 
Let 
\[
\pi_u \colon \mb T^N \to \mb T^{L+M_u}\mbox{ and }\Pi_u:\R^{N}\rightarrow\R^{L+M_u}
\] 
be respectively the (canonical) projection on the first $L+M_u$ coordinates and its differential.

We begin by pointing out the existence of the invariant set.
\begin{lemma}\label{Lem:InvSetStrip} 
As before, for $j\in\{M_u+1,...,M\}$, let $\Lambda_j$ be the attracting sets of $\redu_j$ and $\Upsilon=\mb T\setminus W_s(\Lambda_j)$. 
There exist  $\lambda\in (0,1)$, $\neig>0$, $r_0>0$  so that for each 
$j\in\{M_u+1,...,M\}$ and each $|r|<r_0$,

(i) $|D\redu_j(y)|<\lambda<1$ for every $y\in U_j$ and $\redu_j(x)+r \in U_j,\quad\forall x\in U_j$, where $U_j$ is the $\neig$-neighborhood of $\Lambda_j$.

(ii)  $|D\redu_j|>\lambda^{-1}$ on the $\neig$-neighborhood of $\Upsilon_j$,  $\forall j\in[M_u+1,M]$.
%
\end{lemma}
\begin{proof}
The first assertion in (i) and (ii) follow from continuity of $Dg_j$. Fix $x\in U_j$ and $r\in(-r_0,r_0)$. From the definition of $U_j$, there exists $y\in\Lambda_j$ such that $d(x,y)<\neig$. From the contraction property $d(\redu_j(x),\redu_j(y))<\lambda d(x,y)<\lambda\neig$ and choosing $r_0<(1-\lambda)\neig$,
\[
d(\redu_j(x)+r,\redu_j(y))<\lambda \neig+r_0<\neig.
\]
From the invariance of $\Lambda_j$, $\redu_j(y)\in\Lambda_j$, the lemma follows. 
\end{proof}

Let 
\begin{equation}\label{Eq:DefRProdFixINt}
\mc R:=U_{M_u+1}\times...\times U_M\quad\mbox{and}\quad \mc S:=\mb T^{L+M_u}\times \mc R \subset \mb T^N . 
\end{equation}
 Lemma \ref{Lem:InvSetStrip} implies that provided the $\epsilon$ from the truncated  system is below  $r_0/2$, the set $\mc S$ is forward invariant under $F_\epsilon$.
It follows that for each attracting periodic orbit  $O(z_s)$ of $g_{M_u+1}\times \dots \times g_{M}\colon \mb T^{M_s}\to \mb T^{M_s}$,   the endomorphism $F_\epsilon$ has a fat solenoidal invariant set. 
Indeed, take the union $U$ of the connected components of $\mc R$ containing $O(z_s)$.  Then by the previous lemma, 
$F_\epsilon(\mb T^{L+M_u}\times U)\subset \mb T^{L+M_u} \times U$. The set $\cap_{n\ge 0} F^n_\epsilon(\mb T^{L+M_u} \times U)$ is the analogue
of the usual solenoid but with self-intersections, see  Figure~\ref{Fig:Attractor}. An analogous situation, but where the map is a skew product  
is studied in \cite{MR1862809}.  The set  $\bigcap_{n\ge 0} F^n_\epsilon(\mb T^{L+M_u} \times U)$ will support an invariant measure: 

\begin{theorem}\label{Thm:PhysMeasFep}
Under conditions \eqref{Eq:ThmCond1}-\eqref{Eq:ThmCond3} of Theorem \ref{Thm:Main} with $\eta>0$ sufficiently small
\begin{itemize}
\item for every attracting periodic orbit of  $g_{M_u+1}\times \dots \times g_{M}$, $F_\epsilon$ has an ergodic physical measure, 
\item for each such measure $\nu$, 
the marginal $(\pi_u)_* \nu$ on $\mb T^{L+M_u}\times \{0\}$ has a density $\rho$ satisfying $\forall z_u,\bar z_u\in \mb T^{L+M_u}\times\{0\}$, 
\[
\frac{\rho(z_u)}{\rho(\bar z_u)}\leq \exp\left\{a d_p(z_u,\bar z_u)\right\}, \quad\quad a={\mc O(\Delta^{-1}L^{1+1/p}\degree^{1/q})+\mc O(M)},
\]
\item these are the only physical measures for $F_\epsilon$.
\end{itemize}
\end{theorem}

\begin{figure}
\centering
\begin{subfigure}{0.45\textwidth}
\includegraphics[scale=0.45]{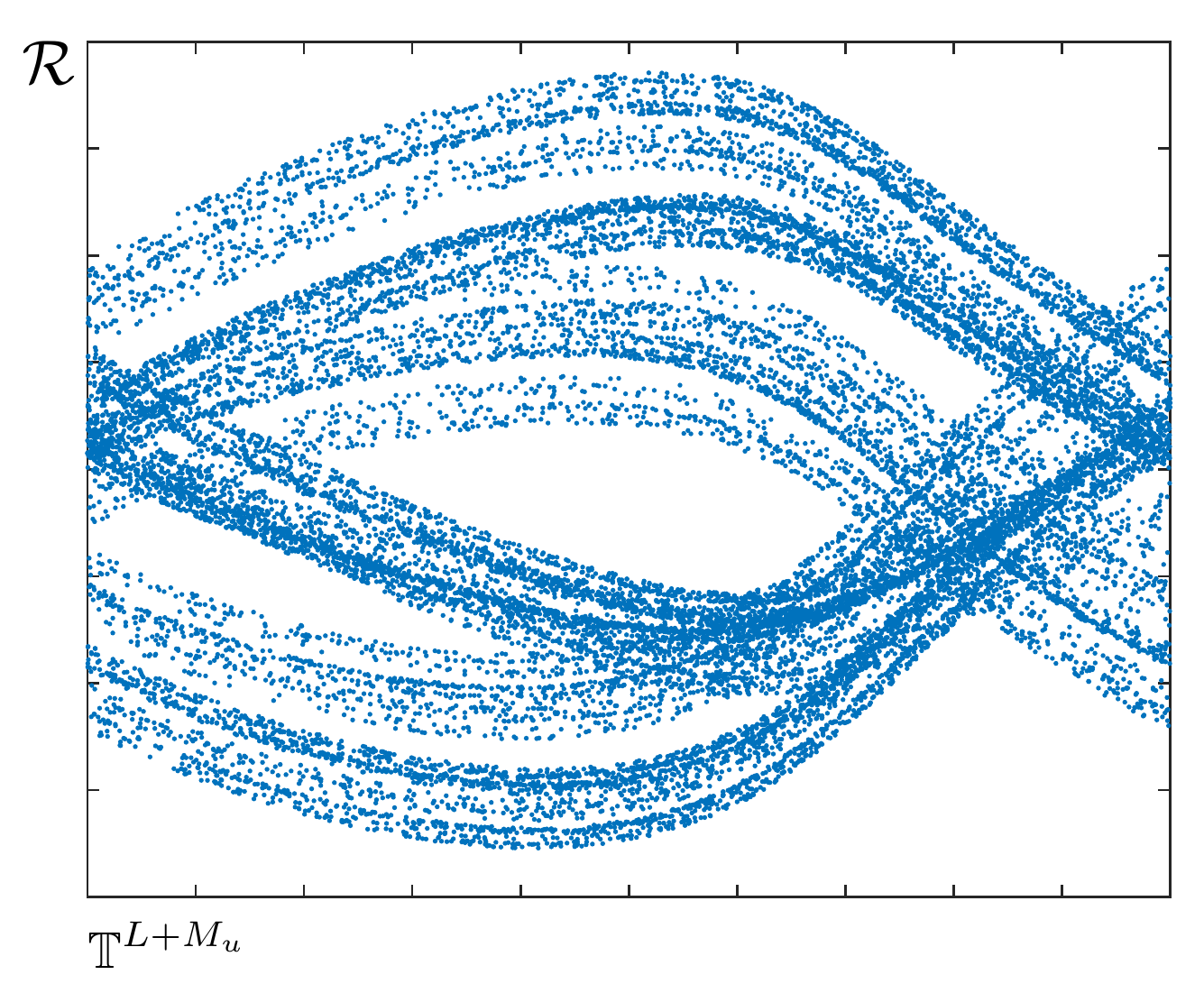}
\caption{}
\end{subfigure}
\begin{subfigure}{0.45\textwidth}
\includegraphics[scale=0.35]{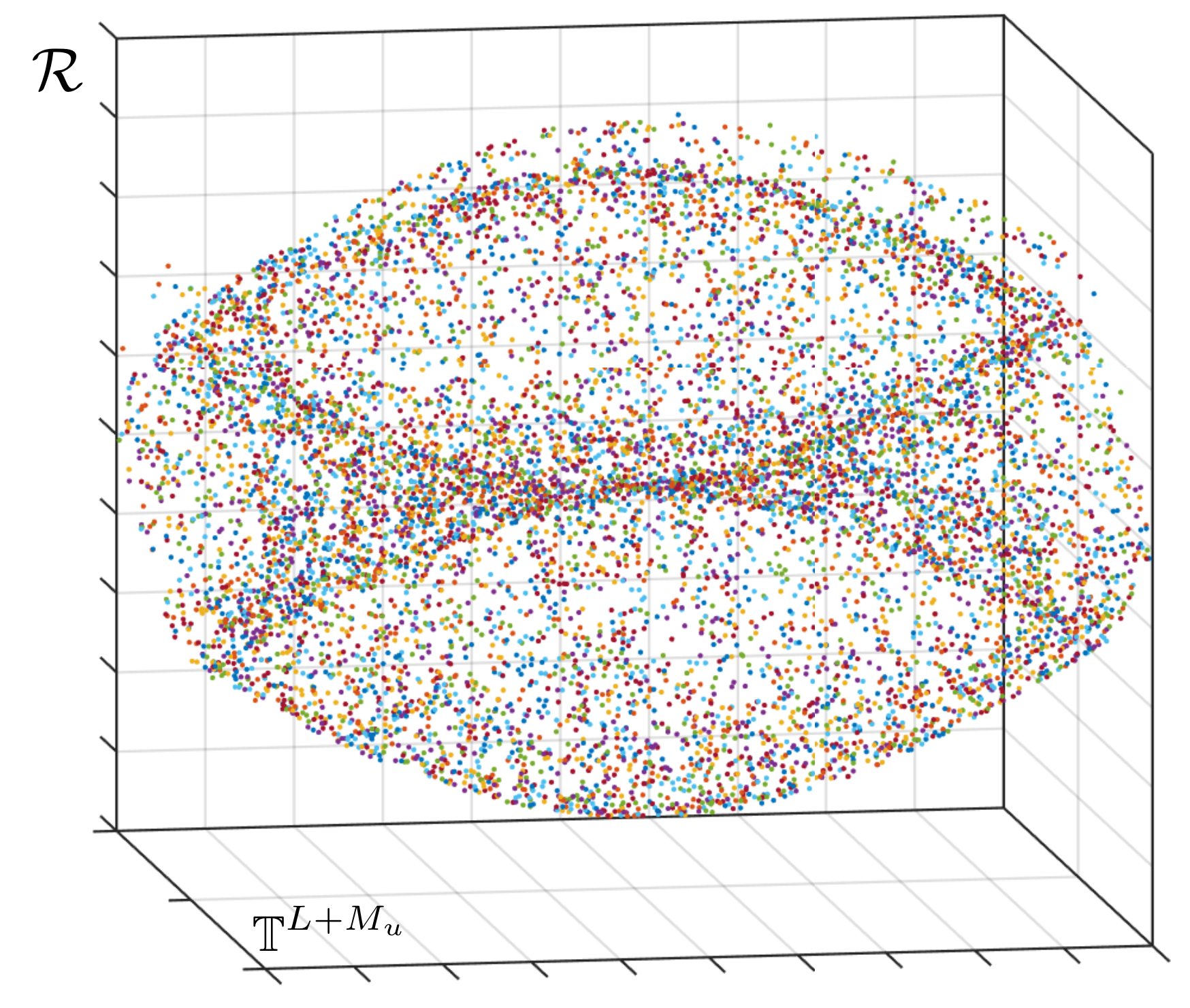}
\caption{}
\end{subfigure}
\caption{Approximate 2D and $3$D representations of one component of the attractor of $F_\epsilon$.}\label{Fig:Attractor}
\end{figure}

This theorem will be proved in Subsection~\ref{subsec:invariantcones}. 

\subsection{Strategy of the proof of Theorem~\ref{Thm:Main} in Presence of Hyperbolic Attractors}\label{Sec:InvStrip}

For the time being, 
we restrict our attention to the case where the threshold of the fluctuations is below $r_0$ as defined in Lemma \ref{Lem:InvSetStrip} and consider
the map ${F_\epsilon}|_{\mc S}:\mc S\rightarrow \mc S$ that we will still call ${F_\epsilon}$ with an abuse of notation. The expression for ${F_\epsilon}$ is the same as in equations \eqref{Eq:CoupDyn1'} and \eqref{Eq:CoupDyn2'}, but now the local phase space for the hubs with a non-empty attractor, $\{L+M_u+1,\dots,L+M=N\}$, is restricted to the open set $\mc R$.


The proof of Theorem~\ref{Thm:Main} will follow from the following proposition.

\begin{proposition}\label{Prop:BadSetMeasAtt} 
For every $s_1\in\mb Z$ and $j\in[M]$ 
\[
\mc B^{(s_1,j)}_{\epsilon, T}:=\bigcup_{t=0}^{T}F_\epsilon^{-t}\left(\mc B^{(s_1,j)}_\epsilon\times \mb T^{M_u}\times \mc R\right)\cap \mc S \subset \mb T^N
\]
is bounded as
\begin{equation}\label{Eq:BadSetBound}
\leb_N\left(\mc B^{(s_1,j)}_{\epsilon,T}\right)\leq T\exp\left[-C \Delta \epsilon^2+\mc O(\Delta^{-1}L^{1+2/p}\degree^{1/q})+\mc O(ML^{1/p})\right].
\end{equation}
\end{proposition}

To prove the above result, we first build families of stable and unstable invariant cones for $F_\epsilon$ in the tangent bundle of $\mc S$ (Proposition \ref{Prop:InvConesForTildeF}) which correspond to contracting and expanding directions for the dynamics, thus proving hyperbolic behaviour of the map. In Section \ref{Sec:AdmManifolds} we define a class of manifolds tangent to the unstable cones whose regularity properties are kept invariant under the dynamics, and we study the evolution of densities supported on them under action of $F_\epsilon$.  Bounding the Jacobian of the map restricted to the manifolds (Proposition \ref{Prop:JacobEst}) one can prove the existence of an invariant cone of densities (Proposition \ref{Prop:DensEvSubm}) which gives the desired regularity properties for the measures. Since the product structure of $\mc B^{(s_1,j)}_\epsilon\times \mb T^{M_u}\times \mc R$ is not preserved under pre-images of $F_\epsilon^{t}$, we approximate it with the  set which is the union of global stable manifolds (Lemma~\ref{Lem:IncSetUnMan}). This last property is preserved taking pre-images. The bound in \eqref{Eq:BadSetBound} will then be a consequence of estimates on the distortion of the holonomy map along stable leaves of $F_\epsilon$ (Proposition \ref{Prop:JacEstBnd}).  

\subsection{Invariant Cone Fields for $F_\epsilon$}\label{Sec:InvCones}
\begin{proposition}\label{Prop:InvConesForTildeF}
There exists $\eta_0>0$ such that if conditions \eqref{Eq:ThmCond1}-\eqref{Eq:ThmCond3} are satisfied with $\eta<\eta_0$, then there exists $C_\#>0$ such that for every $\epsilon>0$ 
\begin{equation}
\epsilon<\min\left\{\frac{1-\lambda}{C_\#},\frac{\lambda^{-1}-1}{C_\#},\epsilon_0\right\}\label{AttCond3}
\end{equation}
\begin{itemize}
\item[(i)]the constant cone fields
\begin{equation}\label{Eq:UnstConeCond}
\mc C_p^u:=\left\{(u,w,v)\in\R^{L+M_u+M_s}\backslash{\{0\}}:\quad\frac{\|v\|_{p,M_s}}{\|u\|_{p,L}+\|w\|_{p,M_u}}<\beta_{u,p}\right\}
\end{equation}
and
\begin{equation}\label{Eq:StabCon}
\mc C_p^s:=\left\{(u,v,w)\in\R^{L+M_u+M_s}\backslash{\{0\}}:\quad
\frac{\|v\|_{p,M_s}}{\|u\|_{p,L}+\|w\|_{p,M_u}} > \frac{1}{\beta_{s,p}}
\right\}
\end{equation}
with 
\begin{align*}
\beta_{u,p}:=\mc O(\Delta^{-1/p}M_s^{1/p}),\quad \beta_{s,p}:=\max\{\mc O(\Delta^{-1}L^{1/p}\degree^{1/q}), \mc O(\Delta^{-1/p}M_u^{1/p})\}
\end{align*}
satisfy $\forall z\in\mb T^{N}$ $D_{z}{F_\epsilon}(\mc C^u)\subset\mc C^u$ and $D_{z}{F_\epsilon}^{-1}(\mc C^s)\subset\mc C^s$.

\item[(ii)] there exists $\bar \sigma$ and $\bar \lambda$ such that, for every $z\in\mb T^{N}$
\begin{align}
\frac{\|D_{z}{F_\epsilon}(u,w,v)\|_p }{\|(u,w,v)\|_p}&\geq\bar\sigma>1,&\forall (u,w,v)\in\mc C_p^u\label{Eq:ExpResUnstCone}\\
\frac{\|D_{z}{F_\epsilon}(u,w,v)\|_p }{\|(u,w,v)\|_p}&\leq\bar \lambda<1, &\forall (u,w,v)\in\mc C_p^s.\label{Eq:ExpResStabCone}
\end{align}
\end{itemize}
\end{proposition}
\begin{remark}
We have constructed the map $F_\epsilon$ in such a way that, when the network is $\eta-$heterogeneous with $\eta$ very small, it results to be \say{close} to the product of uncoupled factors equal to $f$, for the coordinates corresponding to low degree nodes, and equal to $\redu_j$, for the coordinates of the hubs. This is reflected by the width of the invariant cones which can be chosen to be very small for $\eta$ tending to zero, so that $\mc C_p^u$ and $\mc C_p^s$ are very narrow around their respective axis $\R^{L+M_u}\oplus \{0\}$ and $\{0\}\oplus\R^{M_s}$. 
	\end{remark}
\begin{corollary}\label{cor2}  Under the assumptions of the  previous proposition, 
$\pi_u \circ {F_\epsilon}^n \colon \mb T^{L+M_u} \times \{0\} \to \mb T^{L+M_u}$ is a covering map of degree  $\sigma^{n(L+M_u)}$ where $\sigma$
is the degree of the local map. 
\end{corollary}
\begin{proof} This follows from the previous proposition, because $\mb T^{L+M_u}\times \{0\}$ is tangent to the unstable cone, and thus $\pi_u \circ {F_\epsilon}^n$ is a local diffeomorphism between compact manifolds. This implies that every point of $\mb T^{L+M_u}$ has the same number of preimages, and this number equals the degree of the map. Then observe that there is a homotopy bringing $\pi_u\circ{F_\epsilon}$ to the $(L+M_u)$-fold uncoupled product of identical copies of the map $f^n$. The homotopy is obtained by continuously deforming the map letting the coupling strength $\alpha$ go to zero. Since the degree is an homotopy invariant and $\pi_u\circ{F_\epsilon}$ is homotopic to the $(L+M_u)$-fold uncoupled product of identical copies of the map $f^n$,
\[
\deg \pi_u\circ F_\epsilon= \deg \underbrace{f^n\times...\times f^n}_{L+M_u\mbox{ times}}=\sigma^{n(L+\M_u)}.
\] 
 
\end{proof}
\begin{proof}
(i) The expression for the differential of the map ${F_\epsilon}$ is the same as in \eqref{Eq:DiffAuxMap}.
Take $(u,w,v)\in\R^{L}\times\R^{M_u}\times\R^{M_s}$, and suppose $(u',w',v')^{t}:=D_z{{F_\epsilon}}(u,v,w)^{t}$. Then 

\begin{align*}
&(u')_i=\left[f'(x_i)+\frac{\alpha}{\Delta}\sum_{m=1}^MA^{lh}_{im}h_1+\frac{\alpha}{\Delta}\sum_{\ell=1}^LA^{ll}_{i\ell}h_1\right]u_i+\frac{\alpha}{\Delta}\sum _{\ell=1}^LA^{ll}_{i\ell}h_1u_\ell+\\
& \quad \quad  \quad + \frac{\alpha}{\Delta}\sum_{m=1}^{M_u}A^{lh}_{im}h_2w_{m}+\frac{\alpha}{\Delta}\sum_{m=M_u+1}^{M}A^{lh}_{im}h_2v_{m}& 1\leq i\leq L\\
&(w')_j=\sum_{\ell=1}^L\partial_{x_\ell}\xi_{j,\epsilon}u_\ell+\frac{\alpha}{\Delta}\sum_{m=1}^{M_u}A^{hh}_{jm}h_2w_m+\frac{\alpha}{\Delta}\sum_{m=M_u+1}^{M}A^{hh}_{jm}h_2v^m+\\
& \quad \quad \quad + \left[\partial_{y_j}\xi_{j,\epsilon} +\frac{\alpha}{\Delta}\sum_{m=1}^MA^{hh}_{jm}h_2\right]w_{j}&1\leq j\leq M_u\\
&(v')_j=\sum_{\ell=1}^L\partial_{x_\ell}\xi_{j,\epsilon}u_\ell+\frac{\alpha}{\Delta}\sum_{m=1}^{M_u}A^{hh}_{jm}h_2w_m+\frac{\alpha}{\Delta}\sum_{m=M_u+1}^{M}A^{hh}_{jm}h_2v_m+\\
&\quad \quad \quad + \left[\partial_{y_j}\xi_{j,\epsilon}+\frac{\alpha}{\Delta}\sum_{m=1}^MA^{hh}_{jm}h_2\right]v_{j}&M_u< j\leq M
\end{align*}

where we suppressed all dependences of those functions for which we use a uniform bound. 
\begin{align*}
\|u'\|_{p,\R^L}&\geq \left(\sigma-\mc O(\Delta^{-1}\degree)\right)\|u\|_{p,\R^L}-\\
&\quad\quad-\mc O(\Delta^{-1}L^{1/p})\max_{i\in[L]}\left[\sum_{\ell=1}^LA^{ll}_{i\ell}|u_\ell|+\sum_{m=1}^{M_u}A^{lh}_{im}|w_m|+\sum_{m=M_u+1}^{M}A^{lh}_{im}|v_m|\right]\\
&\geq \left(\sigma-\mc O(\Delta^{-1}\degree)\right)\|u\|_{p,\R^L}-\mc O(\Delta^{-1}L^{1/p}\degree^{1/q})(\|u\|_{p,\R^L}+\|w\|_{p,\R^{M_u}}+\|v\|_{p,\R^{M_s}})\\
\\
\|w'\|_{p,\R^{M_u}}&\geq \left(\lambda^{-1}-C_\#\epsilon-\mc O(\Delta^{-1}M)\right)\|w\|_{p,\R^{M_u}}-\\
&\quad\quad-\mc O(\Delta^{-1}M_u^{1/p})\max_{1\leq j\leq M_u}\left[\sum_{\ell=1}^L A^{hl}_{j\ell}|u_\ell|+\sum_{m=1}^{M_u} A^{hh}_{jm}|w_m|+\sum_{m=M_u+1}^{M} A^{hh}_{jm}|v_m|\right]\\
&\geq \left(\lambda^{-1}-C_\#\delta-\mc O(\Delta^{-1}M)\right)\|w\|_{p,\R^{M_u}}-\mc O(\Delta^{-1/p}M_u^{1/p})(\|u\|_{p,\R^L}+\|w\|_{p,\R^{M_u}}+\|v\|_{p,\R^{M_s}})
\end{align*}
and analogously
\begin{align*}
\|v'\|_{p,\R^{M_s}}&\leq (\lambda+C_\#\epsilon+\mc O(\Delta^{-1}M))\|v\|_{p,\R^{M_s}}+\mc O(\Delta^{-1/p}M_s^{1/p})(\|u\|_{p,\R^L}+\|w\|_{p,\R^{M_u}}+\|v\|_{p,\R^{M_s}})
\end{align*}

Suppose that  $(u,w,v)$ satisfies the cone condition $\|u\|_{p,\R^L}+\|w\|_{p,\R^{M_u}}\geq\tau\|v\|_{p,\R^{M_s}}$ for some $\tau$.  Then 
\begin{align*}
\frac{\|u'\|_{p,\R^L}+\|w'\|_{p,\R^{M_u}}}{\|v'\|_{p,\R^{M_s}}}&\geq \frac{\mc F_{11}(\|u\|_{p,\R^L}+\|w\|_{p,\R^{M_u}})-\mc F_{12}\|v\|_{p,\R^{M_s}}}{\mc F_{21}\|v\|_{p,\R^{M_s}}+\mc F_{22}(\|u\|_{p,\R^L}+\|w\|_{p,\R^{M_u}})}\\
&\geq \frac{\mc F_{11}-\tau^{-1}\mc F_{12}}{\tau^{-1}\mc F_{21}+\mc F_{22}}
\end{align*}
with 
\begin{align*}
\mc F_{11}&:=\min\left\{\sigma-\mc O(\Delta^{-1}\degree),\lambda^{-1}-C_\#\epsilon-\mc O(\Delta^{-1}M)\right\}-\max\{\mc O(\Delta^{-1}L^{1/p}\degree^{1/q}), \mc O(\Delta^{-1/p}M_u^{1/p}) \}\\
& = \min\left\{\sigma,\lambda^{-1}-C_\#\epsilon\right\}-\mc O(\eta), \\
\mc F_{12}&:=\max\{\mc O(\Delta^{-1}L^{1/p}\degree^{1/q}), \mc O(\Delta^{-1/p}M_u^{1/p})\}= \mc O(\eta),\\
\mc F_{21}&:=\lambda+C_\#\epsilon+\mc O(\Delta^{-1}M)+\mc O(\Delta^{-1/p}M_s^{1/p}))= \lambda+C_\#\epsilon + \mc O(\eta),\\
\mc F_{22}&:=\mc O(\Delta^{-1/p}M_s^{1/p}))= \mc O(\eta),
\end{align*}
where we used \eqref{Eq:ThmCond1}-\eqref{Eq:ThmCond3}.
The cone $\mc C_p^u$ is forward invariant iff  $\|u'\|_{p,\R^L}+\|w'\|_{p,\R^{M_u}}\geq \tau \|v'\|_{p,\R^{M_s}}$ and therefore if 
\begin{align}
\mc F_{11}-\tau^{-1}\mc F_{12}\geq\mc F_{21}+\tau\mc F_{22}.
\label{Eq:taucond}
\end{align} 
Hence we find $C_*>0$, so that if $\tau=C_*/ \mc F_{22}$ the inequality \eqref{Eq:taucond} is satisfied provided 
 \eqref{AttCond3} holds and  $\eta>0$ is small enough because then 
  $\mc F_{11}>\mc F_{21}$. 

Now let us check when the cone $\mc C_p^s$ is backward invariant.
Suppose that $\|u'\|_{p,\R^L}+\|w'\|_{p,\R^{M_u}}\leq\tau\|v'\|_{p,\R^{M_s}}$, thus
\begin{align*}
{\mc F_{11}\frac{\|u\|_{p,\R^L}+\|w\|_{p,\R^{M_u}}}{\|v\|_{p,\R^{M_s}}}-\mc F_{12}}&\leq\tau\mc F_{21}+\tau\mc F_{22}\frac{\|u\|_{p,\R^L}+\|w\|_{p,\R^{M_u}}}{\|v\|_{p,\R^{M_s}}}\\
\frac{\|u\|_{p,\R^L}+\|w\|_{p,\R^{M_u}}}{\|v\|_{p,\R^{M_s}}}&\leq \frac{\mc F_{12}+\tau\mc F_{21}}{\mc F_{11}-\tau\mc F_{22}}
\end{align*} 
and imposing, yet again, 
\begin{equation}\label{Eq:BackInvTauCond}
\tau^{-1}\mc F_{12}+\mc F_{21}\leq \mc F_{11}-\tau\mc F_{22},
\end{equation}
 implies that $\|u\|_{p,\R^L}+\|w\|_{p,\R^{M_u}}\leq\tau\|v\|_{p,\R^{M_s}}$. Taking $\tau=C_* \mc F_{12}$ with $C_*>0$ small, we obtain 
 that  $\mc C_p^s$ is backward invariant (provided as before  that \eqref{AttCond3} holds and  $\eta>0$ is small).

%
%

(ii) Take $(u,w,v)	\in\mc C_p^u$ such that $\|(u,w,v)\|_{p}=1$. From the above computations, and applying the cone condition
\begin{align}
\|u'\|_{p,\R^L}+\|w'\|_{p,\R^{M_u}}+\|v'\|_{p,\R^{M_s}}&\geq \|u'\|_{p,\R^L}+\|w'\|_{p,\R^{M_u}}\nonumber\\
&\geq{\mc F_{11}(\|u\|_{p,\R^L}+\|w\|_{p,\R^{M_u}})-\mc F_{12}\|v\|_{p,\R^{M_s}}}\nonumber\\
&\geq{\mc F_{11}(1-\beta_{u,p})-\mc F_{12}\beta_{u,p}}\nonumber\\
&\geq \min\left\{\sigma,\lambda^{-1}-C_\#\epsilon\right\}-\mc O(\eta)-\mc O(\eta^2)\label{Eq:ExpConEst1}
\end{align}
where to obtain \eqref{Eq:ExpConEst1} we kept only the largest order in the parameters of the network, after substituting the expressions for $\mc F_{11}$ and $\mc F_{12}$. This means that in conditions \eqref{Eq:ThmCond1}-\eqref{Eq:ThmCond3}, if $\eta>0$ is sufficiently small, \eqref{Eq:ExpResUnstCone} will be satisfied.  Choosing, now, $(u,v,w)\in\mc C_p^s$ of unit norm we get
\begin{align*}
\|u'\|_{p,\R^L}+\|w'\|_{p,\R^{M_u}}+\|v'\|_{p,\R^{M_s}}&\leq (1+\beta_{s,p})\Delta^{-1}\|v'\|_{p,\R^{M_s}}\\
&\leq \lambda+C_\#\epsilon+\mc O(\Delta^{-1}M)+\mc O(\Delta^{-1/p}M^{1/p})+\beta_{s,p}\\
&\leq \lambda +C_\#\epsilon +\mc O(\eta)
\end{align*}
and again whenever conditions \eqref{Eq:ThmCond1}-\eqref{Eq:ThmCond3} are satisfied with $\eta>0$ sufficiently small, \eqref{Eq:ExpResStabCone} is verified.
\end{proof}

\subsection{Admissible Manifolds for $F_\epsilon$}\label{Sec:AdmManifolds}

As in the diffeomorphism case,  the existence of the stable and unstable cone fields implies that the 
the endomorphism ${F_\epsilon}$ admits a natural measure. 

To determine the measure of the set $\mc B_\epsilon\times\mb T^M$ with respect to one of these measures we need to estimate how much the marginals on the coordinates of the low degree nodes differ from Lebesgue measure. To do this we  look at the evolution of densities supported on admissible manifolds, namely manifolds whose tangent space is contained in the unstable cone and whose geometry is controlled. To control the geometry locally, we invoke the Hadamard-Perron graph transform argument (see for example \cite{shub2013global,MR1326374}) (Appendix \ref{Ap:GraphTrans}) which implies that manifolds tangent to the unstable cone which are locally graph of functions in a given regularity class, are mapped by the dynamics into manifolds which are locally graphs of functions in the same regularity class.  
 
As before $\mb T=\R/\sim$ with $x_1\sim x_2$ when $x_1-x_2\in \Z$, so each point in $\mb T$ can be identified 
with a point in $[0,1)$. Define $I=(0,1)$.  
 
\begin{definition}[Admissible manifolds $\mc W_{p,K_0}$] 
For every $K_0>0$ and $1\leq p\leq\infty$ we say that a manifold $W$ of $\mc S$ is {\em admissible} and belongs to the set $\mc W_{p,K_0}$ if 
there exists a differentiable function $E:I^{L+M_u} \rightarrow \mc R$ with Lipschitz differential so that
%
\begin{itemize}
\item $W$ is  the graph $(id, E)(I^{L+M_u})$ of $E$, 
\item  $D_{z_u}E(\R^{L+M_u})\subset\mc C_p^u$, $\forall z_u\in I^{L+M_u}$,
\item and 
\[
\|DE\|_{\Lip}:=\sup_{z_u\neq\bar z_u}\frac{\|D_{z_u}E-D_{\bar z_u}E\|_p}{d_p(z_u,\bar z_u)}\leq K_0,
\]
where, with an abuse of notation, we denoted by $\|\cdot\|_p$ the operator norm of linear transformations from $(\R^{L+M_u},\|\cdot\|_{p,\R^L}+\|\cdot\|_{p,\R^{M_u}})$ to $(\R^{M_s},\|\cdot\|_{p,\R^{M_s}})$.
\end{itemize}
\end{definition} 

\begin{proposition}
Under conditions \eqref{Eq:ThmCond1}-\eqref{Eq:ThmCond3}, for $\eta>0$ sufficiently small,
there is $K_u$ uniform on the network parameters such that for all $z_1,z_2\in \mc S$
the norm
\[
\|D_{z_1}{F_\epsilon}-D_{z_2}{F_\epsilon}\|_{u,p}:=\sup_{(u,w,v)\in \mc C_p^u}\frac{\|(D_{z_1}{F_\epsilon}-D_{z_2}{F_\epsilon})(u,w,v)\|_{p}}{\|(u,w,v) \|_{p}}
\]
satisfies
\[
\|D_{z_1}{F_\epsilon}-D_{z_2}{F_\epsilon}\|_{u,p}\leq K_ud_\infty(z_1,z_2).
\]
\end{proposition}
\begin{proof}
Notice that from the regularity assumptions on the coupling function $h$, we can write the entries of $D_{z_1}{F_\epsilon}-D_{z_2}{F_\epsilon}$ as
\begin{equation}\label{Eq:DiffAuxMapDiff}
[D_{z_1}{F_\epsilon}-D_{z_2}{F_\epsilon}]_{k\ell}=\left\{\begin{array}{ll}
\left[\sum_{\ell=1}^L\mc O(\Delta^{-1})A^{ll}_{k\ell}+\sum_{m=1}^M\mc O(\Delta^{-1})A^{lh}_{km}\right]d_{\infty}(z_1,z_2)& k=\ell \leq L\\
\mc O(\Delta^{-1})A^{ll}_{k\ell}d_\infty(z_1,z_2)& k\neq \ell \leq L\\
\mc O(\Delta^{-1})A^{lh}_{k(\ell-L)}d_\infty(z_1,z_2)& k\leq L, \ell> L\\
\mc O(\Delta^{-1})A^{hl}_{k\ell}d_\infty(z_1,z_2)& k> L, \ell\leq L\\
\mc O(\Delta^{-1})A^{hh}_{(k-L)(\ell-L)}d_\infty(z_1,z_2)&k\neq \ell > L\\
\left[\mc O(1)+\mc O(\Delta^{-1}M)\right]d_{\infty}(z_1,z_2)&k= \ell > L\\
\end{array}\right.
\end{equation}
Take $(u,w,v)\in \mc C_p^u$ such that $\|(u,w,v)\|_p=1$ and $(u',w',v')^t=(D_{z_1}{F_\epsilon}-D_{z_2}{F_\epsilon})(u,w,v)^{t}$.
\begin{align*}
u'_i&=\mc O(\Delta^{-1})\left[\sum_{\ell=1}^LA^{ll}_{k\ell}+\sum_{m=1}^MA^{lh}_{km}\right]u_id_{\infty}(z_1,z_2)+\\
&\phantom{=}+\mc O(\Delta^{-1})\left[\sum _{\ell=1}^LA^{ll}_{i\ell}u_n+\sum_{m=1}^{M_u}A^{lh}_{im}w_m+\sum_{m=1}^{M_s}A^{lh}_{i(m+M_u)}v_{m}\right]d_{\infty}(z_1,z_2)& 1\leq i\leq L\\
w'_j&=\left[\mc O(1)+\mc O(\Delta^{-1}M)\right]w_{j}d_{\infty}(z_1,z_2)\\
&\phantom{=}+\mc O(\Delta^{-1})\left[\sum_iA^{hl}_{ji}u_i+\sum_{m=1}^{M_u}A^{hh}_{jm}w_m+\sum_{m=1}^{M_s}A^{hh}_{j(m+M_u)}v_m\right]d_\infty(z_1,z_2)&1\leq j\leq M_u\\
v'_j&=\left[\mc O(1)+\mc O(\Delta^{-1}M)\right]v_jd_{\infty}(z_1,z_2)+\\
&\phantom{=}+\mc O(\Delta^{-1})\left[\sum_{\ell=1}^{L}A^{hl}_{j\ell}u_\ell+\sum_{m=1}^{M_u}A^{hh}_{jm}w_m+\sum_{m=M_u+1}^{M}A^{hh}_{jm}v_m\right]d_{\infty}(z_1,z_2)&1\leq j\leq M_s
\end{align*}
\begin{align*}
\|u'\|_{p,\R^L}&\leq\mc O(\Delta^{-1}\degree N^{1/p})d_\infty(z_1,z_2)=\mc O(\eta) d_\infty(z_1,z_2) \\
\|w'\|_{p,\R^{M_u}}&\leq\mc O(1)d_\infty(z_1,z_2)\\
\|v'\|_{p,\R^{M_s}}&\leq\mc O(1)d_\infty(z_1,z_2)
\end{align*}
which implies the proposition.
\end{proof}



\begin{lemma}\label{Lem:CovDec}
Suppose that $K_0>\mc O(K_u)$ and $W$ is  an embedded $(L+M_u)-$dimensional torus  
which is the closure of   $W_0\in\mc W_{p,K_0}$. Then, for every $n\in\N$, ${F_\epsilon}^n(W)$ is the closure of a finite 
union of manifolds, $W_{n,k}\in \mc W_{p,K_0}$, $k\in\mc K_n$ (and the difference ${F_\epsilon}^n(W)\setminus  \cup \{W_{n,k}\}_{k\in\mc K_n}$ consists of finite union of manifolds of lower dimension). 
\end{lemma}
\begin{proof}
As in Corollary~\ref{cor2}, 
since $\pi_u|_{W_0}$ is a diffeomorphism, the map $\pi_u\circ{F_\epsilon}^n\circ \pi_u|_{W_0}^{-1}:\mb T^{L+M_u}\rightarrow \mb T^{L+M_u}$ is a well defined local diffeomorphism between compact manifolds, and therefore is a covering map. One can then find a partition $\{R_{n,k}\}_{k\in\mc K_n}$ of $\mb T^{L+M_u}$ such that $\pi_u\circ{F_\epsilon}^n\circ \pi_u|_{W_0}^{-1}(R_{n,k})=\mb T^{L+M_u}$ and, defining $W_{n,k}:=\pi_u\circ{F_\epsilon}^n\circ \pi|_{W_0}^{-1}(R^o_{n,k})$, where $R^o_{n,k}$ is the interior of $R_{n,k}$,
$\pi_u(W_{n,k})=I^{L+M_u}$. From Proposition~\ref{Prop:InvRegularityLip} in Appendix A it follows that $W_{n,k}\in \mc W_{p,K_0}$ and $\{W_{n,k}\}_{k\in\mc K_n}$ is the desired partition.
\end{proof}

\begin{figure}[htbp]
\centering
\includegraphics[scale=0.58]{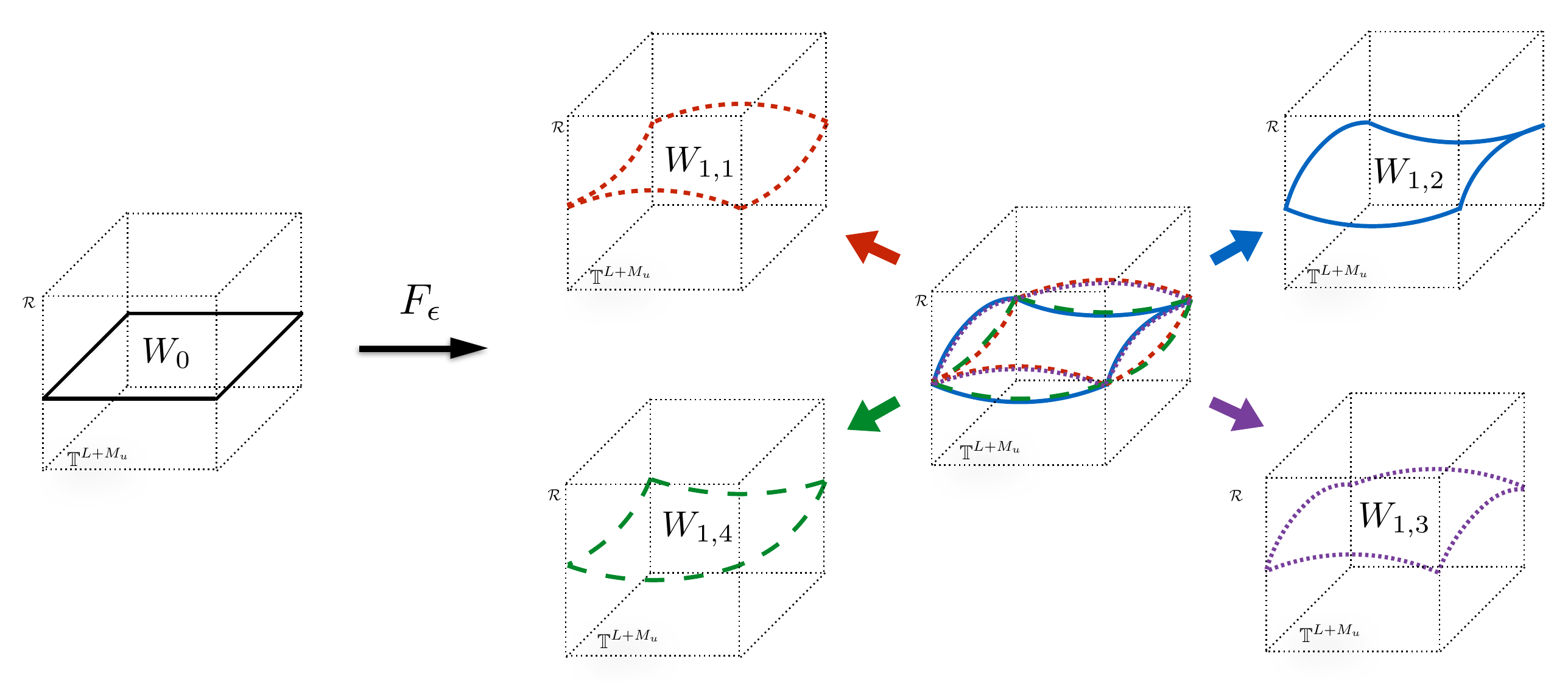}
\caption{The admissible manifold $W_0$ is mapped under $F_\epsilon$ to the union of sub manifolds $W_{1,1}$, $W_{1,2}$, $W_{1,3}$, and $W_{1,4}$.}
\label{Fig:AdmMnfds}
\end{figure}

\subsection{Evolution of Densities on the Admissible Manifolds for $F_\epsilon$}

Recall that $\pi_u$ and $\Pi_u$ are projections on the first $L+M_u$ coordinates in $\mb T^N$ and $\R^N$ respectively.
Given an admissible manifold $W\in\mc W_{p,K_0}$, which is the graph of the function $E:I^{L+M_u}\rightarrow \mc R$, for every $z_u\in I^{L+M_u}$ the map 
\[
\pi_u\circ{F_\epsilon}\circ(id,E)(z_u)
\]
gives the evolution of the first $L+M_u$ coordinates of points in $W$. The Jacobian of this map is given by 
$$
J(z_u)=\left|\Pi_u\cdot D_{(id,E)(z_u)}{F_\epsilon}\cdot(\Id, D_{z_u}E)\right|.
$$
In the next proposition we upper bound the distortion of such a map.
\begin{proposition}\label{Prop:JacobEst}
Let $W\in W_{p,L}$ be an admissible manifold and suppose $z_u,\bar z_u\in I^{L+M_u}$, then
$$
\left|\frac{J(z_u)}{J(\bar z_u)}\right|\leq \exp\left\{[\mc O(\Delta^{-1}L^{1+1/p}\degree^{1/q})+\mc O(M)]d_\infty(z_u,\bar z_u)\right\}.
$$
\end{proposition}
\begin{proof}
\begin{align*}
\frac{|J(z_u)|}{|J(\bar z_u)|}&=\frac{\left|\Pi_u\cdot D_{(id,E)(z_u)}{F_\epsilon}\cdot(\Id, D_{z_u}E)\right|}{\left|\Pi_u\cdot D_{(id,E)(\bar z_u)}{F_\epsilon}\cdot(\Id, D_{z_u}E)\right|}\frac{\left|\Pi_u\cdot D_{(id,E)(\bar z_u)}{F_\epsilon}\cdot(\Id, D_{z_u}E)\right|}{\left|\Pi_u\cdot D_{(id,E)(\bar z_u)}{F_\epsilon}\cdot(\Id, D_{\bar z_u}E)\right|}\\
&=:(A)\cdot(B)
\end{align*}
$(A)$ can be bounded with computations similar to the ones carried on in Proposition \ref{Prop:DistJac}:
\[
(A)\leq \exp\left\{\left[\mc O(\Delta^{-1}\degree L)+\mc O(M)\right]d_\infty(z,\bar z)\right\}.
\]

To estimate $(B)$ we also factor out the number $Df=\sigma$ from the first $L$ columns of $\Pi_uD_{(id,E)(\bar z_u)}{F_\epsilon}$, and  $D\redu_j(\bar y_{u,j})$ from the $(L+j)-$th column when $1\leq j\leq M_u$ and thus obtain
\[
(B)=\frac{\sigma^L}{\sigma^L}\cdot\frac{\prod_{j=1}^{M_u}D\redu_j}{\prod_{j=1}^{M_u}D\redu_j}\cdot\frac{\left|\Pi_u\mc D({(id,E)(\bar z_u)})\cdot(\Id, D_{z_u}E)\right|}{\left|\Pi_u\mc D((id,E)
(\bar z_u))\cdot(\Id, D_{\bar z_u}E)\right|}
\]
where $\mc D(\cdot)$ is the same matrix defined in \eqref{Eq:DExpression} apart from the last $M_s$ columns which are kept equal to the corresponding columns of $D_{\cdot}{F_\epsilon}$. The first two ratios trivially cancel. For the third factor we proceed in a fashion similar to previous computations using Proposition \ref{Prop:EstTool} in the appendix. Defining $\B:=\mc D((id,E)(\bar z_u))-\Id$, we are reduced to estimate
\begin{align*}
\frac{\left|\Id+\Pi_u\cdot \B\cdot(\Id, D_{z_u}E)\right|}{\left|\Id+\Pi_u\cdot \B\cdot(\Id, D_{\bar z_u}E)\right|}
\end{align*}
where we used that $\Pi_u\mc D\cdot(\Id,D_{z_u}E)-\Id=\Pi_u\B\cdot(\Id,D_{z_u}E)$. 

Since $\|(\Id,D_{z_u}E)\|_p\leq(1+\beta_{u,p})$ for any $z_u\in \mc S$, it follows, choosing $\eta>0$ sufficiently small in \eqref{Eq:ThmCond1}-\eqref{Eq:ThmCond3} and from equation \eqref{Eq:bOperatorNormBound} that the operator norm 
\begin{equation}\label{Eq:bContEvar}
\|\Pi_u\cdot \B\cdot (\Id,D_{z_u}E)\|_{p}<\lambda<1
\end{equation} 
 It is also rather immediate to upper bound the column norms of $\Pi_u\cdot \B\cdot (0,D_{z_u}E-D_{\bar z_u}E)$ and obtain
\begin{align*}
\|\Col^i[\Pi_u\B(0,D_{z_u}E-D_{\bar z_u}E)]\|_p&\leq\mc O(\Delta^{-1} L^{1/p}\degree^{1/q})\|DE\|_{\Lip,p}d_p(z_u,\bar z_u)\\
&\leq \mc O(\Delta^{-1} M^{1/p})d_p(z_u,\bar z_u)
\end{align*}
so that by Proposition \ref{Prop:EstTool}, the overall estimate for (B) is 
\begin{equation}\label{Eq:EstRatioDetPointE}
\frac{\left|\Pi_u\cdot \mc D({(id,E)(\bar z_u)})\cdot(\Id, D_{z_u}E)\right|}{\left|\Pi_u\cdot \mc D((id,E)(\bar z_u)\cdot(\Id, D_{\bar z_u}E)\right|}\leq\exp\left\{\mc O(\Delta^{-1}L^{1+1/p}\degree^{1/q})d_p(z_u,\bar z_u)\right\}.
\end{equation}
\end{proof}

\subsection{Invariant Cone of Densities on Admissible Manifolds for $F_\epsilon$}\label{subsec:invariantcones}
Take $W\in\mc W_{p,K_0}$. A density $\rho$ on $W$ is a measurable function $\phi:W\rightarrow\R^+$ such that the integral of $\phi$ over $W$ with respect to $\leb_W$ is one, 
where $\leb_W$ is defined to be the measure obtained by restricting the volume form in $\mb T^N$ to $W$.
The measure ${\pi_u}_{*}(\phi\cdot\leb_W)$ is absolutely continuous with respect to $\leb_{L+M_u}$ on $\mb T^{L+M_u}$ and so its density $\phi_u:\mb T^{L+M_u}\rightarrow\R^+$ is well defined.
\begin{definition}\label{Def:UnstableMarginal}
For every $W\in\mc W_{p,K_0}$ and for every $\phi:W\rightarrow\R^+$ we define $\phi_u$ as
\[
\phi_u:=\frac{d{\pi_u}_{*}(\phi\cdot\leb_W)}{d\leb_{L+M_u}}.
\]
\end{definition}
Consider the set of densities
\[
\mc C_{a,p}(W):=\left\{\phi:W\rightarrow \R^+\mbox{ s.t. }\frac{\phi_u(z_u)}{\phi_u(\bar z_u)}\leq \exp[ad_p(z_u,\bar z_u)]\right\}.
\]
The above set consists of all densities on $W$ whose projection on the first $L+M_u$ coordinates has the prescribed regularity property. 
\begin{proposition}\label{Prop:DensEvSubm}
For every $a>a_c$, where 
\begin{equation}\label{Eq:Acriticattfix}
a_c={\mc O(\Delta^{-1}L^{1+1/p}\degree^{1/q})+\mc O(M)}
\end{equation}
 $W\in \mc W_{p,K_0}$ and $\phi\in\mc C_{a,p}(W)$ the following holds. Suppose that $\{W_k'\}_k$ is the partition of ${F_\epsilon}(W)$ given by Lemma \ref{Lem:CovDec} and that $W_k$ is a manifold of $W$ such that ${F_\epsilon}(W_k)=W_k'\in \mc W_{p,K_0}$, then for every $k$,  the density $\phi_k'$ on $W_k'$ defined as
\[
\phi_k':=\frac{1}{\int_{W_k}\phi d\leb_{W}}\frac{d{F_\epsilon}_*(\phi|_{W_k}\cdot\leb_{W_k})}{d\leb_{W_k'}}
\]
belongs to $\mc C_{a,p}(W_k')$. 
\end{proposition}
\begin{proof}
It is easy to verify that $\phi'_k$ is well-defined. Let $G_k$ be the inverse of the map ${F_\epsilon}|_{W_k}:W_k\rightarrow W_k'$. From Definition \ref{Def:UnstableMarginal} follows that
\[
(\phi'_k)_u:=\frac{d(\pi_u\circ {F_\epsilon}\circ (id,E))_*(\phi_u|_{\pi_u(W_k)}\cdot\leb_{L+M_u})}{d\leb_{L+M_u}}
\]
where $E$ is the map whose graph equals $W$. This implies that
\[
(\phi'_k)_u(z_u)=\frac{\phi_u(G_k(z_u))}{J(G_k(z_u))}
\]
and therefore 
\begin{align*}
\frac{(\phi'_k)_u(z_u)}{(\phi'_k)_u(\bar z_u)}&=\frac{\phi_u(G_k(z_u))}{\phi_u(G_k(\bar z_u))}\frac{J(G_k(\bar z_u))}{J(G_k(z_u))}\\
&\leq \exp\left[\bar\sigma^{-1}a d_p(z_u,\bar z_u)\right] \exp\left\{[\mc O(\Delta^{-1}L^{1+1/p}\degree^{1/q})+\mc O(M)]d_p(z_u,\bar z_u)\right\}\\
&\leq \exp\left\{[\bar\sigma^{-1}a+\mc O(\Delta^{-1}L^{1+1/p}\degree^{1/q})+\mc O(M)]d_p(z_u,\bar z_u))\right\} .
\end{align*}
Taking $a_c$ as in \eqref{Eq:Acriticattfix}, the proposition is verified. 
\end{proof}

At this point we can prove that the system admits invariant physical measures and that their marginals on the first $L+M_u$ coordinates is in the cone $C_{a,p}$ for $a>a_c$. The main ingredients we use are Krylov-Bogolyubov's theorem, and  Hopf's argument \cite{MR3220769, MR1326374}.
\begin{proof}[Proof of Theorem \ref{Thm:PhysMeasFep}]
Pick a periodic orbit, $O(z_s)$ of $g_{M_u+1}\times \dots \times g_{M}$ and let $U$ be the union of the connected components of $\mc R$ containing points of $O(z_s)$. Pick $y_s\in U$ and take the admissible manifold 
$W_0:=\mb T^{L+M_u}\times\{y_s\}\in \mc W_{p,K_0}$.  
Consider a density $\rho\in \mc C_{a,p}(W_0)$ with $a>a_c$ such that the measure $\mu_0:=\rho\cdot m_{W}$ is the probability measure supported on $W_0$ with density $\rho$ with respect to the Lebesgue measure on 
$W_0$. Consider the sequence of measures $\{\mu_t\}_{t\in\N_0}$ defined as
\[
\mu_t:=\frac{1}{t+1}\sum_{i=0}^{t}{F_\epsilon}^i_*\mu_0.
\]
From Lemma \ref{Lem:CovDec} we know that ${F_\epsilon}^i(W_0)=\bigcup_{k\in\mc K_i}W_{i,k}$ modulo a negligible set w.r.t. ${F_\epsilon}^i_*(\mu_0)$, and that 
\[
{F_\epsilon}^i_*(\mu_0)=\sum_{k\in\mc K_i}{F_\epsilon}^i_*\mu_0(W_{i,k})\mu_{i,k},
\]
where $\mu_{i,k}$ is a probability measure supported on $W_{i,k}$ for all $i$ and $k\in\mc K_i$. It is a consequence of Proposition \ref{Prop:DensEvSubm} that $\mu_{i,k}=\rho_{i,k}\cdot m_{W_{i,k}}$ with $\rho_{i,k}\in C_{a,p}(W_{i,k})$. Since ${F_\epsilon}$ is continuous, every subsequence of $\{\mu_t\}_{t\in\N_0}$ has a convergent subsequence in the set of all probability measures of $\mc S$ with respect to the weak topology ($\{\mu_t\}_{t\in\N_0}$ is tight). Let $\bar \mu$ be a probability measure which is a limit of a  converging subsequence. By convexity of the cone $\mc C_{a,p}$ the second  assertion 
of the theorem holds for $\bar \mu$. 

Now let $V_1,\dots,V_n$ be the components of $U$ where $n$ is the period of $O(z_s)$.  Since all stable manifolds are tangent to a constant cone which has a 
very small angle (in particular less then $\pi/4$) with the vertical direction (corresponding to the last $M_s$ directions of $\mb T^N$), 
they will intersect all horizontal tori $\mb T^{L+M_u}\times\{y\}$ with $y\in V_k$. It follows from the standard arguments that $\bar \mu$ has absolutely continuous disintegration on foliations of local unstable manifolds, which in the case of an endomorphism, are defined on a set of histories called inverse limit set (see \cite{qian2009smooth} for details). Following the standard Hopf argument (\cite{MR3220769, MR1326374}), one first notices that fixed a point $x\in V_i$ on the support of $\bar \mu$, a history $\bar x\in(\mb T^N)^\N$, and a continuous observable $\phi$, from the definition of $\bar \mu$, almost every point on the local unstable manifold associated to the selected history has  a well defined forward asymptotic Birkhoff average (computed along $\bar x$) and every point on that stable manifold through $x$ has the same asymptotic forward Birkhoff average. The aforementioned property of the stable leaves implies that any two unstable manifolds are crossed by the same stable leaf. This, together with absolute continuity of the stable foliation, implies that forward Birkhoff averages of $\phi$ are constant almost everywhere on the support of $\bar \mu$ which implies ergodicity. 
\end{proof}

\subsection{Jacobian of the Holonomy Map along Stable Leaves of $F_\epsilon$}
In order to prove Proposition~\ref{Prop:BadSetMeasAtt} we need to upper bound the Jacobian of the holonomy map along stable leaves. 
It is known that for a $C^2$ uniformly hyperbolic (or even partially hyperbolic)  diffeomorphism the holonomy map along the stable leaves is absolutely continuous with respect to the induced Lebesgue measure on the transversal to the leaves \cite{hasselblatt2005partially}. This can be easily generalised to the non-invertible case. First of all, let us recall the definition of holonomy map. We consider holonomies between manifolds tangent to the unstable cone.

\begin{definition}
Given $D_1$ and $D_2$ embedded disks of dimension $L+M_u$, tangent to the unstable cone $\mc C^u$, we define the holonomy map $\pi:D_1\rightarrow D_2$ 
\[
\pi(x)=W^s(x)\cap D_2.
\]
As before we define $\leb_D$ be the Lebesgue measure on $D$ induced by the volume form on $\mb T^N$. 
\end{definition}
\begin{remark}
For the truncated dynamical system ${F_\epsilon}$, fixing $D_1$, one can always find a sufficiently large $D_2$ such that the map is well defined everywhere in $D_1$.
\end{remark}

\begin{proposition}\label{Prop:JacEstBnd}
Given $D_1$ and $D_2$  admissible embedded disks tangent to the unstable cone, the holonomy map $\pi:D_1\rightarrow D_2$  associated to ${F_\epsilon}$ is absolutely continuous with respect to $m_{D_1}$ and $m_{D_2}$ which are the restrictions of Lebesgue measure to the two embedded disks. Furthermore, if $J_s$ is the Jacobian of $\pi$, then 
\begin{equation}\label{Eq:JacEstUppBnd}
J_s(z)\leq\exp\left\{[\mc O(\Delta^{-1}L\degree )+\mc O(M)]\frac{1}{1-\lambda}d_\infty(z,\pi(z))\right\},\quad \forall z\in D_1.
\end{equation}
\end{proposition}
\begin{proof}
The absolute continuity follows from results in \cite{MR889254} (see Appendix \ref{Ap:GraphTrans}), as well as the estimate on the Jacobian. In fact it is proven in \cite{MR889254} that 
\[
J_s(z)=\prod_{k=0}^\infty\frac{\Jac\left(D_{z_k}{F_\epsilon}| V_k\right)}{\Jac\left(D_{\bar z_k}{F_\epsilon}|\bar V_k\right)}\quad \forall z\in D_1,
\]
where $z_k:={F_\epsilon}^k(z)$, $\bar z_k:={F_\epsilon}^k(\pi(z))$, $V_k:=T_{z_k}{F_\epsilon}^k(D_1)$ and $\bar V_k:=T_{\bar z_k}{F_\epsilon}^k(D_2)$. Since  $D_1$ and $D_2$  are tangent 
to the unstable cone, one can write 
${F_\epsilon}^k(D_1)$ and ${F_\epsilon}^k(D_2)$ locally as graphs of functions $E_{1,k}:B^u_\delta(z_k)\rightarrow B^s_\delta(z_k)$ and $E_{2,k}:B^u_\delta(\bar z_k))\rightarrow B^s_\delta(\bar z_k)$, with $E_{i,k}$ given by application of the graph transform on $E_{i,k-1}$, and $E_{i,0}$ such that $(\Id,D_zE_{i,0})(\R^{L+M_u})= T_z D_i$.

For every $k\in\N\cup\{0\}$, $(\Id, D_{\pi_u(z_k)}E_{1,k})\circ \Pi_u|_{V_k}=\Id|_{V_k}$, 
\[
\Jac(\Id, D_{\pi_u(z_k)}E_{1,k})\Jac(\Pi_u|_{V_k})=1
\]
and analogously
\[
\Jac(\Id, D_{\pi_u(\bar z_k)}E_{2,k})\Jac(\Pi_u|_{\bar V_k})=1.
\]
Since 
\begin{align*}
\left|\Pi_u\circ D_{z_k}{F_\epsilon}(\Id,D_{\pi_u(z_k)}E_{1,k})\right|&=\Jac(\Pi_u\circ D_{z_k}{F_\epsilon}(\Id,D_{\pi_u(z_k)}E_{1,k}))\\
&=\Jac(\Pi_u|_{V_k})\Jac(\Id, D_{\pi_u(z_k)}E_{1,k})\Jac\left(D_{z_k}{F_\epsilon}|V_k\right)
\end{align*}
and, analogously,
\begin{align*}
\left|\Pi_u\circ D_{\bar z_k}{F_\epsilon}(\Id,D_{\pi_u(\bar z_k)}E_{2,k})\right|&=\Jac(\Pi_u\circ D_{\bar z_k}{F_\epsilon}(\Id,D_{\pi_u(\bar z_k)}E_{2,k}))\\
&=\Jac(\Pi_u|_{\bar V_k})\Jac(\Id, D_{\pi_u(\bar z_k)}E_{2,k})\Jac\left(D_{\bar z_k}{F_\epsilon}|\bar V_k\right).
\end{align*}
So
\begin{align*}
J_s(z)&= \prod_{k=0}^\infty\frac{\Jac\left(D_{z_k}{F_\epsilon}|V_k\right)}{\Jac\left(D_{\bar z_k}{F_\epsilon}|V_k\right)}\frac{\Jac\left(D_{\bar z_k}{F_\epsilon}|V_k\right)}{\Jac\left(D_{\bar z_k}{F_\epsilon}|\bar V_k\right)}\\
&=\prod_{k=0}^\infty\frac{\left|\Pi_u\circ D_{z_k}{F_\epsilon}(\Id,D_{\pi_u(z_k)}E_{1,k})\right|}{\left|\Pi_u\circ D_{\bar z_k}{F_\epsilon}(\Id,D_{\pi_u(z_k)}E_{1,k})\right|}\frac{\left|\Pi_u\circ D_{\bar z_k}{F_\epsilon}(\Id,D_{\pi_u(z_k)}E_{1,k})\right|}{\left|\Pi_u\circ D_{\bar z_k}{F_\epsilon}(\Id,D_{\pi_u(\bar z_k)}E_{2,k})\right|}.
\end{align*}
The first ratio can be deduced with minor changes from the estimate \eqref{Eq:EstRatioDetPointE} in the proof of Proposition \ref{Prop:JacobEst}. So
\begin{align*}
\prod_{k=0}^\infty\frac{\left|\Pi_u\circ D_{z_k}{F_\epsilon}(\Id,D_{\pi_u(z_k)}E_{1,k})\right|}{\left|\Pi_u\circ D_{\bar z_k}{F_\epsilon}(\Id,D_{\pi_u(z_k)}E_{1,k})\right|}&\leq  \exp\left\{[\mc O(\Delta^{-1}L\degree)+\mc O(M)]\sum_{k=0}^{\infty}d_p(z_k,\bar z_k)\right\}\\
&\leq\exp\left\{[\mc O(\Delta^{-1}L\degree)+\mc O(M)]\frac{1}{1-\bar \lambda}d_p(z_0,\bar z_0)\right\}
\end{align*}
where we used the fact that $z_0$ and $\bar z_0$ lay on the same stable manifold and, by Proposition \ref{Prop:InvConesForTildeF}: $d_{p}(z_k,\bar z_k)\leq \bar \lambda^kd_p(z_0,\bar z_0)$.

To estimate the other ratio we proceed making similar computations leading to the estimate in \eqref{Eq:EstRatioDetPointE}. Once more we factor out $\sigma$ from the first $N$ columns of $D_{\bar z_k}{F_\epsilon}$ and, for all $j\in[M]$, $D_{y_j}\redu$ from the $(L+j)-$th column and thus
\begin{align*}
\frac{\left|\Pi_u\circ D_{\bar z_k}{F_\epsilon}(\Id,D_{\pi_u( z_k)}E_{1,k})\right|}{\left|\Pi_u\circ D_{\bar z_k}{F_\epsilon}(\Id,D_{\pi_u(\bar z_k)}E_{2,k})\right|}&=\frac{\sigma^L}{\sigma^L}\cdot\frac{\prod_{j=1}^{M_u}D\redu_j}{\prod_{j=1}^{M_u}D\redu_j}\frac{\left|\Pi_u\circ D(\bar z_k)(\Id,D_{\pi_u( z_k)}E_{1,k})\right|}{\left|\Pi_u\circ D(\bar z_k)(\Id,D_{\pi_u(\bar z_k)}E_{2,k})\right|}\\
&=\frac{\left|\Pi_u\circ \mc D(\bar z_k)(\Id,D_{\pi_u( z_k)}E_{1,k})\right|}{\left|\Pi_u\circ \mc D(\bar z_k)(\Id,D_{\pi_u(\bar z_k)}E_{2,k})\right|}
\end{align*}
where $\mc D(z_k)$ is defined as in \eqref{Eq:DExpression}. Defining for every $k\in\N$ 
\begin{align*}
\B_k&:=\Pi_u\mc D(\bar z_k)(\Id,D_{\pi_u(z_k)}E_{1,k})-\Id\\
&=\Pi_u(\mc D(\bar z_k)-\Id)(\Id,D_{\pi_u( z_k)}E_{1,k})
\end{align*}
and analogously
\begin{align*}
\bar \B_k&:=\Pi_u\mc D(\bar z_k)(\Id,D_{\pi_u(\bar z_k)}E_{2,k})-\Id\\
&=\Pi_u(\mc D(\bar z_k)-\Id)(\Id,D_{\pi_u( \bar z_k)}E_{2,k}).
\end{align*}
we have proved in \eqref{Eq:bContEvar} that $\|\bar \B_k\|,\|\B_k\|\leq\lambda<1$. It remains to estimate the norm of the columns of $\bar \B_k-\B_k$. For all $\ell\in[L]$
\begin{align*}
\|\Col^\ell[\bar \B_k-\B_k]\|&=\|\Col^\ell [\Pi_u(\mc D(\bar z_k)-\Id)(0,D_{\pi_u( z_k)}E_{1,k}-D_{\pi_u(\bar z_k)}E_{2,k})]\|\\
&\leq\left\|\Pi_u(\mc D(\bar z_k)-\Id)|_{0\oplus\R^{M_s}}\right\|d_u(\bar V_k,V_k)\\
&\leq\mc O(\Delta^{-1} \degree)d_u(\bar V_k,V_k).
\end{align*}
where we used that, as can be easily deduced from the definition of $d_u$ in \eqref{Eq:Defd_u} of Appendix \ref{Ap:GraphTrans}, that $\|(0,D_{\pi_u( z_k)}E_{1,k}-D_{\pi_u(\bar z_k)}E_{2,k})\|=d_u(\bar V_k,V_k)$.
By Proposition \ref{Prop:EstTool} in Appendix \ref{Ap:TechComp}, we obtain
\[
\frac{\left|\Pi_u\circ D_{\bar z_k}{F_\epsilon}(\Id,D_{\pi_u( z_k)}E_{1,k})\right|}{\left|\Pi_u\circ D_{\bar z_k}{F_\epsilon}(\Id,D_{\pi_u(\bar z_k)}E_{2,k})\right|}\leq \exp\left\{\mc O(L\Delta^{-1}\degree)d_u(\bar V_k,V_k)\right\}.
\]

By Proposition \ref{Prop:ContTangentSpace} we know that, if $\beta_u$ is sufficiently small, then $d_u(\bar V_k,V_k)\leq \lambda_*d_u(V_0,W_0)$ for some $\lambda_*<1$, and this implies that

\begin{align*}
\prod_{k=0}^\infty \frac{\left|\Pi_u\circ D_{\bar z_k}{F_\epsilon}(\Id,D_{\pi_u( z_k)}E_{1,k})\right|}{\left|\Pi_u\circ D_{\bar z_k}{F_\epsilon}(\Id,D_{\pi_u(\bar z_k)}E_{2,k})\right|}&\leq \exp\left\{\mc O(L\Delta^{-1} \degree)\sum_{k=0}^{\infty}d_u(\bar V_k,V_k)\right\} \\
&\leq \exp\left\{\mc O(L\Delta^{-1}\degree)d_u(V_0,W_0)\right\}\\
&\leq \exp\left\{\mc O(L\Delta^{-1}\degree\beta_u)\right\}.
\end{align*}
\end{proof}

\subsection{Proof of Proposition \ref{Prop:BadSetMeasAtt}}

The following result shows that the set $\mc B_\epsilon^{(s,j)}\times \mb T^{M_u}\times\mc R$, which is the set where fluctuations of the dynamics of a given hub exceed a given threshold, is contained in a set, $\tilde {\mc B}_\epsilon^{(s,j)}$, that is the union of global stable manifolds. This is important to notice because, even if the product structure of the former set is not preserved taking preimages under ${F_\epsilon}$, the preimage of $\tilde {\mc B}_\epsilon^{(s,j)}$ will be again the union of global stable manifolds. Furthermore, if $\beta_s$ is sufficiently small, and this is provided by the heterogenity conditions, the set $\tilde {\mc B}_\epsilon^{(s,j)}$ will be \say{close} (topologically and with respect to the right measures) to ${\mc B}_\epsilon^{(s,j)}$.
\begin{lemma}\label{Lem:IncSetUnMan}
Consider $\mc B_\epsilon^{(s,j)}$ as in \eqref{Eq:DefBadSetComp}. Then there exists a constant $C>0$ such that
\begin{equation}\label{Eq:IncSetUnMan}
\mc B_\epsilon^{(s,j)}\times \mb T^{M_u} \times \mc R \quad \subset\quad \tilde {\mc B}_\epsilon^{(s,j)}:=\bigcup_{z\in \mc B_\epsilon^{(s,j)}\times \mb T^{M_u} \times \mc R}W^s(z)\quad\subset\quad {\mc B}_{\epsilon_1}^{(s,j)}\times \mb T^{M_u} \times \mc R ,
\end{equation}
with $\epsilon_1=\epsilon+C_\# M^{1/p}\beta_{s,p}$. 
\end{lemma}
\begin{proof}
The first inclusion is trivial. Take $z\in {\mc B}_{\epsilon}^{(s,j)}\times  \mb T^{M_u} \times \mc R$ such that 
\[
\left|\frac{1}{\Delta}\sum_i A^{hl}_{ji}\theta_{s_1} (x_i)-\kappa_j\bar\theta_s \right|\ge \epsilon|s_1|.
\]
Since $W^s(z)$ is tangent to the stable cone $\mc C^s$, by \eqref{Eq:StabCon} for any $z'\in W^s(z)$, $d_p(\pi_u(z'),z)\leq\beta_{s,p}$. This implies that 
\begin{align*}
\left|\left|\frac{1}{\Delta}\sum_iA^{hl}_{ji}\theta_{s_1} (x_i)-\kappa_j\bar\theta_s \right|-\left|\frac{1}{\Delta}\sum_iA^{hl}_{ji}\theta_{s_1}(x'_i)-\kappa_j\bar\theta_s\right|\right|&\leq \left|\frac{1}{\Delta}\sum_iA^{hl}_{ji}(\theta_{s_1} (x_i)-\theta_s(x'_i))\right|\\
&\leq\frac{1}{\Delta}\sum_iA^{hl}_{ji}|D\theta_{s_1}|d_p(\pi_u(z'),z)\\
&\leq |s_1|\mc O(M^{1/p}\beta_{s,p})
\end{align*}
proving the lemma. 
\end{proof}

\begin{proof}[Proof of Proposition \ref{Prop:BadSetMeasAtt}]

As in the proof of  Theorem \ref{Thm:PhysMeasFep}, take an embedded $L+M_u$ torus  $W_0\in\mc W_{p,K_0}$ such that $\pi_u|_{W_0}:W_0\rightarrow \mb T^{L+M_u}$ is a diffeomorphism,
a density $\rho\in \mc C_{a,p}(W_0)$ with $a>a_c$ so that $\rho\mu_W$ is a probability measure and the limit $\bar \mu$ of the sequence of measures $\{\mu_t\}_{t\in\N_0}$ defined as
\[
\mu_t:=\frac{1}{t+1}\sum_{i=0}^{t}{F_\epsilon}^i_*\mu_0.
\]
is an SRB measure.
From Lemma \ref{Lem:CovDec} we know that ${F_\epsilon}^i(W_0)=\bigcup_{k\in\mc K_i}W_{i,k}$ modulo a negligible set w.r.t. ${F_\epsilon}^i_*(\mu_0)$, and that 
\[
{F_\epsilon}^i_*(\mu_0)=\sum_{k\in\mc K_i}{F_\epsilon}^i_*\mu_0(W_{i,k})\mu_{i,k},
\]
where $\mu_{i,k}$ is a probability measure supported on $W_{i,k}$ for all $i$ and $k\in\mc K_i$. It is a consequence of Proposition \ref{Prop:DensEvSubm} that $\mu_{i,k}=\rho_{i,k}\cdot m_{W_{i,k}}$ with $\rho_{i,k}\in C_{a,p}(W_{i,k})$. For every $t\in\N_0$,
\begin{align}
\mu_t(\tilde {\mc B}_\epsilon^{(s,j)})&\leq \mu_t({\mc B}_{\epsilon_1}^{(s,j)}\times\T^{M_u}\times \mc R)\nonumber\\
&=\sum_{i=0}^{t}\sum_{k\in\mc K_i}\frac{{F_\epsilon}^i_*\mu_0(W_{i,k})}{t+1}\mu_{i,k}({\mc B}_{\epsilon_1}^{(s,j)}\times \T^{M_u}\times \mc R)\nonumber\\
&=\sum_{i=0}^{t}\sum_{k\in\mc K_i}\frac{{F_\epsilon}^i_*\mu_0(W_{i,k})}{t+1}\int_{{\mc B}_{\epsilon_1}^{(s,j)}\times \mb T^{M_u}}\rho_{i,k}d\leb_{L+M_u}\nonumber\\
&\leq\sum_{i=0}^{t}\sum_{k\in\mc K_i}\frac{{F_\epsilon}^i_*\mu_0(W_{i,k})}{t+1}\exp\left[\mc O (\Delta^{-1}L^{1+2/p}\degree^{1/q})+\mc O(M)\right]\leb_{L+M_u}({\mc B}_{\epsilon_1}^{(s,j)}\times\mb T^{M_u})\label{Eq:Ineq1MeasEst}\\
&= \exp\left[\mc O (\Delta^{-1}L^{1+2/p}\degree^{1/q})+\mc O(M)\right]\leb_{L+M_u}({\mc B}_{\epsilon_1}^{(s,j)}\times \mb T^{M_u})\label{Eq:Ineq2MeasEst}
\end{align}
Since the set $\tilde {\mc B}_\epsilon^{(s,j)}$ might not be in general measurable, in the above and in what follows we abused the notation so that whenever the measure of such set or one of its sections is computed, it should be intended as its outer measure. To prove the bound \eqref{Eq:Ineq1MeasEst} we used the fact that $\rho_{i,k}\in C_{a,p}(W_{i,k})$ for $a>a_c$ and thus its supremum is upper bounded by $\exp\left[\mc O (\Delta^{-1}L^{1+2/p}\degree^{1/q})+\mc O(ML^{1/p})\right]$. \eqref{Eq:Ineq2MeasEst}  follows from the fact that 
\begin{align}
\sum_{i=0}^{t}\sum_{k\in\mc K_i}\frac{{F_\epsilon}^i_*\mu_0(W_{i,k})}{t+1}=\sum_{i=0}^{t}\frac{{F_\epsilon}^i_*\mu_0({F_\epsilon}^i(W_0))}{t+1}=\sum_{i=0}^{t}\frac{1}{t+1}=1.\label{Eq:ProbSum1}
\end{align}
Since the bound is true for every $t\in\N_0$, then it is also true for the weak limit $\bar \mu$
\[
\bar \mu(\tilde {\mc B}_\epsilon^{(s,j)})\leq\exp\left[\mc O (\Delta^{-1}L^{1+2/p}\degree^{1/q})+\mc O(ML^{1/p})\right]\leb_{L+M_u}({\mc B}_{\epsilon_1}^{(s,j)}\times\mb T^{M_u})
\]
and since $\bar \mu$ is invariant $\forall t\in\N$
\[
\bar \mu\left({F_\epsilon}^{-t}(\tilde {\mc B}_\epsilon^{(s,j)})\right)\leq\exp\left[\mc O (\Delta^{-1}L^{1+2/p}\degree^{1/q})+\mc O(ML^{1/p})\right]\leb_{L+M_u}({\mc B}_{\epsilon_1}^{(s,j)}\times\mb T^{M_u}).
\]
From \eqref{Eq:ProbSum1} there exists $i\in\N$ and $k\in\mc K_i$ such that 
\[
\mu_{i,k}\left({F_\epsilon}^{-t}(\tilde {\mc B}_\epsilon^{(s,j)})\right)\leq \bar \mu\left({F_\epsilon}^{-t}(\tilde {\mc B}_\epsilon^{(s,j)})\right)
\]
and thus
\[
m_{W_{i,k}}\left({F_\epsilon}^{-t}(\tilde {\mc B}_\epsilon^{(s,j)})\right)\leq\exp\left[\mc O (\Delta^{-1}L^{1+2/p}\degree^{1/q})+\mc O(ML^{1/p})\right]\leb_{L+M_u}({\mc B}_{\epsilon_1}^{(s,j)}\times\mb T^{M_u}).
\]
Now, pick $y_s\in \mc R$ and consider the the holonomy map along the stable leaves $\pi:W_{i,k}\rightarrow D_{y_s}$ between transversals $W_{i,k}$ and $D_{y_s}=\mb T^{L+M_u}\times\{y_s\}\subset \mc C^u$. We know from Proposition \ref{Prop:JacEstBnd} that the Jacobian of $\pi$ is bounded by \eqref{Eq:JacEstUppBnd} and thus
\[
m_{D_{y_s}}\left({F_\epsilon}^{-t}(\tilde {\mc B}_\epsilon^{(s,j)})\right)\leq \exp\left[\mc O (\Delta^{-1}L^{1+2/p}\degree^{1/q})+\mc O(ML^{1/p})\right]\leb_{L+M_u}({\mc B}_{\epsilon_1}^{(s,j)}\times\mb T^{M_u}).
\]
The above holds for every $y_s\in \mc R$, and so by Fubini
\[
\leb_N\left({F_\epsilon}^{-t}(\tilde {\mc B}_\epsilon^{(s,j)})\right)\leq \exp\left[\mc O (\Delta^{-1}L^{1+2/p}\degree^{1/q})+\mc O(ML^{1/p})\right]\leb_{L+M_u}({\mc B}_{\epsilon_1}^{(s,j)}\times\mb T^{M_u}),
\]
and from the first inclusion in \eqref{Eq:IncSetUnMan} we obtain
\[
\leb_N\left({F_\epsilon}^{-t}({\mc B}_\epsilon^{(s,j)})\right)\leq  \exp\left[\mc O (\Delta^{-1}L^{1+2/p}\degree^{1/q})+\mc O(ML^{1/p})\right]\leb_{L+M_u}({\mc B}_{\epsilon_1}^{(s,j)}\times\mb T^{M_u}).
\]
\end{proof}

\subsection{Proof of Theorem \ref{Thm:Main}}\label{Sec:AttFullStatProof}

In this section ${F_\epsilon}:\mb T^{N}\rightarrow\mb T^{N}$ denotes again the truncated map defined on the whole phase space. Define the uncoupled map $\bo f:\mb T^{N}\rightarrow\mb T^{N}$ 
\[
\bo f(x_1,...,x_L,y_1,...,y_M):=(f(x_1),...,f(x_L),\redu_1(y_1),...,\redu_M(y_M)).
\]
The next lemma evaluates the ratios of the Jacobians of ${F_\epsilon}^t$ and $\bo f^t$ for any fixed $t\in\N$.
\begin{lemma}\label{Lem:EvVolCoupUnc}
\[
\frac{|D_z\bo f^t|}{|D_z{F_\epsilon}^t|}\leq \exp\left[\mc O(ML\Delta^{-1}\degree)+\mc O(M^2)\right]
\]
\end{lemma}
\begin{proof}
For all $i\in [t]$ define $z_i:=\bo f^i(z)$, $\bar z_i:={F_\epsilon}^i(z)$, and $z_0=\bar z_0:=z$.
\begin{align*}
\frac{|D_z\bo f^t|}{|D_z{F_\epsilon}^t|}=\frac{\prod_{k=0}^t|D_{z_k}\bo f|}{\prod_{k=0}^t|D_{\bar z_k}{F_\epsilon}|}=\prod_{k=0}^t\frac{\sigma^L\prod_{m=1}^MD_{y_{k,m}}\redu_m}{|D_{\bar z_k}{F_\epsilon}|}&=\prod_{k=0}^t\frac{\sigma^L\prod_{m=1}^MD_{\bar y_{k,m}}\redu_m\left(1+\frac{D_{y_{k,m}}\redu_m-D_{\bar y_{k,m}}\redu_m}{D_{\bar y_{k,m}}\redu_m}\right)}{|D_{\bar z_k}{F_\epsilon}|}\\
&\leq \exp\left[\mc O(M)\right] \prod_{k=0}^t\frac{1}{|\mc D(\bar z_k)|}
\end{align*}
where $\mc D(z_i)$ is defined as in \eqref{Eq:DExpression}. $1/|\mc D(\bar z_i)|$ can be estimated in the usual way defining $\B(z_i):=\mc D(z_i)-\Id$ and noticing that $1=|\Id+0|$. One can obtain, from the computations leading to \eqref{Eq:RatDestimate}, that
\[
\frac{1}{|\mc D(\bar z_i)|}\leq \exp\left[\sum_{k=1}^{N}\Col^k[\B(z_i)]\right]\leq \exp\left[\mc O(L\Delta^{-1}\degree)+\mc O(M)\right].
\]
and thus
\[
\frac{|D_z\bo f^t|}{|D_z{F_\epsilon}^t|}\leq \exp\left[\mc O(M)\right]\exp\left[\mc O(L\Delta^{-1}\degree)+\mc O(M)\right]\leq\exp\left[\mc O(L\Delta^{-1}\degree)+\mc O(M)\right].
\]
\end{proof}

\begin{lemma}\label{Lem:OneDimDyn}
Consider the set $A^2(\mb T,\mb T)$ of $C^2$ Axiom A endomorphisms on $\mb T$ endowed with the $C^1$ topology. Take a continuous curve $\gamma:[\alpha_1,\alpha_2]\rightarrow A^2(\mb T,\mb T)$. Then, denoting by $\Lambda^\alpha$ and $\Upsilon^\alpha$ respectively the attractor and repellor of $\gamma_\alpha$ for all $\alpha\in[\alpha_1,\alpha_2]$,
\begin{itemize}
\item[(i)] there exist uniform $\neig>0$ and $\lambda\in(0,1)$ such that
\begin{equation}\label{Eq:UnifParaAxCur}
\left|\gamma_\alpha'|_{\Lambda^\alpha_\neig}\right|<\lambda\quad\mbox{and}\quad\left|\gamma_\alpha'|_{\Upsilon^\alpha_\neig}\right|>\lambda^{-1},
\end{equation}
\item[(ii)] there are uniform $r>0$ and $\tau \in\N$ such that for all $\alpha\in[\alpha_1,\alpha_2]$, all sequences $\{\epsilon_i\}_{i=0}^{\tau-1}$ with $\epsilon_i\in(-r,r)$ and all points $x\in\mb T\backslash{\Upsilon^{\alpha}_{\neig}}$, the orbit $\{x_i\}_{i=0}^{\tau}$ defined 
\begin{equation}\label{Eq:DefOrbRand}
x_0:=x\quad\mbox{and}\quad x_i:=\gamma_\alpha(x_{i-1})+\epsilon_i,
\end{equation}
satisfies $x_\tau\in \Lambda^\alpha_\neig$.
\end{itemize}  
\end{lemma}
\begin{proof}
The above lemma is quite standard \cite{MS} and can be easily proved by considering the sets
\[
\bigcup_{\alpha\in [\alpha_1,\alpha_2]}\{\alpha\}\times\Lambda_\alpha\subset[\alpha_1,\alpha_2]\times\mb T\quad \mbox{and}\quad\bigcup_{\alpha\in[\alpha_1,\alpha_2]}\{\alpha\}\times\Upsilon_\alpha\subset[\alpha_1,\alpha_2]\times \mb T
\]
and noticing that they are compact. Then from the $C^1$ assumption on the axiom A map, it follows that all the stated quantities are uniformly bounded. 

\end{proof}

\begin{proof}[Proof of Theorem \ref{Thm:Main}]
 
\textbf{Step 1} Restricting ${F_\epsilon}$ to $\mc S$, we can use Proposition \ref{Prop:BadSetMeasAtt} to get an estimate of the Lebesgue measure of  ${{\mc B}_{\epsilon,T}\times \mb T^{M_u}\times \mc R}$. Define 
\[
{\mc B}_{\epsilon, T,\tau}:=\bigcup_{t=0}^{\tau}{F_\epsilon}^{-t}\left({\mc B}_{\epsilon,T}\times\mb T^{M_u}\times \mc R\right)\cap \mc S.
\] 
To determine the Lebesgue measure of this set we compare it with the Lebesgue measure of
\[
{\mc B}'_{\epsilon, T,\tau}:=\bigcup_{t=0}^{\tau}\bo f^{-t}\left({\mc B}_{\epsilon,T}\times\mb T^{M_u}\times \mc R\right)\cap\mc S.
\] 

For all $y\in \mb T^M$, the map $\bo f|_{\mb T^{L}\times\{y\}}:\mb T^{L}\times\{y\}\rightarrow \mb T^{L}\times\{(\redu_1(y_1),...,\redu_M(y_M))\}$ is an expanding map with constant Jacobian and thus measure preserving if we endow $\mb T^{L}\times\{y\}$ and $ \mb T^{L}\times\{(\redu_1(y_1),...,\redu_M(y_M))\}$ with the induced Lebesgue measure. Fubini's theorem implies that for all $t\in[\tau]$
\[
{\leb_N(\bo f^{-t}({\mc B}_{\epsilon,T}\times\mb T^{M_u}\times \mc R)\cap \mc S)}\leq C(\tau) \frac{\leb_N({\mc B}_{\epsilon,T}\times\mb T^{M_u}\times \mc R)}{\leb_N(\mc S)}
\]
where $C(\tau)$ is a constant depending on $\tau$ and uniform on the network parameters.
And thus
\begin{align*}
\leb_N(\bo f^{-t}({\mc B}_{\epsilon,T}\times\mb T^{M_u}\times \mc R)\cap \mc S)\leq C_\#\leb_N({\mc B}_{\epsilon,T}\times\mb T^{M_u}\times \mc R).
\end{align*}
Now
\[
\leb_N\left({F_\epsilon}^{-t}({\mc B}_{\epsilon,T}\times\mb T^{M_u}\times \mc R)\cap \mc S\right)\leq\leb_N\left(\bo f^{-t}({\mc B}_{\epsilon,T}\times\mb T^{M_u}\times \mc R)\cap \mc S\right)\sup_{z\in\mb T^{N}}\frac{|D_z\bo f^t|}{|D_z{F_\epsilon}^t|}.
\]
By Lemma \ref{Lem:EvVolCoupUnc}, assuming that $\tau\leq T$, we get
\begin{align*}
\leb_N({\mc B}_{\epsilon,T,\tau})&\leq \sum_{t=0}^{\tau} \exp\left[\mc O(L\degree\Delta^{-1})+\mc O(M)\right] C_\#\leb_N({\mc B}_{\epsilon,T}\times\mb T^{M_u}\times \mc R)\\
&\leq T\exp\left[-\mc O(\Delta^{-1})\epsilon^2+\mc O(\Delta^{-1}L^{1+2/p}\degree)+\mc O(ML^{1/p})\right].
\end{align*}
\textbf{Step 2} Define the set $\mc U\subset \mb T^{N}$ as
\[
\mc U:=\mb T^{L+M_u}\times \Upsilon^{M_u+1}_{\neig}\times...\times\Upsilon^{M}_{\neig},
\]
Consider the system $\auxex:\mb T^{N}\rightarrow\mb T^{N}$ obtained redefining ${F_\epsilon}$ on $\mc U^c$ so that if $(x',y')=G(x,y)$, $\pi_u\circ G(x,y)=\pi_u\circ {F_\epsilon}(x,y)$ (the evolution of the \say{expanding} coordinates is unvaried) and 
\[
y'_j=\hat f_j(y_j)+\alpha\sum_pg\left (\frac{1}{\Delta}\sum_nA^{hl}_{jn}\theta_s(x_n)-\kappa_j\bar\theta_s\right)\upsilon_s(y_j)+\frac{\alpha}{\Delta}\sum_{m=1}^MD_{jm}h(y_j,y_m)\mod 1
\] 
where the reduced dynamics is (smoothly) modified to be globally expanding by putting $\hat f_j|_{ \Upsilon^{j}_{\neig}}:=\redu_j|_{ \Upsilon^{j}_{\neig}}$ and $\hat f_j|_{\mb T\backslash \Upsilon^{j}_{\neig}}$ redefined so that $|\hat f_j|\geq\lambda^{-1}>1$ everywhere on $\mb T$. Evidently $G|_{\mc U}={F_\epsilon}|_{\mc U}$. We can then invoke the results of Section \ref{Sec:ExpRedMapsGlob} to impose conditions on $\eta$ and $\epsilon$ to deduce global expansion of the map $G$ (under suitable heterogeneity hypotheses) and  the bounds on the invariant density obtained in that section. In particular one has that for al $T\in \N$
\[
\leb_N\left(\bigcup_{t=0}^T G^{-t}({\mc B}_\epsilon\times \mb T^{M})\right)\leq T\exp\left\{-\Delta\epsilon^2/2+\mc O(\Delta^{-1}N^{1+2/p}\degree^{1/q})+\mc O(MN^{1/p})\right\}
\]
and this implies
\[
\leb_N\left(\bigcup_{t=0}^T G^{-t}({\mc B}_\epsilon\times\mb T^{M})\bigcup {\mc B}_{\epsilon,\tau,T}\right)\leq 2T\exp\left\{-\Delta\epsilon^2/2+\mc O(\Delta^{-1}L^{1+2/p}\degree^{1/q})+\mc O(ML^{1/p})\right\}.
\]
And this concludes the proof of the theorem. 
\end{proof}

\subsection{Mather's Trick and proof of Theorem~\ref{Thm:Main} when $n\neq 1$ } 
\label{subsec:mather}
Until now we have assumed that the reduced maps $\redu_j$ satisfied Definition~\ref{Def:AxiomA} with $n=1$. We now show that any $n\in\N$ will work by constructing an adapted metric via what is known as \say{Mather's trick} (see Lemma 1.3 in Chapter 3 of \cite{MS} or  \cite{hirsch2006invariant}). 


\begin{lemma} It is enough to prove Theorem~\ref{Thm:Main}  for $n=1$. \label{lem:n=1}
\end{lemma}
\begin{proof} 
Assume that $g_j$, $j=1,\dots,M$ satisfies the assumptions in Definition~\ref{Def:AxiomA} for some $(n,m,\lambda,r)$. 
Condition (2) and (3) imply that one can smoothly 
conjugate each of the maps $\redu_j$ so that $|D_x\redu_j|<\lambda$ for all $x\in N_r(\Lambda_j)$,
and $|D_x\redu_j|>\lambda^{-1}$ for all $x\in N_r(\Upsilon_j)$. These conjugations are obtained by changing the metric 
(\say{Mather's trick}).
In other words,  there exists a smooth coordinate change $\phi_j\colon \mb T \to \mb T$ so that for $\tilde \redu_j:=\phi_j \circ \redu_j \circ \phi_j^{-1}$  the properties from Definition~\ref{Def:AxiomA} hold for $n=1$.
Moreover, there exists some uniform constant $C_\#$ only depending on $(n,\lambda,r)$,  so that the $C^2$ norms of $\phi_j$ and $\phi_j^{-1}$ are bounded by $C_\#$. 
Writing $\tilde y_j=\phi_j(y_j)$ and $\tilde z=(\tilde z_1,\dots,\tilde z_n)=(x_1,\dots,x_L,\tilde y_1,\dots,\tilde y_M)$, in these new coordinates
  \eqref{Eq:CoupDyn1}-\eqref{Eq:average} become 
\begin{align}
x_i'&= f(x_i)+\frac{\alpha}{\Delta}\sum_{\ell=1}^L A^{ll}_{i\ell}h(x_i,x_\ell)+\frac{\alpha}{\Delta}\sum_{m=1}^M A^{lh}_{im}\tilde h(x_i,\tilde y_m)  \mod 1& i=1,...,L\label{Eq:CoupDyn1"}\\
\tilde y_j'&=\tilde \redu_j (\tilde y_j)+\tilde \xi_{j}(\tilde z)  \quad \,\, \mod 1& j=1,...,M\label{Eq:CoupDyn2"}
\end{align}
where 
\begin{equation}
\tilde \xi_{j}(z):=\int_{\tilde g_j(\tilde y_j)}^{\tilde g_j(\tilde y_j)+\xi_j}D_t\phi_jdt,\quad\quad\mbox{and} \quad\quad \tilde h(x,\tilde y):=h(x,\phi_j^{-1}(\tilde y)).
 \label{eq:xijeps''}
\end{equation}
In fact 
\begin{align*}
\tilde y_j'=\phi_j(y_j)=\phi_j\left(g_j(y_j)+\xi_{j}\right)=\tilde g_j(\tilde y_j)+\int_{\tilde g_j(\tilde y_j)}^{\tilde g_j(\tilde y_j)+\xi_j}D_t\phi_jdt.
\end{align*}
Then we can define $\tilde \xi_{j,\epsilon}$ as
\[
\tilde \xi_{j,\epsilon}:= \int_{\tilde g_j(\tilde y_j)}^{\tilde g_j(\tilde y_j)+\xi_{j,\epsilon}}D_t\phi_jdt
\]
and define the truncated system as
\begin{align}
x_i'&= f(x_i)+\frac{\alpha}{\Delta}\sum_{\ell=1}^L A^{ll}_{i\ell}h(x_i,x_\ell)+\frac{\alpha}{\Delta}\sum_{m=1}^M A^{lh}_{im}\tilde h(x_i,\tilde y_m)  \mod 1& i=1,...,L\\
\tilde y_j'&=\tilde \redu_j (\tilde y_j)+\tilde \xi_{j,\epsilon}(\tilde z)  \quad \,\, \mod 1& j=1,...,M. 
\end{align}
Since the $\phi_j$ are $C^2$ with uniformly bounded $C^2$ norm, it immediately follows that $\tilde \xi_{j,\epsilon}$ satisfies all the properties satisfied by $\xi_{j,\epsilon}$ listed in Lemma~\ref{Lem:XiProp}.
Assuming that $|\tilde \xi_{j,\epsilon}(\tilde x(t))|\le \xi$ for all $0\le t\le T$, we immediately obtain that 
$$|y_j'- \redu_j(y_j)| =\left |\phi_j^{-1} [ \tilde \redu_j (\phi_j ( y_j))+\tilde \xi_{j,\epsilon}(\tilde z)] - \redu_j(y_j)\right|\le \mc O(\tilde x_{j,\epsilon}(\tilde z)) \le \mc O(\xi).$$
\end{proof} 

\subsection{Persistence of the Result Under Perturbations}

The picture presented in Theorem~\ref{Thm:Main} is persistent under smooth random perturbations of the coordinates. Suppose that instead of the deterministic dynamical system $F:\mb T^N\rightarrow \mb T^N$ we have a stationary Markov chain $\{\mc F_t\}_{t\in\N}$ on some probability space $(\Omega,\mb P)$ with transition kernel 
\[
\mb P (\mc F_{n+1}\in A|\mc F_n=z):=\int_A\phi(y-F(z))dy
\]
 where $\phi:\mb T^N\rightarrow\R^+$ is a density function. The Markov chain describes a random dynamical system where independent random noise distributed according to the density $\phi$ is added to the iteration of  $F$. Take now the stationary Markov chain $\{\mc F_{\epsilon,t}\}_{t\in\N}$ defined by the transition kernel 
\[
\mb P (\mc F_{\epsilon,n+1}\in A|\mc F_{\epsilon,n}=z):=\int_A\phi(y-F_\epsilon(z))dy
\]
where we consider the truncated system instead of the original map in the deterministic drift of the process and restrict, for example, to the case where $F_\epsilon$ is uniformly expanding. The associated transfer operator can be written as $\mc P_\epsilon=P_\phi\circ P_\epsilon$  where $P_\epsilon$ is the transfer operator for $F_\epsilon$ and 
  \[
 (P_\phi\rho)(x)=\int\rho(y)\phi(y-x)dy.
 \]
Let $C_{a,p}$ be a cone of densities invariant under $P_\epsilon$ as prescribed in Proposition~\ref{Prop:ConInv}. It is easy to see that this is also invariant under $P_\phi$ and thus under $\mc P_\epsilon$. In fact, take $\rho\in C_{a,p}$. Then
\begin{align*}
\frac{(P_{\phi}\rho)(z)}{(P_{\phi}\rho)(\bar z)}=\frac{\int\rho(y)\phi(y-z)dy}{\int\rho(y)\phi(y-\bar z)dy}=\frac{\int\rho(z-y)\phi(y)dy}{\int\rho(\bar z-y)\phi(y)dy}\leq\frac{\int\rho(\bar z-y)\exp\{a d_p(z,\bar z)\}\phi(y)dy}{\int\rho(\bar z-y)\phi(y)dy}=\exp\{a d_p(z,\bar z)\}.
\end{align*} 
 This means that there exists a stationary measure for the chain with density belonging to $C_{a,p}$ and that the same estimates we have in Section~\ref{Sec:ExpRedMapsGlob} for the measure of the set $\mc B_\epsilon$ hold. This allows to conclude that the hitting times to the set $\mc B_\epsilon$ satisfy the same type of bound in the proof of Theorem~\ref{Thm:Main}.  Notice the independence of the above on the choice of density $\phi$ for the noise. This implies that all the arguments continue to hold independently on the size of the noise which, however, contribute to spoil the low-dimensional approximation for the hubs in that it randomly perturbs it.

\section{Conclusions and Further Developments}\label{Sec:Conc} 

Heterogeneously Coupled Maps (HCM) are ubiquitous in  applications. Because of the heterogeneous structure and lack of symmetries in the graph, most previously available results and techniques cannot be directly applied to this situation. Even if the behaviour of the local maps is well understood, once they are coupled in a large network, a rigorous description of the system becomes a major challenge, and numerical simulations are used to obtain information on the dynamics. 

The ergodic theory of high dimensional systems presents many  difficulties including 
the choice of reference measure and dependence of decay of correlations on the system size. 
We exploited the heterogeneity to obtain rigorous results for HCM. 
Using an ergodic description, the dynamics of hubs can be predicted from knowledge of the local dynamics and coupling function. 
This  makes it possible to obtain quantitative theoretical predictions of the behaviour of complex networks.  
Thereby, we establish that existence of a range of dynamical phenomena quite different from the ones encountered in homogenous networks. This highlights the need of new paradigms when dealing with high-dimensional dynamical systems with a heterogenous coupling. 

\medskip
\noindent
{\bf Synchronization occurs through a heat bath mechanism.}
For certain coupling functions, hubs can synchronize, unlike poorly connected nodes which remain out of synchrony. The underlying synchronization process is not 
related to direct coupling between hubs, but comes via the coupling  with a  poorly connected nodes. 
So the hub {\em synchronization process} is through a mean-field effect (i.e. the  coupling is {\em through a  \say{heat bath}}).
In HCM synchronization depends on the connectivity layer  (see Subsection~\ref{subsec:predictions+experiments}). We highlighted this feature in the networks three types of hubs having distinct degrees.

\medskip 
\noindent
{\bf Synchronization in random networks - HCM versus Homogeneous.} 
Theorem~\ref{MTheo:C} shows that synchronization occurs  in random homogeneous networks,
but is rare in HCM (see Appendix~\ref{App:RandGrap}).
Recent work (for example \cite{Gol-Stewart98}) showed that structure influences dynamics. What Theorem B shows that it is not strict symmetry, but (probablisitic) homogeneity that makes
synchronization possible. In contrast, in presence of heterogeneity the dynamics changes according to connectivity layers.

\medskip
\noindent
{\bf Importance of Long Transients in High Dimensional Systems.}
Section \ref{Sec:StarNetExamp} shows how certain behaviour can be sustained by a system only for finite time $T$ and, as it turns out,  $T$ is exponentially large in terms of the size of the network being greater than any feasible observation time. 
The issue of  such long transient times, naturally arises in high dimensional systems. 
For example,  given an $N-$fold product of the same expanding map $f$, densities evolve asymptotically to the unique SRB measure exponentially fast, but the rate depends on the dimension and becomes very low for $N\rightarrow \infty$. Take $f$ an expanding map and define
\[
\bo f:=f\times\dots\times f.
\]
Suppose $\nu$ is the  invariant measure for $f$ absolutely continuous with respect to some reference measure $m$ different from $\nu$. Then the push forward $\bo f^t_*(m^{\otimes N})=(f_*^tm)^{\otimes N}$ will converge exponentially fast in some suitable product norm to $\nu^{\otimes N}$ because this is true for each factor separately. However, choosing a large $N$, the rate can be made arbitrarily slow and in the limit of infinite $N$, $\bo f^t_*(m^{\N})$ and $\nu^\N$ are singular for all $t\in\N$. This means that in practice, pushing forward with the dynamics an absolutely continuous initial measure, it might take a very long time before relaxing to the SRB measure even if the system is hyperbolic. This suggests that in order to accurately describe HCM and high-dimensional systems, it is necessary to understand the {\em dependence} of all relevant quantities and bounds {\em on the dimension}. This is often disregarded in the classical literature on ergodic theory.

\subsection{Open problems and new research directions}
With regard to HCM some problems remain open.

\begin{enumerate}
\item 
In Theorem~\ref{Thm:Main} we assumed that the local map $f$ in our model is Bernoulli and that all the non-linearity within the model is
contained in the coupling. This assumption makes it easier to control distortion estimates as the dimension of the network increases. For example, without this assumption, the density 
of the invariant measure in the expanding case, see Section~\ref{Sec:ExpRedMapsGlob}, 
becomes highly irregular as the dimension increases.

\noindent
\textbf{Problem:} \emph{obtain the results in Theorem~\ref{Thm:Main} when $f$ is a general uniformly expanding circle map in $C^{1+\nu}$ with $\nu\in(0,1)$.}
\item
In Theorem \ref{Thm:Main} we gave a description of orbits for finite time until they hit the set $\mc B_\epsilon$ where the fluctuations are above the threshold and the truncated system $F_\epsilon$ differs from the map $F$.

\noindent
\textbf{Problem:} \emph{describe what happens after an orbit enters the set $\mc B_\epsilon$. In particular, find how much time orbits need to escape $\mc B_\epsilon$ and how long it takes for them to return to this set. } 
\noindent

\item 
In the proof of Theorems~\ref{Thm:Main} and  \ref{MTheo:C}   we assume that the reduced dynamics $g_j$, in Eq. \eqref{Eq:RedEqInt}, of each hub node $j\in\{1,...,M\}$ is uniformly hyperbolic.

\noindent
\textbf{Problem:} \emph{find a sufficiently weak argument that allows to describe the case where some of the reduced maps $g_j$ have non-uniformly hyperbolic behaviour, for example, when they have a neutral fixed point.}

In fact, hyperbolicity is a generic condition in dimension one, see \cite{MR3336841,MR2342693}
but not in higher dimensions. An answer to this question would be desirable even in the 
one-dimensional case, but especially when treating multi-dimensional HCM. 
\end{enumerate}

\noindent
The study of HCM and the approach used in this paper also raise more general questions such as:

\begin{enumerate}
\item 
  
\textbf{Problem:} \emph{is the SRB measure supported on the attractor of $F_\epsilon$ absolutely continuous with respect to the Lebesgue measure $m_N$ on the whole space?}

Tsujii proved in \cite{MR1862809} absolute continuity 
of the SRB measure for a non-invertible two-dimensional skew product system. 
Here the main challenges are that the system does not have a skew-product structure, and the perturbation with respect to the product system depend on the dimension. 

\item 
Chimera states refer to \say{heterogeneous} behaviour observed (in simulations and experiments) on homogeneous networks, see \cite{abrams2004chimera}. The emergence of such states is not yet completely understood, but they are widely believed to be associated to long transients. 

\textbf{Problem:} \emph{does the approach of the truncated system shed light on Chimera states?}
\end{enumerate}

\begin{appendices}

\section{Estimates on the Truncated System}\label{App:TruncSyst}

\begin{Hoefd} 
Suppose that $(X_i)_{i\in\N}$ is a sequence of bounded independent random variables on a probability space $(\Omega,\Sigma,\mb P)$, and suppose that there exists $a_i<b_i$ such that $X_i(\omega)\in[a_i,b_i]$ for all $\omega\in\Omega$, then 
$$
\mb P\left(\left|\frac{1}{n}\sum_{i=1}^{n}X_i-\mb E_{\mb P}\left[\frac{1}{n}\sum_{i=1}^{n}X_i\right]\right|\geq \tea \right)\leq2\exp\left[-\frac{2n^2\tea^2}{\sum_{i=1}^n(b_i-a_i)^2}\right]
$$
for all $\tea>0$ and $n\in\N$.
\end{Hoefd}

\begin{proof}[Proof of Proposition \ref{Prop:UppBndBLeb}]
 Hoeffding's inequality can be directly applied to the random variables defined on $(\mb T^{L},\mc B, \leb_L)$ by $X_i:=\theta_{s_1}\circ \pi^i(x)$ where $\pi^i:\mb T^{L}\rightarrow \mb T$ is the projection on the $i$-th coordinate ($1\leq i\leq L$). These are in fact independent by construction and bounded since $\{\theta_{s_1}\}_{s_1\in\mb Z}$ are trigonometric functions on $[-1,1]$.

Consider the set
\[
\mc B_\epsilon=\bigcup_{j=1}^{M}\bigcup_{s_1\in\mb Z}\mc B_\epsilon^{(s_1,j)}.
\]
$B_\epsilon^{(s_1,j)}$ defined as in \eqref{Eq:DefBadSetComp}. Notice that for $s_1=0$, $\mc B_\epsilon^{(0,j)}=\mb T^L$. 
Since $\kappa_j\Delta=d_j$, we can rewrite $\mc B_\epsilon^{(s_1,j)}$ as
\[
\mc B_\epsilon^{(s_1,j)}=\left\{x\in\mb T^{L}:\left|\frac{1}{d_j}\sum_{i=1}^L  A_{ji}\theta_{s_1}(x_i)-\bar\theta_{s_1}\right|>\frac{\epsilon}{\kappa_j}|s_1| \right\}.
\]
Since  $d_j$ is number of non-vanishing terms in the sum, the above is the measurable set where the empirical average over $d_j$ i.i.d bounded random variables, exceeds their common expectation of more than $\epsilon|s_1| /\kappa_j$. Being under the hypotheses of the above Hoeffding Inequality, we can estimate the measure of the set as
\begin{equation}
\leb_L(\mc B_\epsilon^{(s_1,j)})\leq 2\exp\left[-\frac{d_j^2\epsilon^2|s_1|^{2}}{d_j2\kappa_j^2}\right] = 2\exp\left[-\frac{\Delta\epsilon^2|s_1|^{2}}{2\kappa_j}\right]
\label{Beps}
\end{equation}
and this gives
\[
\leb_L(\mc B_\epsilon)\leq\sum_{j=1}^{M}\sum_{s_1\in\mb Z\backslash\{0\}} \leb_{L}(\mc B_\epsilon^{(s_1,j)})\leq 2M\sum_{s_1\in\mb Z\backslash\{0\}}\exp\left[-\frac{\Delta\epsilon^2}{2}|s_1|\right]\leq 4M\frac{\exp\left[-\frac{\Delta\epsilon^2}{2}\right]}{1-\exp\left[-\frac{\Delta\epsilon^2}{2}\right]}
\]
since $\kappa_j<1$, which concludes the proof of  the proposition.
\end{proof}

Here follows an expression for  $DF_\epsilon$. Using  \eqref{Eq:CoupDyn1'} and  \eqref{Eq:CoupDyn2'} and writing as before $z=(x,y)$,
 noting that for $k>L$, $z_k=y_{k-L}$ we get  
\begin{equation}\label{Eq:DiffAuxMap}
[D_{(x,y)}F_\epsilon]_{k\ell}=\left\{\begin{array}{lr}
D_{x_k} f+\frac{\alpha}{\Delta}\sum_{n=1}^N A_{kn}h_1(x_k,z_n) & k=\ell \leq L,\\
\frac{\alpha}{\Delta} A_{k\ell}h_2(x_k,z_\ell)& k\neq \ell, k \leq L,\\
\partial_{x_\ell}\xi_{k-L,\epsilon}& k>L, \ell\leq L\\
\frac{\alpha}{\Delta}A_{k\ell}h_2(y_{k-L},y_{\ell-L})&k\neq \ell > L,\\
D_{y_{k-L}} \redu_{k-L}+\partial_{y_{k-L}}\xi_{k-L,\epsilon}&k= \ell>L. \\
\end{array}\right.
\end{equation}
Here $h_1$ and $h_2$ stand for the partial derivatives of the function $h$ with respect to the first and second coordinate respectively, and where we suppressed, not to additionally cloud the notation some of the functional dependences.

The following lemma summarises the properties $\xi_{j,\epsilon}$ satisfies and that will yield good hyperbolic properties for $F_\epsilon$. 
\begin{lemma}\label{Lem:XiProp}
The functions $\xi_{j,\epsilon}:\mb T^N\rightarrow \R$ defined in Eq. \eqref{eq:xijeps} satisfy
\begin{itemize}\item[(i)] \[
|\xi_{j,\epsilon}|\leq  C_\#(\epsilon+\Delta^{-1}M)
\]
where $C_\#$ is a constant depending only on $\sigma$, $h$, and $\alpha$.
\item[(ii)]
\begin{align*}
\left|\partial_{z_n}\xi_{j,\epsilon}\right|\leq\left\{ \begin{array}{cr}
\mc O(\Delta^{-1})A_{jn} &n\leq L\\
 C_\#\epsilon+\mc O(\Delta^{-1}M)& n=j+L\\
 \mc O(\Delta^{-1})A_{jn}&n>L,\mbox{ }n\neq j+L. 
\end{array}\right.
\end{align*}
\item[(iii)] for all $z,\bar z\in\mb T^N$
\begin{align*}
\left|\partial_{z_n}\xi_{j,\epsilon}(z)-\partial_{z_n}\xi_{j,\epsilon}(\bar z)\right|\leq\left\{ \begin{array}{cr}
\mc O(\Delta^{-1})A_{jn}d_{\infty}(z,\bar z) &n\leq L\\
 \mc O(1)+\mc O(\Delta^{-1}M)d_{\infty}(z,\bar z)& n=j+L\\
 \mc O(\Delta^{-1})A_{jn}d_{\infty}(z,\bar z)&n>L,\mbox{ }n\neq j+L. 
\end{array}\right.
\end{align*}
\end{itemize}
\end{lemma}
\begin{proof}
Proof of (i) follows from the following estimates
\begin{align*}
|\xi_{j,\epsilon}|&\leq C_\#\left(\sum_{s\in\mb Z^2}c_s\epsilon|s_1| +\Delta^{-1}M\right)\\
&\leq C_\#(\epsilon+\Delta^{-1}M)
\end{align*}
where we used that the sum is absolutely convergent.
\noindent
To prove $(ii)$ notice that for $n\leq L$
\begin{align}\label{Eq:PArtDerXi}
\partial_{z_n}\xi_{j,\epsilon}(z)=\partial_{x_n}\xi_{j,\epsilon}(z)=\alpha\sum_{s\in\mb Z^2}c_s D_{(\cdot)}\zeta_{\epsilon|s_1| }\frac{A_{jn}}{\Delta}D_{x_n}\theta_{s_1}
\end{align}
and $|D_{x_n}\theta_{s_1}|\leq 2\pi |s_1|$, so the bound follows from the fast decay rate of the Fourier coefficients. For $n=j+L$ 
\begin{align*}
|\partial_{z_{j+L}}\xi_{j,\epsilon}(z)|=|\partial_{y_j}\xi_{j,\epsilon}(z)|&=\left|\alpha\sum_{s\in\mb Z^2}c_s \zeta_{\epsilon |s_1|}\left(\cdot \right)D_{y_j}\upsilon_{s_2}+\sum_{n=1}^{M}\frac{\alpha}{\Delta} A^h_{jn}\partial_{y_j}h(y_j,y_n)\right|\\
&\leq \epsilon C_\# \sum_{s\in\mb Z^2}|c_s||D_{y_j}\theta_{s_2}||s_1| +\mc O(\Delta^{-1}M).
\end{align*}
 Again the decay of the Fourier coefficients yields the desired bound. For $n>L$ and different from $j+L$ it is trivial. Point (iii) for $n\neq L+j$ follows immediately  from expression \eqref{Eq:PArtDerXi} and by the decay of the Fourier coefficients. For $n=j+L$
 \begin{align*}
 |\partial_{y_j}\xi_{j,\epsilon}(z)-\partial_{y_j}\xi_{j,\epsilon}(\bar z)|&\leq\left|\alpha\sum_{s\in\mb Z^2}c_s \epsilon |s_1|  \left[D_{y_j}\upsilon_{s_2}-D_{\bar y_j}\upsilon_{s_2}\right]+\sum_{n=1}^{M}\frac{\alpha}{\Delta} A^h_{jn}\left[\partial_{y_j}h(y_j,y_n)-\partial_{y_j}h(\bar y_j,\bar y_n)\right]\right|\\
 &\leq \mc O(1+\Delta^{-1}M)d_{\infty }(z,\bar z).
 \end{align*}
Notice that to obtain the last step we need the sequence $\{c_{\bo s}|s_1||s_2|^3\}$ to be summable. In particular, 
 \[
 c_{\bo s}\leq \frac{c_\#}{|s_1|^{2+b}|s_2|^{4+b}},\quad b>0
 \] 
 is a sufficient condition, ensured by picking $h\in C^{10}$.
\end{proof}

\section{Estimate on Ratios of Determinants}\label{Ap:TechComp}

In the following proposition $\Col^k[M]$, with $M\in\mc M_{n\times n}$ a square matrix of dimension $n$, is the $k-$th column of the matrix $M$.   

\begin{proposition}\label{Prop:EstTool}
Suppose that $\|\cdot\|_p:\R^n\rightarrow\R^+$ is the $p$ norm ($1\leq p\leq \infty$) on the euclidean space $\R^n$. Take two square matrices $b_1$ and $b_2$ of dimension $n$. Suppose there is constant $\lambda\in(0,1)$ and such that

\begin{equation}\label{Eq:Cond1Appb}
\|b_i\|_p:=\sup_{\substack{v\in\mb R^n\\ \|v\|_p\leq 1}}\frac{\|b_iv\|_p}{\|v\|_p}\leq \lambda \quad \forall i\in \{1,2\},
\end{equation}
Then
\[
\frac{|\Id+b_1|}{|\Id+b_2|}\leq \exp\left\{\frac{\sum_{k=1}^{n}\|\Col^k [b_1-b_2]\|_p}{1+\lambda}\right\}.
\]
\end{proposition}
\begin{proof}
Given a matrix $M\in\mc M_{n\times n}$ it is a standard formula that 
\[
|M|=\exp[\Tr \log(M)].
\]
\begin{align*}
\frac{|\Id+b_1|}{|\Id+b_2|}&=\exp\left\{\sum_{\ell=1}^\infty\frac{(-1)^{\ell+1}}{\ell} \Tr[b_1^\ell-b_2^\ell]\right\}.
\end{align*}
Substituting the expression
$$
b_1^\ell- b_2^\ell=\sum_{j=0}^{\ell-1}b_1^j(b_1-b_2)b_2^{\ell-j-1}
$$
we obtain
\begin{align}
\Tr(b_1^\ell-b_2^\ell)&= \sum_{i=0}^{\ell - 1}\Tr(b_1^j(b_1-b_2)b_2^{\ell-j-1})\nonumber\\
&=\sum_{j=0}^{\ell-1}\Tr(b_2^{\ell-j-1}b_1^n( b_1-b_2))\nonumber\\
&\leq \sum_{j=0}^{\ell-1}\sum_{k=1}^{n} \|\Col^k[b_2^{\ell-j-1}b_1^j( b_1-b_2)]\|\label{passone}
\end{align}
where we used that the trace of a matrix is upper bounded by the sum of the $p-$norms of its columns (for any $p\in[1,\infty]$). Using conditions \eqref{Eq:Cond1Appb} we obtain
\begin{align}
\Tr(b_1^\ell-b_2^\ell)&\leq \ell \lambda^{\ell-1} \sum_{k=1}^{n}\|\Col^k[b_1- b_2]\|_p. \label{passotwo} 
\end{align}
To conclude
\begin{align*}
\frac{|\Id+b_1|}{|\Id+b_2|}&\leq\exp\left\{\sum_{\ell=1}^\infty(-1)^{\ell+1}\lambda^{\ell-1}\sum_{k=1}^{n}\|\Col^k[b_1- b_2]\|_p\}\right\}	\\	
&=\exp\left\{\frac{\sum_{k=1}^{n}\|\Col^k[b_1- b_2]\|_p}{1+\lambda}\right\}.
\end{align*}
\end{proof}

\section{Transfer Operator}\label{Ap:TranOp}
Suppose that $(M,\mc B)$ is a measurable space. Given a measurable map $F:M\rightarrow M$ define the \emph{push forward}, $F_*\mu$, of any (signed) measure $\mu$ on $(M,\mc B)$ by
\[
F_*\mu(A):=\mu(F^{-1}(A)), \quad \forall A\in\mc B.
\]
The operator $F_*$ defines how mass distribution evolves on $M$ after application of the map $F$. Now suppose that a reference measure $m$ on $(M,\mc B)$ is given. The map $F$ is \emph{nonsingular} if $F_*m$ is absolutely continuous with respect to $m$ and we write it $F_*m\ll m$. If $F$ is nonsingular, given a measure $\mu\ll m$ then also $F_*\mu\ll m$. This means that one can define an operator 
\[
P:L^1(M,m)\rightarrow L^1(M,m) 
\]
such that if $\rho\in L^1$ then $P\rho:=dF_*(\rho\cdot m)/dm$ where $\rho\cdot m$ is the measure with $d(\rho\cdot m)/dm=\rho$. In particular, if $\rho\in L^1$ is a mass density ($\rho\geq 0$ and $\int_M\rho dm$=1) then $P$ maps $\rho$ into the mass density obtained after application of $F$.  One can prove that an equivalent characterization of $P$ is as the only operator that satisfies
\[
\int_M\phi \psi\circ F dm=\int_M P\phi \psi dm,\quad \forall \psi\in L^\infty(M,m)\mbox{ and }\phi\in L^1.
\]
This means that if, for example, $M$ is a Riemannian manifold and $m$ is its volume form and if $F$ is a local diffeomorphism then $P$ can be obtained from the change of variables formula as being 
\[
P\phi(y)=\sum_{\{x:\mbox{ }F(x)=y\}}\frac{\phi(x)}{\Jac F(x)}
\]
where $\Jac F(x)=\frac{d F_*m}{dm}(x)$.
It follows from the definition of $P$ that $\rho\in L^1$ is an \emph{invariant density} for $F$ if and only if $P\rho=\rho$.

\section{Graph Transform: Some Explicit Estimates}\label{Ap:GraphTrans}

We go through once again the argument of the graph transform in the case of a cone-hyperbolic endomorphism of the $n-$dimensional torus. The scope of this result is to compute explicitly bounds on Lipschitz constants for the invariant set of admissible manifolds, and contraction rate of the graph transform (\cite{shub2013global,MR1326374}). 

Consider the torus $\mb T^n$ with the trivial tangent bundle $\mb T^n\times\R^n$. Suppose that $\|\cdot\|:\R^n\rightarrow\R$ is a constant norm on the tangent spaces, and that, with an abuse of notation, $\|x_1-x_2\|$ is the distance between $x_1,x_2\in\mb T^n$ induced by the norm.  Take $n_u,n_s\in \N$ such that $n=n_s+n_u$, and $\Pi_s:\R^n\rightarrow\R^{n_s}$  $\Pi_u:\R^n\rightarrow  \R^{n_u}$  projections for the decomposition $\R^n=\R^{n_u}\oplus \R^{n_s}$.  Identifying $\mb T^{n}$ with $\mb T^{n_u}\times\mb T^{n_s}$, we call $\pi_s:\mb T^n\rightarrow\mb T^{n_s}$ and $\pi_u$ the projection on the respective coordinates. Take $F:\mb T^n\rightarrow \mb T^n$ a $C^2$ local diffeomorphism. We will also define $F_u:=\pi_u\circ F$ and $F_s:=\pi_s\circ F$. Suppose that it satisfies the following assumptions. There are constants $\beta_u,\beta_s>1$, $K_u>0$ and constant cone-fields
\[
\mc C^u:=\left\{v\in\R^n:\quad{\|\Pi_uv\|}\geq \beta_u{\|\Pi_sv\|}\right\}\quad\mbox{and}\quad \mc C^s:=\left\{v\in\R^n:\quad{\|\Pi_sv\|}\geq \beta_s{\|\Pi_uv\|}\right\}
\]
such that:
\begin{itemize}
\item $\forall x\in\mb T^n$, $D_x F(\mc C^u)\subset \mc C^u(F(x))$ and $D_{F(x)}F^{-1}(\mc C^s(F(x)))\subset \mc C^s(x)$;
\item there are real numbers $\lambda_1,\lambda_2,\mu_1,\mu_2\in\R^+$ such that
\begin{align*}
0<\lambda_2&\leq\left\|D_x F|_{\mc C^s}\right\|\leq \lambda_1<1<\mu_1\leq \left\|D_x F|_{\mc C^u}\right\|\leq\mu_2;
\end{align*}
\item 
\[
\|D_{z_1}F-D_{z_2}F\|_u:=\sup_{v\in \mc C^u}\frac{\|(D_{z_1}F-D_{z_2}F)v\|}{\|v\|}\leq K_u\|z_1-z_2\|
\]
\end{itemize}

From now on we denote $(x,y)\in\mb T^n$ a point in the torus with $x\in\mb T^{n_u}$ and $y\in\mb T^{n_s}$. Take $r>0$ and let $B_r^u(x)$ $B^s_r(y)$ be balls of radius $r$ in $\mb T^{n_u}$ and $\mb T^{n_s}$ respectively. Consider 
\[
C^1_{u}(B_r^u(x),B_r^s(y)):=\{\sigma:B_r^u(x)\rightarrow B_r^s(y)\mbox{ s.t. }\|D\sigma\|<\beta_u^{-1} \}.
\]
The condition above ensures that the graph of any $\sigma$ is tangent to the unstable cone. It is easy to prove invertibility of $\pi_u\circ F\circ (id,\sigma)|_{B_r^u(x)}$ for sufficiently small $r$, and it is thus well defined the graph transform 
\[
\Gamma:C^1_u(B_r^u(x),B_r^s(y))\rightarrow C^1_u(B_r^u(F_u(x,y)),B_r^s(F_s(x,y)))
\]
that takes $\sigma$ and maps it to $\Gamma\sigma$ with the only requirement that the graph of $\Gamma\sigma$, $(id,\Gamma\sigma)(B_r^u(F_u(x,y)))$, is contained in $F\circ(id,\sigma)(B_r^u(x))$. An expression for $\Gamma$ is given by
\[
\Gamma\sigma:=[\pi_s\circ F\circ (id,\sigma)]\circ[\pi_u\circ F\circ (id,\sigma)]^{-1}|_{B_r^u(F_u(x,y))}.
\]
The fact that $\|D(\Gamma\sigma)\|\leq\beta_u^{-1}$ is a consequence of the invariance of $\mc C^u$.
Now we prove a result that determines a regularity property for the admissible manifold which is invariant under the graph transform. 
\begin{proposition}\label{Prop:InvRegularityLip}
Consider $\sigma \in C^1_{u,K}(B_r^u(x),B_r^s(y))\subset C^1_{u}(B_r^u(x),B_r^s(y))$ characterized as
\[
\Lip(D\sigma)=\sup_{\substack{x'\neq y'\\ x',y'\in B_r^u(x)}}\frac{\|D_{x'}\sigma-D_{y'}\sigma\|}{\|x'-y'\|}\leq K.
\]
Then the graph transform $\Gamma$ maps $C^1_{u,K}(B_r^u(x),B_r^s(y))$ into $C^1_{u,K}(B_r^u(F_u(x,y)),B_r^s(F_s(x,y)))$ if
\[
K>\frac{1}{1- \frac{\lambda_1}{\mu_1(1-\beta_u^{-1})}}\left(\frac{\mu_2}{\mu_1\lambda_2}K_u\lambda_1\frac{(1+\beta_u^{-1})}{(1-\beta_u^{-1})}+K_u\frac{(1+\beta_u^{-1})^2}{\mu_1(1-\beta_u^{-1})}\right)
\]
\end{proposition}
\begin{proof}
Take $z_1,z_2\in B_r^u(z)$, with $z=\pi_u\circ F\circ(id,\sigma)(x)$ and suppose that $x_1,x_2\in B_r^u(x)$ are such that $\pi_u\circ F\circ (id,\sigma)(x_i)=z_i$. Take $w\in\R^{n_u}$, and suppose that $v_1,v_2\in\R^{n_u}$ satisfy $\Pi_uD_{(x_i,\sigma(x_i))}F(v_i,D_{x_i}\sigma(v_i))=w$. Then
\begin{align*}
\|D_{z_1}(\Gamma\sigma)(w)-&D_{z_2}(\Gamma\sigma)(w)\|\leq\\&\leq \|D_{(x_1,\sigma(x_1))}F(v_1,D_{x_1}\sigma(v_1))-D_{(x_2,\sigma(x_2))}F(v_2,D_{x_2}\sigma(v_2))\|\\
&\leq\|D_{(x_1,\sigma(x_1))}F\left(v_1-v_2,D_{x_1}\sigma(v_1-v_2)\right)\|+\\
&\phantom{+}+\|D_{(x_1,\sigma(x_1))}F(0,D_{x_1}\sigma-D_{x_2}\sigma)v_2\|+\\
&\phantom{+}+\|D_{(x_1,\sigma(x_1))}F-D_{(x_2,\sigma(x_2))}F\|_u\|(v_2,D_{x_2}\sigma v_2)\|\\
&\leq\mu_2\|v_1-v_2\|+\lambda_1\|D_{x_1}\sigma-D_{x_2}\sigma\|\|v_2\|+K_u(1+\beta_u^{-1})\|x_1-x_2\|\|(v_2,D_{x_2}\sigma v_2)\|.
\end{align*}
Now
\[
\|x_1-x_2\|\leq \lambda_1(1+\beta_u^{-1})\|z_1-z_2\|
\]
and 
\begin{align*}
\|v_1-v_2\|&=\|v_1-\Pi_u(D_{(x_2,\sigma(x_2))}F)^{-1}D_{(x_1,\sigma(x_1))}F(\Id,D_{x_1}\sigma)(v_1)\|\\
&=\|\Pi_u(D_{(x_2,\sigma(x_2))}F)^{-1}(D_{(x_1,\sigma(x_1))}F-D_{(x_2,\sigma(x_2))}F)(\Id,D_{x_1}\sigma)(v_1)\|\\
&\leq\lambda_2^{-1}K_u\|x_1-x_2\|\|v_1\|\\
&\leq\lambda_2^{-1}K_u\lambda_1(1+\beta_u^{-1})\|z_1-z_2\|\|v_1\|.
\end{align*}
Taking into account that $\|v_1\|,\|v_2\|\leq\mu_1^{-1}(1-\beta_u^{-1})^{-1}\|w\|$
\[
\Lip(D_\cdot (\Gamma\sigma))\leq \frac{\lambda_1}{\mu_1(1-\beta_u^{-1})}\Lip(D_{\cdot}\sigma)+\frac{\mu_2}{\mu_1}\lambda_2^{-1}K_u\lambda_1\frac{(1+\beta_u^{-1})}{(1-\beta_u^{-1})}+K_u\frac{(1+\beta_u^{-1})^2}{\mu_1(1-\beta_u^{-1})}
\]
and this gives the condition of invariance of the proposition.
\end{proof}

\begin{proposition}\label{Prop:ContGraphtransC0}
For all $\sigma_1,\sigma_2\in C^1_{u}(B_r^u(x),B_r^s(y))$
\[
\sup_{z\in B^u_r(F_u(x,y))}\|(\Gamma\sigma_1)(z)-(\Gamma\sigma_2)(z)\|\leq [\lambda_1+\lambda_1^2\mu_1^{-1}\beta_u^{-1}+\mu_2\mu_1^{-1}\lambda_1\beta_u^{-1}]\sup_{t\in B^u_r(x)}\|\sigma_1(t)-\sigma_2(t)\|
\]
Then if
\[
\lambda_1+\lambda_1^2\mu_1^{-1}\beta_u^{-1}+\mu_2\mu_1^{-1}\lambda_1\beta_u^{-1}<1
\]
 $\Gamma:C^1_{u}(B_r^u(x),B_r^s(y))\rightarrow C^1_{u}(B_r^u(F_u(x,y)),B_r^s(F_s(x,y)))$ is a contraction in the $C^0$ topology.
\end{proposition}
\begin{proof}
Take $\sigma_1,\sigma_2\in C^1_{u}(B_r^u(x),B_r^s(y))$, and $z\in B_{r}^u(F_u(x,y))$, and suppose that $x_1,x_2\in B_r^u(x)$ are such that $F_u(x_1,\sigma_1(x_1))=z$ and $F_u(x_2,\sigma_2(x_2))=z$. 
\begin{align*}
\|(\Gamma\sigma_1)(z)-(\Gamma\sigma_2)(z)\|&=\|F_s(x_1,\sigma_1(x_1))-F_s(x_2,\sigma_2(x_2))\|\\
&\leq \|F_s(x_1,\sigma_1(x_1))-F_s(x_1,\sigma_2(x_1))\|+\| F_s(x_1,\sigma_2(x_1))-F_s(x_2,\sigma_2(x_2))\|\\
&\leq \lambda_1\|\sigma_1(x_1)-\sigma_2(x_1)\|+\lambda_1\Lip(\sigma_1)\|x_1-x_2\|+\Lip(F)\beta_u^{-1}\|x_1-x_2\|.
\end{align*}
The following estimates hold
\begin{align}
\|x_1-x_2\|&=\|x_1-(F_u\circ(id,\sigma_2))^{-1}\circ(F_u\circ(id,\sigma_1))(x_1)\|\nonumber\\
&=\|x_1-(F_u\circ(id,\sigma_2))^{-1}[F_u\circ(id,\sigma_2)(x_1)+F_u\circ(id,\sigma_1)(x_1)-F_u\circ(id,\sigma_2)(x_1)]\|\nonumber\\
&\leq\|x_1-x_1\|+\|D_{\bar x}(F_u\circ(id,\sigma_2))^{-1}\|\|F_u\circ(id,\sigma_1)(x_1)-F_u\circ(id,\sigma_2)(x_1)\|\nonumber\\
&\leq \|D_{\bar x}F_u\|^{-1}\lambda_1\|\sigma_1(x_1)-\sigma_2(x_1)\|\nonumber\\
&\leq \mu_1^{-1}\lambda_1\|\sigma_1(x_1)-\sigma_2(x_1)\|\label{Eq:estdiffpoints}
\end{align}
and hence
\[
\|(\Gamma\sigma_1)(z)-(\Gamma\sigma_2)(z)\|\leq [\lambda_1+\lambda_1^2\mu_1^{-1}\beta_u^{-1}+\mu_2\mu_1^{-1}\lambda_1\beta_u^{-1}]\|\sigma_1(x_1)-\sigma_2(x_1)\|.
\]

\end{proof}

Consider $V\subset \mc C^u$ any linear subspace of dimension $n_u$ contained in $\mc C^u$. This is uniquely associated to $L:\R^{n_u}\rightarrow\R^{n_s}$, such that $(\Id, L)(\R^{n_u})=V$. 
\begin{definition}
Given any two $V_1,V_2\subset\mc C^u$ linear spaces of dimension $n_u$, we can define the distance 
\begin{equation}\label{Eq:Defd_u}
d_u(V_1,V_2):=\sup_{\substack{u\in \R^{n_u}\\ \|u\|=1}}\|L_1(u)-L_2(u)\|,
\end{equation}
\end{definition}
(which is also the operator norm of the difference of the two linear morphisms defining the subspaces).
\begin{proposition}\label{Prop:ContTangentSpace}
If 
\[
\mu_1^{-1}\left[\lambda_1+\frac{\beta_u\lambda_1}{\mu_1(1-\beta_u)}\right]<1
\]
then $D_zF$ is a contraction with respect to $d_u$ for all $z\in\mb T^n$. 
\end{proposition}
\begin{proof}
Pick $L_1,L_2:\R^{n_u}\rightarrow\R^{n_s}$ with $\|L_i\|<\beta_u$. They define linear subspaces $V_i=(\Id,L_i)(\R^{n_u})$ which, as a consequence of the condition on the norm of $L_i$, are tangent to the unstable cone. $V_1$ and $V_2$ are transformed by $D_z F$ into subspaces $V_1'$ and $V_2'$. This subspaces are the graph of linear transformations $L'_1,L'_2:\R^{n_u}\rightarrow\R^{n_s}$ ($\|L_i'\|\leq\beta_u$). Analogously to the graph transform one can find explicit expression for $L_i'$ in terms of $L_i$:
\[
L_i'=\Pi_s\circ D_zF\circ (\Id,L_i)\circ[\Pi_u\circ D_zF\circ (\Id, L_i)]^{-1}.
\]
To prove the proposition we then proceed analogously to the proof of Proposition \ref{Prop:ContGraphtransC0}. Pick $u\in\R^{n_u}$ and suppose that $u_1,u_2\in\R^{n_u}$ are such that
\[
(\Id,L_1')(u)=D_zF\circ (\Id,L_1)(u_1)\quad(\Id,L_2')(u)=D_zF\circ (\Id,L_2)(u_2).
\]
With the above definitions
\begin{align*}
\|L_1'(u)-L_2'(u)\|&=\|\Pi_s\circ D_zF(\Id,L_1)(u_1)-\Pi_s\circ D_zF(\Id,L_1)(u_2)\|\\
&\leq\|\Pi_s\circ D_zF(\Id,L_1)(u_1)-\Pi_s\circ D_zF(\Id,L_2)(u_1)\|\\
&\phantom{=}+\|\Pi_s\circ D_zF(\Id,L_2)(u_1-u_2)\|\\
&\leq \lambda_1\|L_1-L_2\|\|u_1\|+\beta_u\mu_2(1+\beta_u)\|u_1-u_2\|.
\end{align*}
\begin{align*}
\|u_1-u_2\|&=\|u_1-[\Pi_u\circ D_zF\circ (\Id, L_2)]^{-1}\Pi_u\circ D_zF\circ (\Id, L_1)(u_1)\|\\
&= \|[\Pi_u\circ D_zF\circ (\Id, L_2)]^{-1}\Pi_u\circ D_zF\circ (0, L_1-L_2)(u_1)\|\\
&\leq\|\Pi_u\circ D_zF\circ (\Id, L_2)^{-1}\|\beta_u\lambda_1\|L_1-L_2\|\|u_1\|\\
&\leq \frac{\beta_u\lambda_1}{\mu_1(1-\beta_u)}\|L_1-L_2\|\|u_1\|
\end{align*}
The two estimates together imply that
\[
\|L_1'-L_2'\|\leq\mu_1^{-1}\left[\lambda_1+\frac{\beta_u\lambda_1}{\mu_1(1-\beta_u)}\right]\|L_1-L_2\|.
\]
\end{proof}

\section{Proof of Theorem \ref{MTheo:B}}\label{appendix:thmc}

%


Let $g \colon \mb T\to \mb T$  and for  $\omega\in \R$ define  $g_\omega=g+\omega$.  
Let $\underline \omega=(\dots,\omega_n,\omega_{n-1},\dots,\omega_0)$ with $\omega_i\in (-\epsilon',\epsilon')$ with $\epsilon'>0$ small. 
Define $g^k_{\underline \omega}= g_{\omega_k} \circ \dots \circ g_{\omega_1}\circ g_{\omega_0}$.

\bp\label{Prop:ConvergenceToAtt} Let $g\colon \mb T\to \mb T$ be $C^2$ and hyperbolic (in the sense of Definition~\ref{Def:AxiomA}), and assume that $g$ has an attracting set $\Lambda$
(consisting of periodic orbits). 
Then there exist $\chi\in (0,1)$, $C>0$ so that for each $\epsilon>0$ and $T=1/\epsilon$ the following holds. 

 There exists a set $\Omega\subset \mb T$ of measure $1-\epsilon^{1-\chi}$ so that
for any $k\ge T_0$, and any $\underline \omega= 
(\dots,\omega_n,\omega_{n-1},\dots,\omega_0)$ with $|\omega_i|\le C\epsilon$, and for each $k\ge T_0$, 
\begin{itemize}
\item $g^k_{\underline \omega}$  maps each component $J$ of $\Omega$ into components
of the immediate basin of the periodic attractor of $g$; 
\item the distance of $g^k_{\underline \omega}(J)$ to a periodic attractor of $g$ is at most $\epsilon$. 
\end{itemize}
\ep

The proof of this proposition follows from the next two lemmas: 

\bl Let $g\colon \mb T\to \mb T$ be $C^2$ and hyperbolic (in the sense of Definition~\ref{Def:AxiomA}), and assume that $g$ has an attracting set $\Lambda$
(consisting of periodic orbits). Then the repelling hyperbolic set $\Upsilon=\mb T \setminus W^s(\Lambda)$ of $g$ 
is a Cantor set with  Hausdorff dimension $\chi'<1$.  Moreover, for each $\chi\in (\chi',1)$, the Lebesgue measure of the $\epsilon$-neighborhood $N_\epsilon(\Upsilon)$ of $\Upsilon$
is at most $\epsilon^{1-\chi}$ provided $\epsilon>0$ is sufficiently small. 
\el
\begin{proof} It is well known that the set $\Upsilon$ is a Cantor set, see \cite{MS}.
Notice that by definition $g^{-1}(\Upsilon)=\Upsilon$. 
It is also well known that the Hausdorff dimension
of a hyperbolic set $\Upsilon$ associated to a $C^2$ one-dimensional map is $<1$  and that this dimension  is equal to its Box dimension, 
see \cite{MR1489237}.  Now take a covering of $\Upsilon$ with intervals
of length $\epsilon$, and let $N(\epsilon)$ be the smallest number of such intervals that are needed. By the definition of Box dimension
$\lim_{\epsilon\to 0}  \frac{\log N(\epsilon)}{\log(\epsilon)} \to \chi'$. It follows that $N(\epsilon)\le \frac{1}{\epsilon^{\chi}}$ for $\epsilon>0$ small.  
It follows that the Lebesgue measure of $N_\epsilon$ is at most $N(\epsilon)\epsilon \le \epsilon^{1-\chi}$ for $\epsilon>0$ small. 
\end{proof}

For simplicity assume that $n=1$ in Definition~\ref{Def:AxiomA}. As  in Subsection~\ref{subsec:mather}
the general proof can be reduced to this case.

\bl  Let $g$ and $g^k_{\underline \omega}$ as above. 
Then there exists $C>0$ so that for each $\epsilon>0$ sufficiently small, 
and taking $\tilde N= N_\epsilon(\Upsilon)$ and $|\omega_i|< \epsilon'=C\epsilon$ we have the following:
\begin{enumerate}
\item  $g^k_{\underline \omega}(\mb T \setminus \tilde N)\subset \mb T \setminus \tilde N$ for all $k\ge 1$.
\item  $\mb T\setminus \tilde N$ consists of at most $1/\epsilon$ intervals.
\item Take $T_0=2/\epsilon$. Then for each $k\ge T_0$, $g^k_{\underline \omega}$ maps each component $J$ of  $\mb T\setminus \tilde N$ 
into a component of the immediate basin of a  periodic attractor of $g$.  Moreover, $g^k_{\underline \omega}(J)$ has length $<\epsilon$
and has distance $<\epsilon$ to a periodic attractor of $g$.
\end{enumerate}
\el 
\begin{proof}
The first statement follows from the fact that we assume that $|Dg|>1$ on $\Upsilon$, because $\Upsilon$ is backward invariant, and  by continuity.  
To prove the second statement let $J_i$ be the components of $\mb T \setminus N_{\epsilon/4}(\Upsilon)$.
If $J_i$ has length $<\epsilon$ then $J_i$ is contained in $\mb T \setminus N_{\epsilon}(\Upsilon)$.
So the remaining intervals $J_i$ all have length $\ge \epsilon$ and cover $\mb T \setminus N_{\epsilon}(\Upsilon)$.
The second statement follows. To see the third statement, notice that the only components of $\mb T \setminus \Upsilon$
containing periodic points are those that contain periodic attractors. Since $\Upsilon$ is fully invariant, 
$\mb T \setminus \Upsilon$ is forward invariant. In particular, if $J'$ is component of $\mb T \setminus \Upsilon$ 
then there exists $k$ so that $g^k(J')$ is contained in the immediate basin of a periodic attractor of $g$ and 
$J',\dots,g^k(J')$ are all contained in different components of $\mb T \setminus \Upsilon$.  This, together with 
1) and 2) implies that each component of  $\mb T\setminus \tilde N$ is mapped in at most $1/\epsilon$ steps
into the immediate basin of $g$.  Since the periodic attractor is hyperbolic, it follows that under $1/\epsilon$ further 
iterates this interval has length $<\epsilon$ and has distance at most $\epsilon$ to a periodic attractor (here we use that 
$\epsilon>0$ is sufficiently small so that also $2/\epsilon>m$). 
\end{proof}
\begin{proof}[Proof of Theorem \ref{MTheo:B}] a) Fix an integer $\sigma\ge 2$, $\alpha\in \R$, $\kappa\in (0,1]$. The  map  $\mathcal F\colon C^k(\mb T \times \mb T, \R) \to C^k(\mb T,\R)$ defined  by $\mathcal F(h)(x)=\int h(x,y)\, dy$ 
is continuous. Since  the set of hyperbolic $C^k$ maps $g\colon \mb T \to \mb T$ is open and dense
in the $C^k$ topology, see  \cite{MR2342693},  it follows that the set of $C^k$ functions $h\in C^k(\mb T \times \mb T, \R)$  for which 
$x\mapsto \sigma x + \alpha \kappa \int h(x,y) \, dm_1(y) \mod 1$ is hyperbolic is also open and dense in the $C^k$ topology, which proves the first statement of the theorem. (The above is true for $k\in\N$, $k=\infty$, or $k=\omega$). 

To prove b), first of all recall that if $g\in C^k(\mb T,\mb T)$ is a hyperbolic map with a critical point $x\in\mb T$, then $g$ has a periodic attractor and $x$ belongs to its basin. If $h\in C^k(\mb T \times \mb T, \R)$, supposing that  $\mc F(h)(x)$ is not constant, then
\begin{equation}\label{Eq:negcond}
\exists x\in\mb T\quad\mbox{ s.t. }\quad\frac{d\mc F(h)(x)}{dx}<0
\end{equation}
Condition \eqref{Eq:negcond} holds for an open and dense set $\Gamma''\subset C^k(\mb T \times \mb T, \R)$. Pick $h\in \Gamma''$, then from \eqref{Eq:negcond} follows that there exist an open neighbourhood $V$ of $h$, and an interval $\mc I\subset\R$ such that $g_{\beta,h}(x)=\sigma x+\beta\mc F(h)(x)\mod 1$ has a critical point for all $h\in V$ and $\beta\in \mc I$. Since the map $\mc I\times V\rightarrow C^k(\mb T,\mb T)$ is continuous, there is an open and dense subset of $\mc I\times V$ for which the map $g_{\beta,h}$ is hyperbolic, and thus has a finite periodic attractor. Furthermore, if $g_{\beta,h}$ has a periodic attractor, by structural stability there is an open interval  $\mc I_\beta$ such that also $g_{\beta',h}$ has a periodic attractor for all $\beta'\in\mc I_\beta$. Once the existence of a hyperbolic periodic attractor is established, the rest of the proof follows from Theorem \ref{Thm:Main} and Proposition \ref{Prop:ConvergenceToAtt}.

\end{proof}

The following two propositions contain rigorous statement regarding the example presented in the introduction of the paper. 

\bp \label{Prop:AppTbeta1}
For any $\beta\in \R$, the map  $T_{\beta}(x)= 2 x   - \beta \sin (2\pi x)   \mod 1$
has at most two periodic attractors $O_1, O_2$ with $O_1=-O_2$. 
\ep
\begin{proof} The map $T_\beta$ extends to an entire map on $\C$ and therefore each periodic attractor has a critical point in its basin \cite{MR1216719}. 
This implies that there are at most two periodic attracting orbits.  Note that $T_\beta(-x)=-T_\beta(x)$ and therefore if $O$ is a finite set in $\R$ corresponding
to a periodic orbit of $T_\beta$, then so is $-O$, and it follows that if $T_\beta$ has two periodic attractors $O_1$ and $O_2$ then $O_1=-O_2$. 
(If $T_\beta$ has only one periodic attractor $O$, then  one has $O=-O$.) Notice that indeed 
there exist parameters $\beta$ for which $T_\beta$ has two attracting orbits.  For example, when $\beta=1.25$ then $T_\beta$ has two distinct attracting fixed points. 
\end{proof}

\bp \label{Prop:AppTbeta2}
Given $\kappa_0,\kappa_1,\dots,\kappa_m$, 
consider the families $T_{\beta,j}(x)= 2 x   - \beta\kappa_j \sin (2\pi x)   \mod 1$.
Then 
\begin{enumerate}
\item there exists an open and dense subset  $\mathcal I'$ of $\R$ so that for each $\beta\in \mathcal I'$
each of the maps  $T_{\beta,j}$, $j=1,\dots,m$  is hyperbolic. 
\item there exists $\beta_0>0$ and an open and dense subset $\mathcal I$ of $(-\infty,-\beta_0)\cup (\beta_0,\infty)$ so that for each $\beta\in \mathcal I$,
each of the maps $T_{\beta,j}$, $j=1,\dots,m$  is hyperbolic and has a periodic attractor.
\end{enumerate}
\ep
\begin{proof}
Let $\mathcal H$ be the set of parameters $\beta\in \R$ so that $T_{\beta}(x)= 2 x   - \beta\sin (2\pi x)   \mod 1$
is hyperbolic.  By \cite{MR3336841}, the set $\mathcal H$ is open and dense. 
It follows that $(1/\kappa_j)\mathcal H$ is also open and dense. 
Hence $(1/\kappa_1)\mathcal H\cap \dots \cap (1/\kappa_m)\mathcal H$ is open and dense.
It follows in particular that this intersection is open and dense  in $\R$. 

For each $|\beta|>2\pi$ the map  $T_{\beta}(x)= 2 x   - \beta\sin (2\pi x)   \mod 1$ has a critical point, and so  if such a map $T_\beta$ is
hyperbolic then, by definition, $T_\beta$ has one or more periodic attractors (and each critical point is in the
basin of a periodic attractor). So if we take $\beta_0=\max(2\pi / \kappa_1,\dots,2\pi/\kappa_m)$ the second assertion follows.
\end{proof}

\section{Proof of Theorem~\ref{MTheo:C}}\label{App:RandGrap}

The study of global synchronization of chaotic systems has started
in the eighties for systems in the ring \cite{fujisaka1983,heagy1994}.  This
approach was generalized for undirected networks of
diffusively coupled systems merging numerical computations of
Lyapunov exponents and transverse instabilities of the synchronous
states. See also \cite{dorfler2014synchronization,eroglu2017} for a review. These results have been generalized to weighted and directed
graphs via dichotomy estimates \cite{Pereira2014}. In our Theorem C,
we make use of these ideas to obtain an open set of coupling function
such that the networks will globally synchronize for random homogeneous networks.
Simultaneously, our Theorem A guarantees that any coupling function in
this set can exhibit hub synchronization.

\begin{proof}[Proof of the Theorem~\ref{MTheo:C}]

First we recall that the manifold $\mc S$ is invariant $F(\mc S) \subset \mc S$. Indeed, if the system is in $\mc S$ at a time $t_0$, hence $x_1(t_0)=\cdots =x_N(t_0)$, then because $h(x(t_0),x(t_0))=0$ the whole coupling term vanishes and the evolution of the network will be given by $N$ copies of the evolution of $x(t_0)$. Hence, we notice that the dynamics on $\mc S$ is the dynamics of the uncoupled chaotic map, $x_i(t+1) = f(x_i(t))$ for all $t\ge t_0$ and $i=1,\dots, N$. Our goal is to show that for certain diffusive coupling functions, $\mc S$ is normally attracting. The proof of item a) can be adapted from \cite{Pereira2014}.

\noindent
{\bf Step 1 Dynamics near $\mathcal{S}$}. 
In a neighborhood of $\mc S$ we can write
$x_i = s + \psi_i$ where $s(t+1) = f(s(t))$ and $| \psi_i | \ll 1$.
Expanding the coupling in Taylor series, we obtain 
\begin{eqnarray}
\psi_i(t+1)=f^{\prime} (s(t))\psi_i(t)) + \frac{\alpha}{\Delta}  \sum_{j}
A_{ij} [ h_1(s(t),s(t))\psi_i(t) +  h_2(s(t),s(t))\psi_j(t) + R(\psi_i(t),\psi_{j}(t))] \nonumber
\end{eqnarray}
where $h_i$ stands for the derivative of $h$ in the $i$th entry and $R(\psi_i,\psi_j)$ is a nonlinear remainder, by Lagrange Theorem we have $R(\psi_i,\psi_j) < C( |\psi_i|^2 +  |\psi_j|^2)$, for some positive  constant $C =C(A,h,f)$. Moreover, because $h$ is diffusive 
\[ h_1(s(t),s(t)) = - h_2(s(t),s(t)).\]
Defining $\omega(s(t)) 
:= h_1(s(t),s(t))$ and  entries of the Laplacian matrix $L_{ij} = A_{ij} -
d_i \delta_{ij}$, we can write the first variational equation in compact form by introducing
$\Psi = (\psi_1, \cdots, \psi_n) \in \mathbb{R}^n$. Indeed, 
\begin{equation}\label{mu}
\Psi(t+1) =\left [ f^{\prime} (s(t))  I_N - \frac{\alpha}{\Delta} \omega(s(t)) L \right]  \Psi(t).
\end{equation}
Because the laplacian is symmetric, it admits a spectral decomposition
$L = U \Lambda U^*$, where $U$ is the matrix of eigenvectors and $\Lambda =$diag$(\lambda_1,\dots,\lambda_N)$ the matrix of eigenvalues. Also its eigenvalues can organized in increasing order
$$
0=\lambda_1 < \lambda_2 \le \cdots \le \lambda_N,
$$
as the operator is positive semi-definite. The eigenvalue $\lambda_1 =
0$ is always in the spectrum as  every row of $L$ sums to zero. Indeed, consider
$\mathbf{1} = (1,\cdots,1) \in \mathbb{R}^n$ then $L \mathbf{1} = 0$.
Notice that this direction $\mathbf{1}$ is associated with the
synchronization manifold $\mc S$. All the remaining eigenvectors
correspond to transversal
directions to $\mc S$. The Laplacian $L$ has a spectral gap $\lambda_2 >0$ because the
network is connected, as is shown in Theorem \ref{SpectralBounds}. So,
we introduce new coordinates $\Theta = U \Psi$  to diagonalize $L$. Notice that 
by construction $\Psi $ is not in the subspace generated by $\{ \mathbf{1}\}$, and thereby $\Psi$ is associated to the dynamics in the transversal eigenmodes. Writing $\Theta = (\theta_1,\dots,\theta_N)$, we obtain the dynamics for the $i$-th component
$$
\theta_i(t+1) = [ f^{\prime} (s(t)) + \alpha \lambda_i  \omega(s(t)) ]  \theta_i.
$$
Thus, we decoupled all transversal modes. Since we are interested in the
transverse directions we only care about $\lambda_i > 0$. This is
equivalent to the linear evolution of  Eq. (\ref{mu}) restricted
to the subspace orthogonal  $\mathbf{1}$.


\noindent
{\bf Step 2. Parametric Equation for Transversal Modes}. As we
discussed, the modes $\theta_i$ with $i=2,\cdots,N$
correspond to the dynamics transversal to $\mathcal{S}$. If these modes are
damped the manifold $\mathcal{S}$ will be normally attracting. Because all equations 
are the same up to a factor $\lambda_i$, we can tackle them all at once by
considering a parametric equation
\begin{equation}\label{para}
z(t+1) = [ f^{\prime} (s(t)) - \beta  \omega(s(t)) ]  z(t).
\end{equation}
This equation will have a uniformly exponentially attracting trivial solution if
\begin{equation}\label{st}
\nu := \sup_{t>0}  \| f^{\prime}(s(t)) - \beta  \omega(s(t))  \| < 1.
\end{equation}

Now pick any $\phi\in C^1(\mb T; \mb R)$ with $\frac{d\phi}{dx}(0)\neq 0$, and suppose that $h'(x,y)$ is a diffusive coupling function with $\|h'(x,y)-\phi(y-x)\|_{C^1}<\epsilon$.

Because $f^{\prime}(s(t)) = \sigma$ and
\[
\omega(s(t)) = -\frac{d\phi}{dx}(0) + \frac{\partial}{\partial x}\left[h'(x,y)-\phi(y-x)\right] (s(t),s(t)),
\]
the condition in Eq. (\ref{st}) is always satisfied as long as
\begin{equation}\label{Eq:Ineqbeta}
\left | \sigma - \beta \frac{d\phi}{dx}(0)\right | + |\beta|\epsilon < 1
\end{equation}
Suppose that $\frac{d\phi}{dx}(0)>0$ (the negative case can be dealt with analogously). 
Define
\[
\beta_c^1 := (\sigma-1) \left(\frac{d\phi}{dx}(0)\right)^{-1} \mbox{~  and ~} \beta_c^2 := (\sigma+1) \left(\frac{d\phi}{dx}(0)\right)^{-1}.
\]
Then, there is an interval $\mc I \subset (\beta_c^1, \beta_c^2)$ such that all $\beta \in \mc I$ 
the inequality \eqref{Eq:Ineqbeta} holds. From the parametric equation we can obtain
the $i$-th equation for the transverse mode by setting $\beta = \frac{\alpha}{\Delta} \lambda_i$ and $\theta_i$'s  will decay to  zero exponentially fast if

\begin{equation}\label{Eq:CondEig}
\beta_c^1 < \frac{\alpha}{\Delta} \lambda_2 \le \cdots \le \frac{\alpha}{\Delta}
\lambda_N <  \beta_c^2.
\end{equation}

Hence, if  the eigenvalues satisfy

\begin{equation}\label{SyncC}
\frac{\lambda_N}{\lambda_2} < \frac{\sigma+1}{\sigma-1},
\end{equation}

\noindent
then one can find an interval $I\subset\R$ for the coupling strength, such that Eq. \eqref{Eq:CondEig} is satisfied for every $\alpha\in I$. 

\noindent
{\bf Step 3. Bounds for Laplacian Eigenvalues}. Theorem \ref{Gp}
below shows that almost every graph  $G \in \mathcal{G}_p$
$$
\frac{\lambda_N(G)}{\lambda_2(G)} = 1 + o(1)
$$
Hence, condition Eq. (\ref{SyncC}) is met and we guarantee that the transversal instabilities are damped uniformly and exponentially fast, and as a consequence the manifold $\mathcal{S}$ is normally attracting. We illustrate such a network in Figure \ref{2k}. Indeed, since  the coordinates $\theta_i$ of the linear approximation decay to zero exponentially $\theta_i(t) \le C e^{-\eta t}$ for all $i=2,\cdots,N$ with $\eta>0$, then the full nonlinear equations synchronize. Indeed,  $\| \Psi (t) \| \le \tilde Ce^{-\eta t}$, which means that the 
the first variational equation Eq. \ref{mu} is uniformly stable. To tackle the nonlinearities in the remainder, we notice that for any $\varepsilon>0$ there is $\delta_0>0$ 
and $ C_{\varepsilon}>0$ such that for all  $| x_i (t_0) - x_j(t_0) | \le \delta_0$, the nonlinearity is small and by a Gr\"onwall type estimate we have
$$
| x_i (t) - x_j(t) | \le C_{\varepsilon} e^{- (t-t_0) (\eta - \varepsilon)}.
$$
this will precisely happen when the condition Eq. (\ref{SyncC}) is satisfied. The open set for coupling function follows as uniform exponential attractivity is an open property. The proof of item a) is therefore complete.

For the proof of $b)$ we use Steps 1 and 2, and only change the
spectral bounds. From Theorem \ref{SpectralBounds} we obtain
$$
\frac{\lambda_N}{\lambda_2} >  \frac{d_{N,N}}{d_{1,N}}
$$
hence as heterogeneity increases as the ratio tends to infinity for $N \rightarrow \infty$ and condition Eq. (\ref{SyncC})
is never met regardless the value of $\alpha$. Hence there are always
unstable modes,  and
the synchronization manifold $\mc S$ is unstable.
\end{proof}

The spectrum of the Laplacian is related to many important graph invariants. In particular the diameter $D$ of the graph, which is the maximum distance between any two nodes. Therefore, if the graph is connected the $D$ is finite. 

\begin{theorem}\label{SpectralBounds}
Let $G$ be a simple network of size $N$ and $L$ its associated Laplacian. Then:
\begin{enumerate}
\item  \emph{\cite{Mohar91}} $ \lambda_2 \ge \displaystyle \frac{4}{N D} $
\item \emph{\cite{Fiedler73}} $ \lambda_2 \le  \displaystyle \frac{N}{N-1} d_1 $
\item \emph{\cite{Fiedler73}} $ \frac{N}{N-1}d_{\max} \le \lambda_N \le 2 d_{\max}
$
\end{enumerate}
\label{boundl}
\end{theorem}

The proof of the theorem  can be found in references
we provide in the theorem.

\begin{theorem}[\cite{Mohar92}]\label{Gp}
Consider the ensemble of  random graphs $\mathcal{G}_p$ with $p > \frac{\log N}{N}$, then a.s.
$$
\lambda_2 > Np - f(N)  \,\,\,\, \mbox{~and~} \,\,\,\, \lambda_N< pN + f(N)
$$
where
$$
f(N) = \sqrt{(3+\varepsilon)(1-p)pN \log N}
$$
for $\varepsilon>0$ arbitrary.
\end{theorem}

~

\noindent
{\it Regular Networks.} 
Consider a network of $N$ nodes, in which each
node is coupled to its $2K$ nearest neighbors. See an Illustration in Figure \ref{2k} when 
$K=2$. In such regular network every node has the same degree $2K$.
 \begin{figure}[htbp]
\centering
\includegraphics[width=3in]{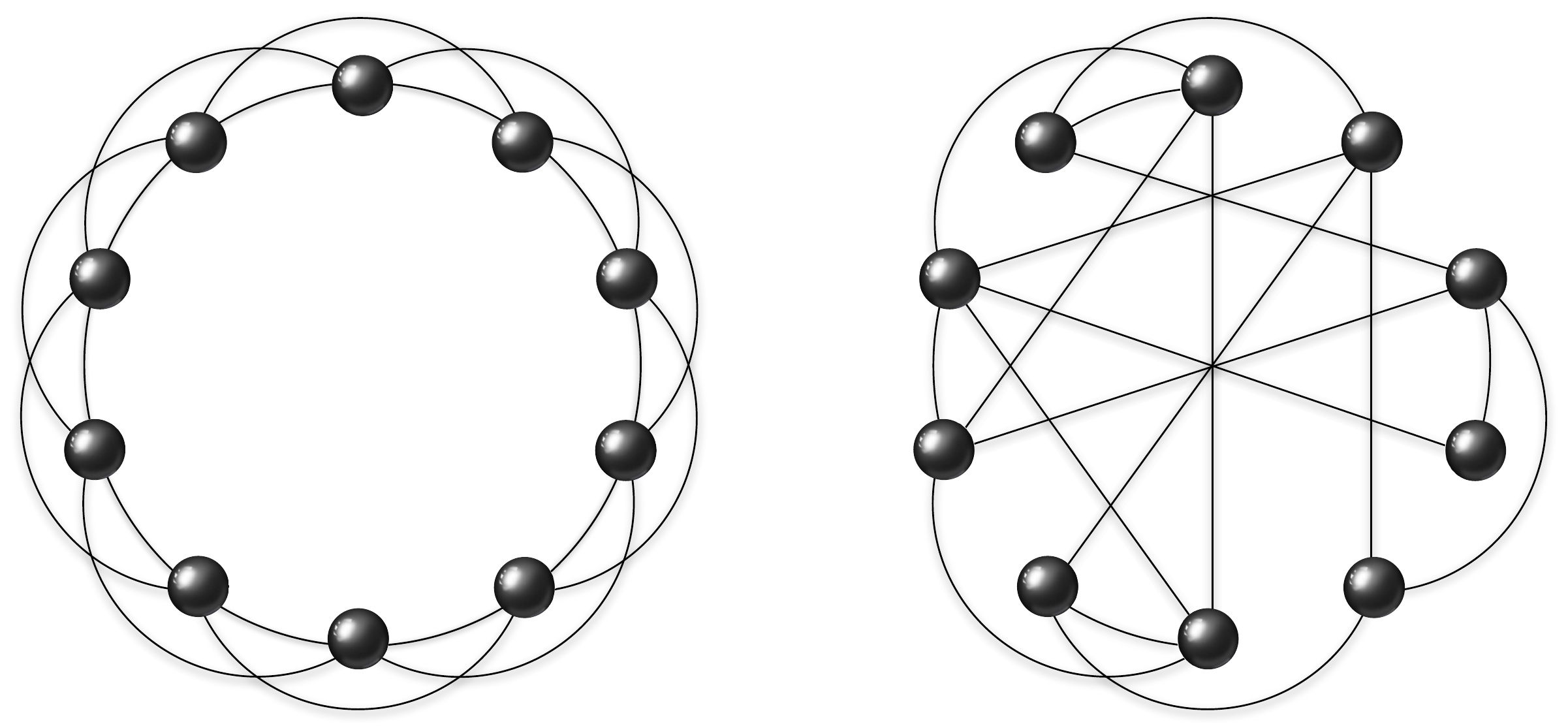}
\caption{On the left panel we present a regular network where every node connects to its $2$ left and $2$ right nearest neighbors. Such networks shows poor synchronization properties in the large $N$ limit if $K\ll N$ as shown in Eq. (\ref{reg}). On the right panel, we depict a random (Erd\"os-R\'enyi) network where every connection is a Bernouilli random variable with success probability $p=0.3$. Such random networks tend to be homogeneous (nodes have $pN$ connections) and they exhibit excellent synchronization properties. 
}
\label{2k}
\end{figure}\\

Whenever,  $K \ll N$ the network will not display synchronization. This is because the diameter $D$ of the network (the maximal distance between any two nodes) is proportional to $N$. In this case, roughly speaking, the network is essentially disconnected as $N \rightarrow \infty$.  However, as  $K \rightarrow N/2$ the network is optimal for synchronization. Here the diameter of the network is extremely small as the graph is close to  a full graph. 

Indeed, since the Laplacian is circulant, it can be diagonalized by discrete Fourier Transform, and eigenvalues of a regular graph can be obtained explicitly  \cite{Barahona2002}
\[
\lambda_j = 2K -\frac{\sin \left( \frac{(2K +1)\pi (j-1) }{N}\right)}{\sin \pi (j-1) /N}, \mbox{~for~} j=2,\dots,N.
\]
Hence, we can obtain the asymptotics in $K \ll N$ for the synchronization Eq. (\ref{SyncC}). Using a Taylor expansion in this expression, we obtain
\begin{equation}\label{reg}
\frac{\lambda_N}{\lambda_2} \approx \frac{(3\pi +2)N^2}{2\pi^3K^ 2 }
\end{equation}
Hence, when $K \ll N$ synchronization is never attained.  

From a graph theoretic perspective, when $K\ll N$, e.g. $K$ is fixed and $N \rightarrow \infty$
then $\lambda_2 \sim 1/N^2$, implying that the bound in Theorem \ref{SpectralBounds} is tight, as the diameter of such networks is roughly $D \sim N$. 

This is in stark constrast to random graphs, the mean degree of each nodes is approximately $d_{i,N} = pN$. However, even in the limit $d_{i,N} \ll N$, randomness  drastically reduces the diameter of the graph, in fact, in the  model we have $D \propto \log N$ (again $p > \log N / N$). Although the regular graphs exhibit a quite different synchronization scenario when compared to homogeneous random graphs, if we include a layer of highly connected nodes can still exhibit distinct dynamics across levels.

\noindent

\section{Random Graphs}\label{Sec:ApRandGrap}
A  {\em random graph model} of size $N$ is a probability measure on the set $\mc G(N)$ of all graphs on $N$ vertices. Very often random graphs are defined by models that assign probabilities to the presence of given edges between two nodes.
The random graphs we consider here are a slight generalization of the model proposed  in \cite{FanChung}, 
adding a layer of hubs to their model. Our terminology is that of
\cite{FanChung,Bollobas}.
Let  ${\bf w}(N)=(w_1,\dots,w_N)$ be an ordered vector of  positive real numbers, i.e.
such that $w_1\le w_2 \le \dots \le w_N$. We construct a random graph where the expectation of the degrees is close to the one as listed in ${\bf w}(N)$ (see  
Proposition~\ref{prop:expectationdegree}). Let $\rho=1/(w_1+\dots+w_N)$.  
Given integers $0\le M<N$, we say that ${\bf w}$ is an  {\em admissible heterogeneous vector of
degrees with $M$ hubs and $L=N-M$ low degree nodes}, if
\begin{equation}
w_{N} w_{L}\rho\leq 1 . \label{eq:prob}
\end{equation}
To such a vector $w={\bf w}(N)$ we associate the probability measure $\mb P_{\bf w}$ on the set $\mc G(N)$  of all graphs on $N$ vertices, 
i.e., on the space of $N\times N$ random adjacency matrices $A$ with coefficients $0$ and $1$, taking the entries of $A$ i.i.d.  and so that 
\[
\mb P_{\bf w}(A_{in}=1) =\left\{ \begin{array}{c r} 
 w_i w_n \rho
& \mbox{~ when ~}  i\le L  \mbox{~\, or ~} n\le L\\
r&\mbox{~when~} i,n\geq L
\end{array} \right.
\]
We assigned constant probability $0\leq r\leq 1$ of having a connection among the hubs to simplify computations later, but different probability could have been assigned without changing the final outcome. Notice that the admissibility condition \eqref{eq:prob} ensures that the above probability is well defined. The pair $\mc G_{\bf w}=(\mc G(N),\mb P_{\bf w})$ is called a  {\em random graph} of size $N$. We are going to prove the following proposition.

\begin{proposition}\label{Prop:SatHetRand}
Let $\{{\bf w}(N)\}_{N\in\N}$ be a sequence of admissible vectors of
heterogeneous degrees such that ${\bf w}(N)$ has $M:=M(N)$ hubs. If there
exists $p\in[1,\infty)$ such that the entries of
the vector satisfy

\begin{align}
&\lim_{N\rightarrow\infty}{w_1}^{-1}L^{1/p}{\beta}^{1/q}=0 \label{Cond1}\\
&\lim_{N\rightarrow\infty}{w_1}^{-1/p}{M}^{1/p}=0 \label{Cond2}\\
&\lim_{N\rightarrow\infty}{w_1}^{-2}{\beta}L^{1+2/p}=0 \label{Cond3}\\
&\lim_{N\rightarrow\infty}w_1^{-1}ML^{1/p}=0\label{Cond4}
\end{align}
with $\beta(N):=\max\{w_{L}, N^{1/2}\log N\}$, then for any $\eta>0$ the probability that a graph in $\mc G_{\bf w}$  satisfies \eqref{Eq:ThmCond1}-
\eqref{Eq:ThmCond3} tends to 1, for $N\rightarrow\infty$.
\end{proposition}

To prove the theorem above we need the following result on concentration of the degrees of a random graph around their expectation.

\begin{proposition}\label{prop:expectationdegree}
Given an admissible vector of degrees ${\bf w}$ and the associated random graph $\mc G_{\bf w}$, the in-degree of the $k-$th node,
$d_k=\sum_{\ell=1}^{n}A_{k\ell}$, satisfies for every $k\in \N$ and
$C\in\R^+$
\[
\mb P\left(|d_k-\mb E[d_k]|>C\right)\leq \exp\left\{-\frac{NC^2}{2}\right\},
\]
where
\[
\mb E[d_k]=\left\{\begin{array}{cr}
w_k&1\leq k\leq L\\
w_k\left(1-\rho\sum_{\ell=L+1}^Nw_\ell\right)+Mr&k>L
\end{array}\right..
\]
\end{proposition}
\begin{proof}
Suppose $1\leq k\leq L$:
\[
\mb E[d_k]=\sum_{\ell=1}^Nw_kw_\ell\rho=w_k,
\]
From Hoeffding inequality we know that
\[
\mb P\left(\left|\frac{1}{N}\sum_{\ell=1}^{N}A_{k\ell}-\frac{w_k}{N}\right|>\frac{C}{N}\right)\leq
2\exp\{-N C^2/2\}.
\]
Suppose $k>L$.
\[
\mb E[d_k]=\sum_{\ell=1}^Lw_kw_\ell\rho+r M=w_k\left(1-\rho\sum_{\ell=L+1}^Nw_\ell\right)+r M.
\]
Again by Hoeffding
\[
\mb P\left(\left|d_k-\mb E[d_k]\right|>N\epsilon\right)\leq
2\exp\{-N\epsilon^2/2\}.
\]
\end{proof}

\begin{proof}[Proof of Proposition \ref{Prop:SatHetRand}]
For every $N\in\N$ consider the graphs in $\mc G(N)$
\[
Q_N:=\bigcap_{k=1}^N\{|d_k-\mb E[d_k]|<C_k(N)\}
\]
for given numbers $\{C_k(N)\}_{N\in\N,k\in[N]}\subset\R^+$. Since
$d_k$ are independent random variables, one obtains
\begin{align*}
\mb P(Q_N)&\geq\prod_{k=1}^N\left(1-\exp\left\{-K\frac{C_k(N)^2}{N}\right\}\right)
\end{align*}
if we choose $C_k(N)=(N\log(N))^{1/2}g(N)$ with
$g(N)\rightarrow\infty$ at any speed then
\[
\lim_{n\rightarrow\infty}\mb P(Q_N)=1.
\]
\noindent
Taken any graph $G\in Q_N$, the maximum degree satisfies
\begin{align*}
\Delta\geq w_1\left(1-\mc O(M^{-1}w_1^{-1}L)\right)-C(N)
\end{align*}
and the maximum degree for a low degree node will be $\degree<w_{L}+C(N)$. So, from conditions \eqref{Cond1}-\eqref{Cond4}, in the limit for $N\rightarrow\infty$
\begin{align*}
\frac{M^{1/p}}{\Delta^{1/p}}&\leq
\frac{M^{1/p}}{w_1^{1/p}}\frac{1}{\left[1-\frac{C(N)}{w_1}\right]^{1/p}}\rightarrow 0\\
\frac{N^{1/p}\degree^{1/q}}{\Delta}&\leq\frac{N^{1/p}\left[w_{L}+C(N)\right]^{1/q}}{w_1\left[1-\frac{C(N)}{w_1}\right]}\leq\frac{\left[\frac{L^{q/p}w_{L}}{w_1^{q}}+\frac{L^{q/p}C(N)}{w_1^{q}}\right]^{1/q}}{\left[1-\frac{C(N)}{w_1}\right]}\rightarrow 0\\
\frac{ML^{1/p}}{\Delta}&\leq\frac{ML^{1/p}}{w_1}\frac{1}{\left[1-\frac{C(N)}{w_1(n)}\right]}\rightarrow 0\\
\frac{L^{1+2/p}\degree}{\Delta^{2}}&\leq
\frac{L^{1+2/p}}{w_1^2}\frac{\left[w_{L}+C(N)\right]}{\left[1-\frac{C(N)}{w_1}\right]^{2}}=\frac{\left[\frac{L^{1+2/p}w_{L}}{w_1^2}+\frac{L^{1+2/p}C(N)}{w_1^2}\right]}{\left[1-\frac{C(N)}{w_1}\right]^{2}}\rightarrow 0
\end{align*}
which proves the proposition.
\end{proof}

\end{appendices}
\bibliographystyle{amsalpha}
\bibliography{bibliography}

\newcommand{\etalchar}[1]{$^{#1}$}
\providecommand{\bysame}{\leavevmode\hbox to3em{\hrulefill}\thinspace}
\providecommand{\MR}{\relax\ifhmode\unskip\space\fi MR }
\providecommand{\MRhref}[2]{%
  \href{http://www.ams.org/mathscinet-getitem?mr=#1}{#2}
}
\providecommand{\href}[2]{#2}
\begin{thebibliography}{MMAN13}

\bibitem[AADF11]{Field-etal-2011}
M.~Aguiar, P.~Ashwin, A.~Dias, and M.~Field, \emph{Dynamics of coupled cell
  networks: synchrony, heteroclinic cycles and inflation}, J. Nonlinear Sci.
  \textbf{21} (2011), no.~2, 271--323. \MR{2788857}

\bibitem[AB02]{BA}
R{\'e}ka Albert and Albert-L{\'a}szl{\'o} Barab{\'a}si, \emph{Statistical
  mechanics of complex networks}, Reviews of modern physics \textbf{74} (2002),
  no.~1, 47.

\bibitem[AS04]{abrams2004chimera}
Daniel~M Abrams and Steven~H Strogatz, \emph{Chimera states for coupled
  oscillators}, Physical review letters \textbf{93} (2004), no.~17, 174102.

\bibitem[BDEI{\etalchar{+}}98]{Baladi2}
Viviane Baladi, Mirko Degli~Esposti, Stefano Isola, Esa J{\"a}rvenp{\"a}{\"a},
  and Antti Kupiainen, \emph{The spectrum of weakly coupled map lattices},
  Journal de math{\'e}matiques pures et appliqu{\'e}es \textbf{77} (1998),
  no.~6, 539--584.

\bibitem[Ber93]{MR1216719}
Walter Bergweiler, \emph{Iteration of meromorphic functions}, Bull. Amer. Math.
  Soc. (N.S.) \textbf{29} (1993), no.~2, 151--188. \MR{1216719}

\bibitem[BGP{\etalchar{+}}09]{bonifazi2009gabaergic}
Paolo Bonifazi, Miri Goldin, Michel~A Picardo, Isabel Jorquera, A~Cattani,
  Gregory Bianconi, Alfonso Represa, Yehezkel Ben-Ari, and Rosa Cossart,
  \emph{Gabaergic hub neurons orchestrate synchrony in developing hippocampal
  networks}, Science \textbf{326} (2009), no.~5958, 1419--1424.

\bibitem[Bir57]{MR0087058}
Garrett Birkhoff, \emph{Extensions of {J}entzsch's theorem}, Trans. Amer. Math.
  Soc. \textbf{85} (1957), 219--227. \MR{0087058}

\bibitem[Bol01]{Bollobas}
B{\'e}la Bollob{\'a}s, \emph{Random graphs}, second ed., Cambridge Studies in
  Advanced Mathematics, vol.~73, Cambridge University Press, Cambridge, 2001.
  \MR{1864966}

\bibitem[BP02]{Barahona2002}
Mauricio Barahona and Louis~M Pecora, \emph{Synchronization in small-world
  systems}, Physical review letters \textbf{89} (2002), no.~5, 054101.

\bibitem[BR01]{Baladi1}
Viviane Baladi and Hans~Henrik Rugh, \emph{Floquet spectrum of weakly coupled
  map lattices}, Communications in Mathematical Physics \textbf{220} (2001),
  no.~3, 561--582.

\bibitem[BRS{\etalchar{+}}12]{CAS}
Murilo~S Baptista, Hai-Peng Ren, Johen~CM Swarts, Rodrigo Carareto, Henk
  Nijmeijer, and Celso Grebogi, \emph{Collective almost synchronisation in
  complex networks}, PloS one \textbf{7} (2012), no.~11, e48118.

\bibitem[BS88]{BunSinai}
L~A Bunimovich and Ya~G Sinai, \emph{Spacetime chaos in coupled map lattices},
  Nonlinearity \textbf{1} (1988), no.~4, 491.

\bibitem[Bus73]{MR0336473}
P.~J. Bushell, \emph{Hilbert's metric and positive contraction mappings in a
  {B}anach space}, Arch. Rational Mech. Anal. \textbf{52} (1973), 330--338.
  \MR{0336473}

\bibitem[CF05]{chazottes2005dynamics}
Jean-Ren{\'e} Chazottes and Bastien Fernandez, \emph{Dynamics of coupled map
  lattices and of related spatially extended systems}, vol. 671, Springer,
  2005.

\bibitem[CL06]{FanChung}
Fan Chung and Linyuan Lu, \emph{Complex graphs and networks}, CBMS Regional
  Conference Series in Mathematics, vol. 107, Published for the Conference
  Board of the Mathematical Sciences, Washington, DC; by the American
  Mathematical Society, Providence, RI, 2006. \MR{2248695}

\bibitem[CLP16]{climenhaga2016geometric}
Vaughn Climenhaga, Stefano Luzzatto, and Yakov Pesin, \emph{The geometric
  approach for constructing {S}inai--{R}uelle--{B}owen measures}, Journal of
  Statistical Physics (2016), 1--27.

\bibitem[DB14]{dorfler2014synchronization}
Florian D{\"o}rfler and Francesco Bullo, \emph{Synchronization in complex
  networks of phase oscillators: A survey}, Automatica \textbf{50} (2014),
  no.~6, 1539--1564.

\bibitem[dMvS93]{MS}
Welington de~Melo and Sebastian van Strien, \emph{One-dimensional dynamics},
  Ergebnisse der Mathematik und ihrer Grenzgebiete (3) [Results in Mathematics
  and Related Areas (3)], vol.~25, Springer-Verlag, Berlin, 1993. \MR{1239171}

\bibitem[DSL16]{MR3556527}
Jacopo De~Simoi and Carlangelo Liverani, \emph{Statistical properties of mostly
  contracting fast-slow partially hyperbolic systems}, Invent. Math.
  \textbf{206} (2016), no.~1, 147--227. \MR{3556527}

\bibitem[DZ09]{dembo2009large}
Amir Dembo and Ofer Zeitouni, \emph{Large deviations techniques and
  applications}, vol.~38, Springer Science \& Business Media, 2009.

\bibitem[ELP17]{eroglu2017}
Deniz Eroglu, Jeroen~SW Lamb, and Tiago Pereira, \emph{Synchronisation of chaos
  and its applications}, Contemporary Physics \textbf{58} (2017), no.~3,
  207--243.

\bibitem[EM14]{Mirollo2014}
Jan~R. Engelbrecht and Renato Mirollo, \emph{Classification of attractors for
  systems of identical coupled {K}uramoto oscillators}, Chaos \textbf{24}
  (2014), no.~1, 013114, 10. \MR{3402638}

\bibitem[Fer14]{Fern}
Bastien Fernandez, \emph{Breaking of ergodicity in expanding systems of
  globally coupled piecewise affine circle maps}, Journal of Statistical
  Physics \textbf{154} (2014), no.~4, 999--1029.

\bibitem[Fie73]{Fiedler73}
Miroslav Fiedler, \emph{Algebraic connectivity of graphs}, Czechoslovak Math.
  J. \textbf{23(98)} (1973), 298--305. \MR{0318007}

\bibitem[FY83]{fujisaka1983}
Hirokazu Fujisaka and Tomoji Yamada, \emph{Stability theory of synchronized
  motion in coupled-oscillator systems}, Progress of theoretical physics
  \textbf{69} (1983), no.~1, 32--47.

\bibitem[GGMA07]{Paths}
Jes{\'u}s G{\'o}mez-Gardenes, Yamir Moreno, and Alex Arenas, \emph{Paths to
  synchronization on complex networks}, Physical review letters \textbf{98}
  (2007), no.~3, 034101.

\bibitem[GM13]{MR3064670}
Georg~A. Gottwald and Ian Melbourne, \emph{Homogenization for deterministic
  maps and multiplicative noise}, Proc. R. Soc. Lond. Ser. A Math. Phys. Eng.
  Sci. \textbf{469} (2013), no.~2156, 20130201, 16. \MR{3064670}

\bibitem[GS06]{Gol-Stewart2006}
Martin Golubitsky and Ian Stewart, \emph{Nonlinear dynamics of networks: the
  groupoid formalism}, Bull. Amer. Math. Soc. (N.S.) \textbf{43} (2006), no.~3,
  305--364. \MR{2223010}

\bibitem[GSBC98]{Gol-Stewart98}
Martin Golubitsky, Ian Stewart, Pietro-Luciano Buono, and J.~J. Collins,
  \emph{A modular network for legged locomotion}, Phys. D \textbf{115} (1998),
  no.~1-2, 56--72. \MR{1616780}

\bibitem[HCP94]{heagy1994}
JF~Heagy, TL~Carroll, and LM~Pecora, \emph{Synchronous chaos in coupled
  oscillator systems}, Physical Review E \textbf{50} (1994), no.~3, 1874.

\bibitem[HP05]{hasselblatt2005partially}
Boris Hasselblatt and Yakov Pesin, \emph{Partially hyperbolic dynamical
  systems}, Handbook of dynamical systems \textbf{1} (2005), 1--55.

\bibitem[HPS06]{hirsch2006invariant}
Morris~W Hirsch, Charles~Chapman Pugh, and Michael Shub, \emph{Invariant
  manifolds}, vol. 583, Springer, 2006.

\bibitem[Izh07]{Izhikevich2007dynamical}
Eugene~M Izhikevich, \emph{Dynamical systems in neuroscience}, MIT press, 2007.

\bibitem[Kan92]{Kaneko1992}
Kunihiko Kaneko, \emph{Overview of coupled map lattices}, Chaos \textbf{2}
  (1992), no.~3, 279--282. \MR{1184469}

\bibitem[KH95]{MR1326374}
Anatole Katok and Boris Hasselblatt, \emph{Introduction to the modern theory of
  dynamical systems}, Encyclopedia of Mathematics and its Applications,
  vol.~54, Cambridge University Press, Cambridge, 1995, With a supplementary
  chapter by Katok and Leonardo Mendoza. \MR{1326374}

\bibitem[KL99]{keller1999stability}
Gerhard Keller and Carlangelo Liverani, \emph{Stability of the spectrum for
  transfer operators}, Annali della Scuola Normale Superiore di Pisa-Classe di
  Scienze \textbf{28} (1999), no.~1, 141--152.

\bibitem[KL04]{Keller3}
\bysame, \emph{Coupled map lattices without cluster expansion}, Discrete and
  Continuous Dynamical Systems \textbf{11} (2004), 325--336.

\bibitem[KL05]{Keller1}
\bysame, \emph{A spectral gap for a one-dimensional lattice of coupled
  piecewise expanding interval maps}, Dynamics of coupled map lattices and of
  related spatially extended systems, Springer, 2005, pp.~115--151.

\bibitem[KL06]{Keller2}
\bysame, \emph{Uniqueness of the {SRB} measure for piecewise expanding weakly
  coupled map lattices in any dimension}, Communications in Mathematical
  Physics \textbf{262} (2006), no.~1, 33--50.

\bibitem[KSvS07]{MR2342693}
O.~Kozlovski, W.~Shen, and S.~van Strien, \emph{Density of hyperbolicity in
  dimension one}, Ann. of Math. (2) \textbf{166} (2007), no.~1, 145--182.
  \MR{2342693}

\bibitem[Kur84]{Kuramoto84}
Y.~Kuramoto, \emph{Chemical oscillations, waves, and turbulence}, Springer
  Series in Synergetics, vol.~19, Springer-Verlag, Berlin, 1984. \MR{762432}

\bibitem[KY10]{Koiller-Young}
Jos{\'e} Koiller and Lai-Sang Young, \emph{Coupled map networks}, Nonlinearity
  \textbf{23} (2010), no.~5, 1121.

\bibitem[MMAN13]{Motter}
Adilson~E Motter, Seth~A Myers, Marian Anghel, and Takashi Nishikawa,
  \emph{Spontaneous synchrony in power-grid networks}, Nature Physics
  \textbf{9} (2013), no.~3, 191--197.

\bibitem[Mn87]{MR889254}
Ricardo Ma\~n{\'e}, \emph{Ergodic theory and differentiable dynamics},
  Ergebnisse der Mathematik und ihrer Grenzgebiete (3) [Results in Mathematics
  and Related Areas (3)], vol.~8, Springer-Verlag, Berlin, 1987, Translated
  from the Portuguese by Silvio Levy. \MR{889254}

\bibitem[Moh91]{Mohar91}
Bojan Mohar, \emph{Eigenvalues, diameter, and mean distance in graphs}, Graphs
  Combin. \textbf{7} (1991), no.~1, 53--64. \MR{1105467}

\bibitem[Moh92]{Mohar92}
\bysame, \emph{Laplace eigenvalues of graphs---a survey}, Discrete Math.
  \textbf{109} (1992), no.~1-3, 171--183, Algebraic graph theory (Leibnitz,
  1989). \MR{1192380}

\bibitem[NMLH03]{Nishikawa}
Takashi Nishikawa, Adilson~E Motter, Ying-Cheng Lai, and Frank~C Hoppensteadt,
  \emph{Heterogeneity in oscillator networks: Are smaller worlds easier to
  synchronize?}, Physical review letters \textbf{91} (2003), no.~1, 014101.

\bibitem[NRS16]{nijholt2016graph}
Eddie Nijholt, Bob Rink, and Jan Sanders, \emph{Graph fibrations and symmetries
  of network dynamics}, Journal of Differential Equations \textbf{261} (2016),
  no.~9, 4861--4896.

\bibitem[Per10]{Hubs}
Tiago Pereira, \emph{Hub synchronization in scale-free networks}, Physical
  Review E \textbf{82} (2010), no.~3, 036201.

\bibitem[PERV14]{Pereira2014}
Tiago Pereira, Jaap Eldering, Martin Rasmussen, and Alexei Veneziani,
  \emph{Towards a theory for diffusive coupling functions allowing persistent
  synchronization}, Nonlinearity \textbf{27} (2014), no.~3, 501--525.
  \MR{3168262}

\bibitem[Pes97]{MR1489237}
Yakov~B. Pesin, \emph{Dimension theory in dynamical systems}, Chicago Lectures
  in Mathematics, University of Chicago Press, Chicago, IL, 1997, Contemporary
  views and applications. \MR{1489237}

\bibitem[PG14]{porter2014dynamical}
Mason~A Porter and James~P Gleeson, \emph{Dynamical systems on networks: a
  tutorial}, arXiv preprint arXiv:1403.7663 (2014).

\bibitem[PSV01]{Epidemic}
Romualdo Pastor-Satorras and Alessandro Vespignani, \emph{Epidemic spreading in
  scale-free networks}, Physical review letters \textbf{86} (2001), no.~14,
  3200.

\bibitem[QXZ09]{qian2009smooth}
Min Qian, Jian-Sheng Xie, and Shu Zhu, \emph{Smooth ergodic theory for
  endomorphisms}, Springer, 2009.

\bibitem[RGvS15]{MR3336841}
Lasse Rempe-Gillen and Sebastian van Strien, \emph{Density of hyperbolicity for
  classes of real transcendental entire functions and circle maps}, Duke Math.
  J. \textbf{164} (2015), no.~6, 1079--1137. \MR{3336841}

\bibitem[RS15]{rink2015coupled}
Bob Rink and Jan Sanders, \emph{Coupled cell networks: semigroups, lie algebras
  and normal forms}, Transactions of the American Mathematical Society
  \textbf{367} (2015), no.~5, 3509--3548.

\bibitem[SB16]{Balint}
Fanni S{\'e}lley and P{\'e}ter B{\'a}lint, \emph{Mean-field coupling of
  identical expanding circle maps}, Journal of Statistical Physics \textbf{164}
  (2016), no.~4, 858--889.

\bibitem[S{\'e}l16]{selley2016symmetry}
Fanni~M S{\'e}lley, \emph{Symmetry breaking in a globally coupled map of four
  sites}, arXiv preprint arXiv:1612.01310 (2016).

\bibitem[Shu13]{shub2013global}
Michael Shub, \emph{Global stability of dynamical systems}, Springer Science
  \&amp; Business Media, 2013.

\bibitem[SSTC01]{shil2001}
L.~P. Shilnikov, A.L. Shilnikov, D.V. Turaev, and L.O. Chua, \emph{Methods of
  qualitative theory in nonlinear dynamics}, vol.~5, World Scientific, 2001.

\bibitem[Str00]{strogatz2000kuramoto}
Steven~H Strogatz, \emph{From {K}uramoto to {C}rawford: exploring the onset of
  synchronization in populations of coupled oscillators}, Physica D: Nonlinear
  Phenomena \textbf{143} (2000), no.~1, 1--20.

\bibitem[SVM07]{MR2316999}
J.~A. Sanders, F.~Verhulst, and J.~Murdock, \emph{Averaging methods in
  nonlinear dynamical systems}, second ed., Applied Mathematical Sciences,
  vol.~59, Springer, New York, 2007. \MR{2316999}

\bibitem[Tsu01]{MR1862809}
Masato Tsujii, \emph{Fat solenoidal attractors}, Nonlinearity \textbf{14}
  (2001), no.~5, 1011--1027. \MR{1862809}

\bibitem[Via97]{ViaSdds}
Marcelo Viana, \emph{Stochastic dynamics of deterministic systems}, Lecture
  Notes XXI Braz. Math. Colloq. IMPA Rio de Janeiro, 1997.

\bibitem[WAH88]{weiss1988}
CO~Weiss, NB~Abraham, and U~H{\"u}bner, \emph{Homoclinic and heteroclinic chaos
  in a single-mode laser}, Physical Review Letters \textbf{61} (1988), no.~14,
  1587.

\bibitem[Wil12]{MR3220769}
Amie Wilkinson, \emph{Smooth ergodic theory}, Mathematics of complexity and
  dynamical systems. {V}ols. 1--3, Springer, New York, 2012, pp.~1533--1547.
  \MR{3220769}

\end{thebibliography}

\
\end{document}